\documentclass[12pt,reqno]{amsart}

\usepackage{amsmath,amsfonts,amsthm,amssymb,color}
\usepackage[T1]{fontenc}
\usepackage[latin1]{inputenc}
\usepackage[extra]{tipa}
\usepackage{xy} %options graphiques
\usepackage{graphicx}
\usepackage{pstricks}
\usepackage{pstricks-add}
\usepackage{pst-node}

\newcommand{\der}{\delta}
\newcommand{\dn}{\cs_{n}}

\newcommand{\hddx}{\dti}

\newcommand{\hdx}{\delti}

\newcommand{\hxu}{\tilde{x}^1}
\newcommand{\hxd}{\tilde{x}^2}
\newcommand{\hz}{\hat z}
\newcommand{\hph}{\hat\phi}
\newcommand{\ka}{\kappa}

\newcommand{\iot}{\int_{0}^{t}}

\newcommand{\ist}{\int_{s}^{t}}
\newcommand{\lot}{[\ell_1,\ell_2]}
\newcommand{\norm}[1]{\lVert #1\rVert}

\newcommand{\ott}{[0,T]}
\newcommand{\ou}{[0,1]}

\newcommand{\xti}{\tilde{x}}
\newcommand{\Xti}{\tilde{X}}
\newcommand{\Yti}{\tilde{Y}}

\newcommand{\yti}{\tilde{y}}
\newcommand{\zti}{\tilde{z}}
\newcommand{\Iti}{\tilde{I}}

\newcommand{\yun}{y^{(1)}}
\newcommand{\zun}{z^{(1)}}
\newcommand{\zde}{z^{(2)}}

\newcommand{\Dti}{\tilde{D}}

\newcommand{\hphi}{\hat{\phi}}
\newcommand{\delti}{\tilde{\delta}}
\newcommand{\Lati}{\tilde{\Lambda}}
\newcommand{\cacti}{\tilde{\cac}}
\newcommand{\Rti}{\tilde{R}}
\newcommand{\Ati}{\tilde{A}}
\newcommand{\gti}{\tilde{g}}
\newcommand{\fti}{\tilde{f}}

\newcommand{\hti}{\tilde{h}}
\newcommand{\Mti}{\tilde{M}}
\newcommand{\dti}{\tilde{d}}

\newcommand{\rti}{\tilde{r}}
\newcommand{\ytiun}{\yti^{(1)}}
\newcommand{\ytide}{\yti^{(2)}}
\newcommand{\ztiun}{\zti^{(1)}}
\newcommand{\ztide}{\zti^{(2)}}

\newcommand{\yde}{y^{(2)}}
\newcommand{\deldeti}{\doubletilde{$\delta$}}
\newcommand{\cacdeti}{\doubletilde{$\mathcal{C}$}}
\newcommand{\Rdeti}{\doubletilde{$R$}}
\newcommand{\Udeti}{\doubletilde{$U$}}
\newcommand{\Adeti}{\doubletilde{$A$}}
\newcommand{\Bdeti}{\doubletilde{$B$}}
\newcommand{\Uti}{\tilde{U}}
\newcommand{\Vti}{\tilde{V}}
\newcommand{\Cti}{\tilde{C}}
\newcommand{\Qti}{\tilde{\cq}}
\newcommand{\gati}{\tilde{\gamma}}
\newcommand{\Mtil}{\tilde{M}}

\newcommand{\Fti}{\tilde{F}}
\newcommand{\Bti}{\tilde{B}}
\newcommand{\Xdeti}{\doubletilde{$X$}}
\newcommand{\Cdeti}{\doubletilde{$C$}}
\newcommand{\Xep}{X^\ep}
\newcommand{\Xtiunep}{\Xti^{1,\ep}}
\newcommand{\Xtideep}{\Xti^{2,\ep}}
\newcommand{\Xtitrep}{\Xti^{3,\ep}}
\newcommand{\Xunep}{X^{1,\ep}}
\newcommand{\Xet}{X^\eta}
\newcommand{\Xtiunet}{\Xti^{1,\eta}}
\newcommand{\Xtideet}{\Xti^{2,\eta}}
\newcommand{\Xtitret}{\Xti^{3,\eta}}
\newcommand{\Xde}{X^\Delta}
\newcommand{\Xtiunde}{\Xti^{1,\Delta}}
\newcommand{\Xtidede}{\Xti^{2,\Delta}}
\newcommand{\Xtitrde}{\Xti^{3,\Delta}}
\newcommand{\Xdetiquep}{\Xdeti^{4,\ep}}
\newcommand{\Xdetiquet}{\Xdeti^{4,\eta}}
\newcommand{\Xdetiqude}{\Xdeti^{4,\Delta}}
\newcommand{\Xepun}{X^{\ep,(1)}}
\newcommand{\Xetun}{X^{\eta,(1)}}
\newcommand{\Xepde}{X^{\ep,(2)}}
\newcommand{\Xetde}{X^{\eta,(2)}}

\DeclareMathOperator{\id}{\text{Id}}

\newcommand{\ca}{{\mathcal A}}
\newcommand{\cb}{{\mathcal B}}
\newcommand{\cac}{{\mathcal C}}

\newcommand{\ce}{{\mathcal E}}
\newcommand{\cf}{{\mathcal F}}

\newcommand{\ch}{{\mathcal H}}
\newcommand{\cj}{{\mathcal J}}
\newcommand{\cl}{{\mathcal L}}
\newcommand{\cm}{{\mathcal M}}
\newcommand{\cn}{{\mathcal N}}

\newcommand{\cq}{{\mathcal Q}}

\newcommand{\cs}{{\mathcal S}}

\newcommand{\cZ}{{\mathcal Z}}
\newcommand{\cz}{{\mathcal Z}}

\newcommand{\al}{\alpha}
\newcommand{\be}{\beta}
\newcommand{\ga}{\gamma}
\newcommand{\ep}{\varepsilon}
\newcommand{\la}{\lambda}
\newcommand{\si}{\sigma}

\newcommand{\laa}{\Lambda}

\newcommand{\N}{{\mathbb N}}

\newcommand{\R}{{\mathbb R}}

\newcommand{\lla}{\left\langle}
\newcommand{\rra}{\right\rangle}
\newcommand{\lcl}{\left\{}
\newcommand{\rcl}{\right\}}
\newcommand{\lp}{\left(}
\newcommand{\rp}{\right)}
\newcommand{\lc}{\left[}
\newcommand{\rc}{\right]}
\newcommand{\lln}{\left|}
\newcommand{\rrn}{\right|}

\newcommand{\bean}{\begin{eqnarray*}}
\newcommand{\eean}{\end{eqnarray*}}
\newcommand{\ben}{\begin{enumerate}}
\newcommand{\een}{\end{enumerate}}
\newcommand{\beq}{\begin{equation}}
\newcommand{\eeq}{\end{equation}}

\newtheorem{theorem}{Theorem}[section]

\newtheorem{corollary}[theorem]{Corollary}

\newtheorem{definition}[theorem]{Definition}

\newtheorem{hypothesis}{Hypothesis}
\newtheorem{lemma}[theorem]{Lemma}

\newtheorem{proposition}[theorem]{Proposition}

\theoremstyle{remark}
\newtheorem{remark}[theorem]{Remark}

\begin{document}

\title[Rough Volterra equations]{Rough Volterra equations 2: convolutional generalized integrals}
\author{Aurélien Deya \and Samy Tindel}
\address{
{\it Aurélien Deya and Samy Tindel:}
{\rm Institut {\'E}lie Cartan Nancy, Nancy-Universit\'e, B.P. 239,
54506 Vand{\oe}uvre-l{\`e}s-Nancy Cedex, France}.
{\it Email: }{\tt deya@iecn.u-nancy.fr}, {\tt tindel@iecn.u-nancy.fr}
}

\keywords{Rough paths theory; Stochastic Volterra equations; Fractional Brownian motion.}

\subjclass[2000]{60H05, 60H07, 60G15}

\date{\today}

\begin{abstract}
We define and solve Volterra equations driven by an irregular signal, by means of a variant of the rough path theory allowing to handle generalized integrals weighted by an exponential coefficient. The results are applied to the fractional Brownian motion with Hurst coefficient $H>1/3$.
\end{abstract}

\maketitle

\tableofcontents

%%%%%%%%%%%%%%%%%%%%%%%%%%%%%%%%%%%%%%%%%%%
%%%%%%%%%%%%%%%%%%%%%%%%%%%%%%%%%%%%%%%%%%%
\section{Introduction}

Let $x$ be a general $n$-dimensional Hölder continuous path with Hölder exponent $\ga>0$, an initial condition $a\in\R^d$, and $\si:\R_+\times\R_+\times\R^d\to\R^{d,n}$ a smooth enough function. Then a general form of stochastic Volterra equation driven by $x$ (considered as a noisy input) can be written as:
\begin{equation}\label{eq:volterra}
y_t=a+\int_0^t \si(t,u,y_u)\, dx_u,
\quad\mbox{ for }\quad
s\in[0,T],
\end{equation}
where $T$ an arbitrary positive constant. This kind of system being widely used in the physical and biological literature, its noisy version has also been intensively studied when the driving motion $x$ is a Brownian motion  \cite{BM1,BM2,Le} or a general semi-martingale \cite{Pr}. If the coefficient $\si$ is also considered as a random function, which is natural in  many situations, some anticipative stochastic calculus techniques are required in order to solve equation (\ref{eq:volterra}), and we refer to \cite{AN,CLP,CD,OZ,NR,PP} for the main results in this direction. It should be mentioned at this point that the last of those references \cite{PP} is motivated by financial models of capital growth rate, which goes beyond the classical physical or biological applications of Volterra equations.

\smallskip

It seems then quite natural to  generalize the aforementioned results, and consider systems like (\ref{eq:volterra}) driven by general continuous processes, whose prototype can be thought of as a $n$-dimensional fractional Brownian motion. In this case, and when one desires to go beyond the Young case $\ga>1/2$, rough paths type techniques must come into the picture. However, the classical rough path theory introduced by Terry Lyons \cite{LQ-bk} (see also the nice introductions \cite{FV,Le}) is mostly designed to handle the case of diffusion type equations, and there have been an intensive activity during the last couple of years in order to extend these semi-pathwise techniques to other systems, such as delay equations \cite{NNT} or PDEs \cite{CF,GT}. The current article fits then into this global project, and we shall see how to perturb the original rough path setting in order to handle systems like (\ref{eq:volterra}).

\smallskip

Before we come to a description of our main results, let us mention a few choices we have made for this paper:

\smallskip

\noindent
{\it (i)} Like in \cite{NNT,GT}, we have chosen to work with a variant of the rough path theory introduced by Gubinelli in \cite{Gu}, and called algebraic integration. This method is based on some simple enough algebraic considerations, and this relative simplicity makes it amenable to intuitions on possible generalizations of the original setting, beyond the diffusion case. In the case of Volterra equations handled here, we will see that, in spite of the huge amount of technical details involved in our proofs, the main ideas on which our constructions rely are quite natural.

\smallskip

\noindent
{\it (ii)} We have specialized equation (\ref{eq:volterra}) in the following manner: instead of considering a general coefficient of the form  $\si(t,u,y_u)$, we have assumed that the coefficient $\si$ can be decomposed under the form $\phi(t-u) \, \si(x)$, for a given kernel $\phi:\R_+\to\R$ and a matrix-valued function $\si$ defined on $\R^d$. Furthermore, an additional hypothesis is made on the kernel $\phi$: we assume that it can be written as the Laplace transform of a certain function $\hph$, namely that
\beq\label{eq:dcp-phi-laplace-intro}
\phi(v)=\int_0^\infty e^{-v \xi} \hat\phi(\xi) \, d\xi, 
\quad\mbox{and}\quad
\int_0^\infty (1+\xi)^\beta | \hat\phi(\xi) | \, d\xi <\infty,
\eeq
for a certain $\beta>0$. This additional assumption is made in order to take advantage of the multiplicative property of the exponential function, and it should be noticed here that the same kind of results could have been obtained by means of Fourier (instead of Laplace) transforms. The integrability hypothesis on $\hat\phi(\xi)$ is morally equivalent to a regularity condition on our kernel $\phi$. Once these assumptions are made, and up to a an application of Fubini's theorem which can be justified easily in case of a smooth driving process $x$, our Volterra system can be written as:
\beq\label{eq:sde-laplace-1}
y_t=a+\int_0^\infty d\xi\, \hph(\xi) \int_0^t e^{- \xi(t-u)}\,  \si(y_u) \, dx_u .
\eeq

\smallskip

\noindent
{\it (iii)} An additional  cosmetic change is the following: in order to ease some of our future expansions, we transpose the matrix notations given before and set $y_t\equiv y_t^*$. With this little change in the usual notations, one is left with the following system:
\beq\label{eq:sde-laplace}
y_t=a+\int_0^\infty d\xi\, \hph(\xi) \int_0^t e^{- \xi(t-u)} \, dx_u \,  \si(y_u).
\eeq
This is the general form under which we shall solve our Volterra problem.

\smallskip

With these preliminaries in hand, the main results contained in this paper can be roughly summarized as follows (see Theorem \ref{main-theorem-rough} below for a precise statement):
\begin{theorem}\label{thm:1.1}
Let $x$ be a $n$-dimensional fractional Brownian motion with Hurst parameter $H>1/3$. Assume that $\phi$ can be decomposed as (\ref{eq:dcp-phi-laplace-intro}), with $\beta=2$, and that $\si$ is a $\cac^{3,\textbf{\textit{b}}}$-function. Then equation (\ref{eq:sde-laplace}) admits a unique solution on any arbitrary interval $\ott$, in a class of paths called convolutional controlled processes, and where the integral with respect to $x$ has to be interpreted as in Proposition \ref{prop:intg-controlled-proc}.
\end{theorem}

\smallskip

Let us now say a few words about the methodology we have adopted in order to solve our equation: as mentioned before, it consists in an elaboration of the tools introduced in \cite{Gu}. Let us recall that these latter reference relies on the definition of an elementary operator $\delta$, which transforms for instance a function $f$ of one variable $t\in\ott$ into a function of two variables as $(\der f)_{ts}=f_t-f_s$. Under some algebraic and analytic conditions, this operator $\der$ can be inverted, its inverse is called $\Lambda$, and this inverse allows to construct a generalized integral of Young type. If one wants to solve an equation of the form $dy_t= dx_t \, \si(y_t)$, a possible strategy is then the following: remark first that the a priori increments of $y$ can be decomposed as:
\beq\label{eq:weak-ctrl-intro}
(\der y)_{ts}= (\der x)_{ts} \, \zeta_s + r_{ts},
\quad\mbox{with}\quad
\zeta_s=\si(y_s), \,\mbox{and } r_{ts}= \int_s^t dx_u \, \lc \si(y_u)-\si(y_s) \rc.
\eeq
Furthermore, if $x$ is a $\ga$-Hölder process, one also expects $y$ to be $\ga$-Hölder continuous. Thus, if $\si$ is regular enough, $\zeta$ will inherit the same regularity, and it is also easily conceived that $r$ should have the double regularity, namely $2\ga$-Hölder. This is precisely the structure asked in \cite{Gu} for the solution to the diffusion-type equation $dy_t= dx_t \, \si(y_t)$, and a process admitting the decomposition (\ref{eq:weak-ctrl-intro}) is called a weakly controlled processes. The second important point in the theory is to notice that a reasonable definition of the integral $\int_s^t dx_u \, [ \si(y_u)-\si(y_s) ]$ can be given for a controlled process, provided that the so-called Levy area associated to $x$ (formally defined as $\int\int dx dx$) can be constructed, and thanks to the operator $\Lambda$ mentioned above. This integration step transforms a weakly controlled process into a weakly controlled process, and allows to settle a fixed point argument for the resolution of the diffusion equation driven by $x$.

\smallskip

Let us try to explain now what has to be changed to the original algebraic integration setting in order to handle the case of a Volterra equations: 

\smallskip

\noindent
{\it (i)}
Observe first that in order to solve equation (\ref{eq:sde-laplace}), the main step is to define accurately the rough integral $\yti_t(\xi)\equiv\int_0^t e^{- \xi(t-u)} \, dx_u \,  \si(y_u)$ for all $t,\xi\in\R_+$. As will be explained at Section \ref{sec:heuristic-young}, an important step in this direction is to note that, in order to get some increments of $\yti$ involving only integrals of the form $\ist$, one has to introduce some twisted increments of the form $\delti \yti_{ts}(\xi)\equiv \der \yti_{ts}(\xi)-(e^{-\xi(t-s)}-1)\yti_s(\xi)$. Then it is easily checked, in case of smooth paths $y$ and $x$, that the a priori twisted increments of the solution $\yti$ to equation (\ref{eq:sde-laplace}) can be expressed as 
$$
\delti \yti_{ts}(\xi) = \ist e^{-\xi(t-v)} \, dx_v \, \si(y_v).
$$
The operator $\delti$ will thus play a central role in our computations, a fact which is reminiscent from the calculations contained in \cite{GT} for the definition of rough PDEs. It turns out that the operator $\delti$ can also be inverted under some algebraic and analytic conditions. This inverse gives then raise to a generalized convolutional Young integral, which is at the core of our definition of the integral $\int_0^t e^{- \xi(t-u)} \, dx_u \,  \si(y_u)$.

\smallskip

\noindent
{\it (ii)}
The notion of controlled paths has also to be changed for the resolution of equation (\ref{eq:sde-laplace}), and we shall introduce a notion of convolutional controlled path, which will be detailed at Section \ref{sec:conv-ctrl-path}. They are basically defined as in equation (\ref{eq:weak-ctrl-intro}), except that $\der x$ is replaced by an increment of the form $x^1$, with $x^1_{ts}=\ist \phi(t-v)\, dx_v$, which is assumed to exist once and for all. Then as in the diffusion case, we are able to define a natural extension of the notion of integral for those convolutional controlled processes, provided that some double iterated integrals based on $x$ can be defined. More specifically, the equivalent of the notion of Levy area in our Volterra context is an increment indexed by the Laplace variable, of the form:
\beq\label{eq:def-tx2-intro}
\tilde x^2_{ts}(\xi)=\int_0^\infty d\eta\, \hph(\eta) 
\ist e^{- \xi(t-v_1)} \, dx_{v_1} \int_s^{v_1} e^{- \eta(v_1-v_2)} \, dx_{v_2}.
\eeq
Here again, this definition is only formal in case of a Hölder path $x$, but once it is assumed to exist and to satisfy suitable analytic and algebraic hypotheses, a good notion of integral can be constructed for convolutional controlled processes. This allows again a fixed point procedure in order to solve our Volterra equation. A quick glimpse at the proof of Theorem \ref{main-theorem-rough} will show however that this fixed point procedure is trickier than in the diffusion case.

\smallskip

\noindent
{\it (iii)} An essential step in our approach is thus a good definition of the double integral (\ref{eq:def-tx2-intro}), and the study of its regularity in $(s,t)$. This can be done quite easily (up to some Garsia type regularity theorems which have to be proven) when $x$ is taken as a Brownian motion, and when the integrals with respect to $x$ are interpreted in the Itô sense. However, an important part of the current article will be devoted to the definition of (\ref{eq:def-tx2-intro}) when $x$ is a fractional Brownian motion with $1/3<H<1/2$. Here again, the path we have followed in order to obtain this definition is not completely standard, and let us say a few words about it. Indeed, the usual way to define a double integral like (\ref{eq:def-tx2-intro}) in case of a fBm is to use Stratonovich integrals, in the sense of the Malliavin calculus as explained in \cite{Nu-cours}. However, this way to compute our iterated integrals involves a decomposition of the Stratonovich integral into a Skorokhod type term plus a trace term, which is hard (though not impossible) to analyze in case of an exponentially weighted integral like ours. We have thus decided to adopt another strategy, and have resorted to an analytical approximation of the fractional Brownian motion introduced in \cite{Un}. This latter approximation, which will be recalled at Section \ref{sec:applic-rough-fbm}, has the advantage to yield almost explicit and elementary computations, based on the analysis of singularities for some locally analytic functions defined on the complex plane. In our opinion, the calculations we obtain are thus more elegant than in the Malliavin calculus setting.

\smallskip

As we mentioned before, the resolution of rough Volterra equations relies thus on a few simple ideas. These ideas are however long to formalize when one wishes to give most of the details of the calculations, which explains the bulk of the  current article. It should also be mentioned at this point that we could have tried to solve equation (\ref{eq:volterra}), in its general form, without recurring to twisted convolutional increments as we did, by just following the standard algebraic integration formalism. This is in fact what we have done in the companion paper \cite{DT}, and this idea works fine for the Young case, namely for a Hölder coefficient $\ga>1/2$. However,  this first method fails to give a global existence result in the case $1/3<\ga<1/2$, mainly because the Picard iterations don't lead to a contraction property (we refer the reader to \cite{DT} for a more detailed argument). This important drawback justifies the introduction of the convolutional generalized integration we have used in the current paper.

Here is how our article is organized: we recall some basic definitions of algebraic integration at Section \ref{sec:alg-integration}. Section \ref{sec:Young-volterra} is devoted to the simpler case of Young equations, which allows to explain our method with less technical apparatus. Then at Section \ref{sec:rough-volterra} we move to the rough case of our Volterra equation, and explain all the details of the method we have chosen in order to solve it. Finally, we apply our theory to the fractional Brownian motion case at Section \ref{sec:applic-rough-fbm}.

%%%%%%%%%%%%%%%%%%%%%%%%%%%%%%%%%%%%%%%%%%%
%%%%%%%%%%%%%%%%%%%%%%%%%%%%%%%%%%%%%%%%%%%
\section{Algebraic integration}
\label{sec:alg-integration}

This section is devoted to recall the very basic elements of the algebraic integration theory introduced in \cite{Gu}, in order to fix notations for the remainder of the paper. We also include a proof of the existence of the so-called sewing map $\Lambda$ which is simpler than the one contained in the original paper \cite{Gu}, and is even a further simplification of the proof proposed in \cite{GT}.

\subsection{Increments}\label{sec:incr}

As mentioned in the introduction, the extended integral we deal
with is based on the notion of increment, together with an
elementary operator $\der$ acting on them. The notion of increment can be introduced in the following way:  for two arbitrary real numbers
$\ell_2>\ell_1\ge 0$, a vector space $V$, and an integer $k\ge 1$, we denote by
$\cac_k(V)$ the set of continuous functions $g : [\ell_1,\ell_2]^{k} \to V$ such
that $g_{t_1 \cdots t_{k}} = 0$
whenever $t_i = t_{i+1}$ for some $i\le k-1$.
Such a function will be called a
\emph{$(k-1)$-increment}, and we will
set $\cac_*(V)=\cup_{k\ge 1}\cac_k(V)$. The operator $\der$
alluded to above can be seen as an operator acting on
$k$-increments, 
and is defined as follows on $\cac_k(V)$:
\begin{equation}
  \label{eq:coboundary}
\delta : \cac_k(V) \to \cac_{k+1}(V) \qquad
(\delta g)_{t_1 \cdots t_{k+1}} = \sum_{i=1}^{k+1} (-1)^i g_{t_1
  \cdots \hat t_i \cdots t_{k+1}} ,
\end{equation}
where $\hat t_i$ means that this particular argument is omitted.
Then a fundamental property of $\der$, which is easily verified,
is that
$\delta \delta = 0$, where $\delta \delta$ is considered as an operator
from $\cac_k(V)$ to $\cac_{k+2}(V)$.
 We will denote $\cZ\cac_k(V) = \cac_k(V) \cap \text{Ker}\delta$
and $\cb \cac_k(V) =\cac_k(V) \cap \text{Im}\delta$.

\smallskip

Some simple examples of actions of $\der$,
which will be the ones we will really use throughout the paper,
 are obtained by letting
$g\in\cac_1$ and $h\in\cac_2$. Then, for any $t,u,s\in\lot$, we have
\begin{equation}
\label{eq:simple_application}
  (\der g)_{ts} = g_t - g_s,
\quad\mbox{ and }\quad
(\der h)_{tus} = h_{ts}-h_{tu}-h_{us}.
\end{equation}
Furthermore, it is readily checked that
the complex $(\cac_*,\delta)$ is \emph{acyclic}, i.e.
$\cZ \cac_{k+1}(V) = \cb \cac_{k}(V)$ for any $k\ge 1$. In particular, the following basic property, which we
label  for further use, holds true:
\begin{lemma}\label{exd}
Let $k\ge 1$ and $h\in \cz\cac_{k+1}(V)$. Then there exists a (non unique)
$f\in\cac_{k}(V)$ such that $h=\der f$.
\end{lemma}

\noindent
Observe that Lemma \ref{exd} implies that all the elements
$h \in\cac_2(V)$ such that $\der h= 0$ can be written as $h = \der f$
for some (non unique) $f \in \cac_1(V)$. Thus we get a heuristic
interpretation of $\der |_{\cac_2(V)}$:  it measures how much a
given 1-increment  is far from being an {\it exact} increment of a
function (i.e. a finite difference).

\smallskip

Notice that our future discussions will mainly rely on
$k$-increments with $k \le 2$, for which we will use some
analytical assumptions. Namely,
we measure the size of these increments by H\"older norms
defined in the following way: for $f \in \cac_2(V)$ let
$$
\norm{f}_{\mu} \equiv
\sup_{s,t\in\lot}\frac{\lVert f_{ts}\rVert_V}{|t-s|^\mu},
\quad\mbox{and}\quad
\cac_1^\mu(V)=\lcl f \in \cac_2(V);\, \norm{f}_{\mu}<\infty  \rcl.
$$
In the same way, for $h \in \cac_3(V)$, set
\begin{eqnarray}
  \label{eq:normOCC2}
  \norm{h}_{\gamma,\rho} &=& \sup_{s,u,t\in\lot}
\frac{\lVert h_{tus}\rVert_V}{|u-s|^\gamma |t-u|^\rho}\\
%\quad\mbox{and}\quad
\norm{h}_{\mu} &\equiv &
\inf\left \{\sum_i \norm{h_i}_{\rho_i,\mu-\rho_i} ;\, h  =\sum_i h_i,\, 0 < \rho_i < \mu \right\} ,\nonumber
\end{eqnarray}
where the last infimum is taken over all sequences $\{h_i \in \cac_3(V) \}$ such that $h
= \sum_i h_i$ and for all choices of the numbers $\rho_i \in (0,z)$.
Then  $\norm{\cdot}_\mu$ is easily seen to be a norm on $\cac_3(V)$, and we set
$$
\cac_3^\mu(V):=\lcl h\in\cac_3(V);\, \norm{h}_{\mu}<\infty \rcl.
$$
Eventually,
let $\cac_3^{1+}(V) = \cup_{\mu > 1} \cac_3^\mu(V)$,
and remark that the same kind of norms can be considered on the
spaces $\cZ \cac_3(V)$, leading to the definition of some spaces
$\cZ \cac_3^\mu(V)$ and $\cZ \cac_3^{1+}(V)$. In order to avoid ambiguities, we shall denote by $\cn[f;\, \cac_j^\kappa]$ the $\kappa$-Hölder norm on the space $\cac_j$, for $j=1,2,3$. For $\zeta\in\cac_j(V)$, we also set $\mathcal{N}[\zeta;\mathcal{C}_{j}^{0}(V)]=\sup_{s\in[\ell_1; \ell_2]^j}\lVert \zeta_s\rVert_{V}$.

\vspace{0.2cm}

With these notations in mind,
the following proposition is a basic result which is at the core of our approach to path-wise integration:

\begin{theorem}[The sewing map] \label{prop:Lambda}
Let $\mu >1$. For any $h\in \cz \cac_3^\mu([0,1]; V)$, there exists a unique $\Lambda h \in \cac_2^\mu([0,1];V)$ such that $\der( \Lambda h )=h$. Furthermore,
\begin{eqnarray} \label{contraction}
\norm{ \Lambda h}_\mu \leq c_\mu  \, \cn[h;\, \cac_3^{\mu}(V)],
\end{eqnarray}
with $c_\mu =2+2^\mu \sum_{k=1}^\infty k^{-\mu}$. This gives rise to a linear continuous map $\laa:  \cz \cac_3^\mu([0,1]; V) \rightarrow \cac_2^\mu([0,1];V)$ such that $\der \laa =\id_{ \cz \cac_3^\mu([0,1]; V)}$.
\end{theorem}

\begin{proof}
The original proof of the existence of $\laa$ (with a somewhat different constant $C_\mu$) can be found in \cite{Gu} and has been simplified in \cite{GT}. We give here a more elementary proof, which will be easy to adapt to the pertubated incremental operator $\delti$ (see Section \ref{sect:convol-increments}).

\smallskip

\noindent
{\it Uniqueness.} 
Let $M,M' \in \cac_2^\mu$ such that $\der M=\der M' =h$. In particular, $\der(M-M')=0$, and thus, according to \ref{exd}, $M-M'=\der q$, with $q\in \cac_1$. But then $q \in \cac_1^\mu$ with $\mu>1$, hence $q$ is constant, and as a consequence, $M=M'$.

\smallskip

\noindent
{\it Existence.}
By definition of $\cz \cac_3^\mu$, we know that there exists $B \in \cac_2$ such that $\der B=h$. Consider now the sequence $(\pi^n)_n$ of dyadic partitions of $[0,1]$, that is
$$\pi^n=\{ 0=t_0^n \leq t_1^n \leq \dots \leq t_{2^n}^n =1\}, \ \mbox{with} \ t_i^n=\frac{i}{2^n},$$
and set, for all $s,t\in [0,1]$,

$$M^n_{ts}=\begin{cases}
0 & \mbox{if} \ \pi^n \cap (s,t)=\emptyset,\\
B_{ts}-B_{t_{j}^n s}-B_{t t_{j}^n} & \mbox{if} \ \pi^n \cap (s,t)=\{t_{j}^n\},\\
B_{ts}-B_{t_{j}^n s}-B_{t t_{l}^n}-\sum_{i=j}^{l-1} B_{t_{i+1}^n t_i^n} & \mbox{if} \ \pi^n \cap (s,t)=\{t_j^n \leq \dots \leq t_l^n \}.
\end{cases}$$
It is readily checked that the mapping $M^n:s,t \mapsto M^n_{ts}$ is continuous on $[0,1]^2$. We are now going to show that the sequence $(M^n)_{n\in \mathbb{N}}$ converges in the space $\cac([0,1]^2;V)$ of continuous functions on $[0,1]^2$, endowed with the norm $\cn[\cdot\, ;\, \cac_2^{0}(V)]$.

\smallskip

Let then $s,t \in [0,1]$, $n\in \mathbb{N}$, and denote
\bean
\pi^n \cap (s,t) &=& \{ t_j^n \leq t_{j+1}^n \leq \dots  \leq t_l^n \}\\
 & =& \{ t_{2j}^{n+1} \leq t_{2j+2}^{n+1} \leq \dots \leq t_{2l-2}^{n+1} \leq t_{2l}^{n+1} \}, \ \  \mbox{with} \ j \leq l \leq 2^n.
\eean
If $s< t_{2j-1}^{n+1} $ and $t\leq t_{2l+1}^{n+1}$, then
$$\pi^{n+1} \cap (s,t)=\{ t_{2j-1}^{n+1} \leq t_{2j}^{n+1} \leq t_{2j+1}^{n+1} \leq \dots \leq t_{2l-1}^{n+1} \leq t_{2l}^{n+1} \},$$
and in that case
$$M^{n+1}_{ts}-M^n_{ts}=(\der B)_{t_{2j}^{n+1} t_{2j-1}^{n+1} s }+\sum_{i=j}^{l-1} (\der B)_{ t_{2i+2}^{n+1} t_{2i+1}^{n+1}t_{2i}^{n+1} },$$
which, since $\der B=h$, leads to
$$\norm{ M_{ts}^{n+1}-M_{ts}^n }_V \leq \cn[h;\, \cac_3^{\mu}(V)] \lp \frac{1}{2^{n+1}} \rp^\mu (1+l-j).$$
We proceed likewise for the cases $(s \geq t_{2j-1}^{n+1},t>t_{2l+1}^{n+1})$, $(s < t_{2j-1}^{n+1}, t > t_{2l+1}^{n+1})$ and $(s \geq t_{2j-1}^{n+1}, t\leq t_{2l+1}^{n+1})$, to finally get
\bean
\lVert M_{ts}^{n+1}-M_{ts}^n \rVert_V &\leq & \cn[h;\, \cac_3^{\mu}(V)] \lp \frac{1}{2^{\mu}} \rp^{n+1} (2+l-j)\\
& \leq & \frac{\cn[h;\, \cac_3^{\mu}(V)]}{2^\mu} \lcl 2 \lp \frac{1}{2^\mu}\rp^n +\lp \frac{1}{2^{\mu-1}} \rp^n \rcl,
\eean
and thus
$$\cn[ M^{n+1}-M^n ;\cac_2^0(V)] \leq \frac{\cn[h;\, \cac_3^{\mu}(V)]}{2^\mu} \lcl 2 \lp \frac{1}{2^\mu}\rp^n +\lp \frac{1}{2^{\mu-1}} \rp^n \rcl.$$
Since we have considered $\mu >1$, this proves that the series $\sum_n \cn[M^{n+1}-M^n;\, \cac_2^{0}(V)]$ converges, and thus $\sum_n (M^{n+1}-M^n)$ converges in $\cac([0,1]^2;V)$ endowed with the norm $\cn[\cdot\, ;\, \cac_2^{0}(V)]$, the latter space being complete. But, invoking the fact that $M^0=0$, we have $M^N=\sum_{n=0}^{N-1} (M^{n+1}-M^n)$, which entails the uniform convergence of $M^N$ towards an element $M\in \cac([0,1]^2;V)$. We can already notice that for all $n$, $M^n_{tt}=0$, which yields the same property for $M$, so that $M\in \cac_2$.

\smallskip

Take now $0\leq s \leq u \leq t \leq 1$ and denote $\pi^n \cap (s,u)=\{t_j^n, \dots, t_l^n \}$, $\pi^n \cap [u,t)=\{ t_{j'}^n, \dots, t_{l'}^n \}$, hence $\pi^n \cap (s,t)=\{t_j^n, \dots, t_l^n \}\cup \{ t_{j'}^n, \dots, t_{l'}^n \}$. Thus,
$$M^n_{ts}=B_{ts}-B_{t_j^n s}-B_{t t_{l'}^n}-\sum_{i=j}^{l-1} B_{t_{i+1}^n t_i^n}-\sum_{i=j'}^{l'-1} B_{t_{i+1}^n t_i^n}-B_{ t_{j'}^n t_l^n}.$$

We will assume that $t_{j'}^n > u$, the case $t_{j'}^n =u$ leading to the same relation (\ref{eq:chasles}). Then
\begin{multline*}
M^n_{ts} =B_{ts}+ \lc B_{us}-B_{t_j^n s}-B_{u t_l^n }-\sum_{i=j}^{l-1} B_{ t_{i+1}^n t_i^n} \rc +\lc B_{tu}-B_{t_{j'}^nu}-B_{t t_{l'}^n}-\sum_{i=j'}^{l'-1} B_{ t_{i+1}^n t_i^n} \rc \\
+B_{u t_n^n}+B_{t_{j'}^n u}-B_{us}-B_{tu}-B_{t_{j'}^nt_l^n},
\end{multline*}
which can be written as
\begin{eqnarray}\label{eq:chasles}
M^n_{ts}=M^n_{us}+M^n_{tu}+h_{tus}-h_{ t_{j'}^n u t_l^n},
\end{eqnarray}
Since $h\in \cac_3^\mu$, $\lim_{n\rightarrow \infty} h_{ t_{j'}^n u t_l^n} =0$, so that, by letting $n$ tend to infinity in the previous relation, we get $\der M=h$.

\smallskip

Finally, let us show that for any $s,t \in [0,1]$ and $n\in \mathbb{N}$,
\begin{eqnarray} \label{ineq:cmu}
\lVert M^n_{ts} \rVert_V \leq c_\mu \, \cn[h;\, \cac_3^{\mu}(V)] \lln t-s\rrn^\mu,
\end{eqnarray}
which will prove inequality (\ref{contraction}) and as a consequence, the Hölder regularity of $M$. To this end, fix $s,t \in [0,1]$, $n\in \mathbb{N}$. If $\pi^n \cap (s,t)=\emptyset$, the result is obvious. If $\pi^n \cap (s,t)=\{t_j^n\}$, $M^n_{ts}=(\der B)_{t t_j^n s}=h_{t t_j^n s}$, hence $\lVert M^n_{ts}\rVert_V \leq \cn[h;\, \cac_3^{\mu}(V)] \lln t-s \rrn^\mu \leq c_\mu \, \cn[h;\, \cac_3^{\mu}(V)] \lln t-s\rrn^\mu$. If $\pi^n \cap (s,t)=\{t_j^n, \dots ,t_l^n \}$, pick $k \in \{j+1, \dots, l-1\}$ such that 
$$| t_{k+1}^n -t_{k-1}^n | \leq \frac{2}{l-j-1} \lln t-s \rrn.$$
At this point, the previous relation does not seem very relevant insofar as the distances between two successive points of $\pi^n$ are equal. In fact, this relation will make sense when we iterate the scheme. Consider indeed the new partition $\hat\pi=\{t_j^n, \dots, t_{k-1}^n,t_{k+1}^n, \dots,t_l^n \}$ and define $\hat{M}_{ts}^n$ according to the same principle as $M^n_{ts}$, using $\hat\pi$ instead of $\pi^n \cap (s,t)$. Then
$$M^n_{ts}-\hat{M}_{ts}^n = B_{t_{k+1}^n t_{k-1}^n}-B_{t_k^n t_{k-1}^n}-B_{t_{k+1}^n t_k^n}= h_{t_{k+1}^n t_k^n t_{k-1}^n},$$
and as a result
$$\lVert M^n_{ts}-\hat{M}_{ts}^n \rVert_V \leq \cn[h;\, \cac_3^{\mu}(V)] \frac{2^\mu}{(l-j-1)^\mu } \lln t-s \rrn^\mu.$$
We iterate the procedure until the partition reduces to the empty set, to get
$$\lVert M^n_{ts} \rVert_V \leq  \cn[h;\, \cac_3^{\mu}(V)] \lln t-s \rrn^\mu \lp 2+2^\mu \sum_{k=1}^{l-j-1} \frac{1}{k^\mu}\rp \leq c_\mu \, \cn[h;\, \cac_3^{\mu}(V)] \lln t-s \rrn^\mu.$$

\end{proof}

The following corollary gives a first relation between the structures we have just introduced and generalized integrals, in the sense that it connects the operators $\der$ and $\Lambda$ with Riemann sums.

\begin{corollary}[Integration of small increments]
\label{cor:integration}
For any 1-increment $g\in\cac_2 (V)$, such that $\der g\in\cac_3^{1+}$,
set
$
\delta f = (\id-\Lambda \delta) g
$.
Then
$$
(\delta f)_{ts} = \lim_{|\Pi_{ts}| \to 0} \sum_{i=0}^n g_{t_{i+1}\, t_i},
$$
where the limit is over any partition $\Pi_{ts} = \{t_0=t,\dots,
t_n=s\}$ of $[t,s]$ whose mesh tends to zero. The
1-increment $\delta f$ is the indefinite integral of the 1-increment $g$.
\end{corollary}
\begin{proof}
For any partition $\Pi_{t}=\{s=t_0 < t_1 <...<t_n=t\}$ of $[s,t]$, write
$$(\der f)_{ts}=\sum_{i=0}^n (\der f)_{t_{i+1}t_i} =\sum_{i=0}^n g_{t_{i+1}t_i}-\sum_{i=0}^n \laa_{t_{i+1}t_i}(\der g).$$
Observe now that for some $\mu >1$ such that $\der g \in \cac_3^\mu$,
$$\Big\Vert\sum_{i=0}^n \laa_{t_{i+1}t_i}(\der g)\Big\Vert_V \leq \sum_{i=0}^n \norm{\laa_{t_{i+1}t_i}(\der g)}_V \leq \cn[\laa(\der g);\, \cac_2^\mu(V)] \, \lln \Pi_{ts}\rrn^{\mu-1} \, \lln t-s\rrn,$$
and as a consequence, $\lim_{\lln \Pi_{ts}\rrn \rightarrow 0} \sum_{i=0}^n \laa_{t_{i+1}t_i}(\der g) =0$.
\end{proof}

\subsection{Computations in $\cac_*$}\label{cpss}

For sake of simplicity, let us assume for the moment
that $V=\R$, and set $\cac_k(\R)=\cac_k$. Then
the complex $(\cac_*,\delta)$ is an (associative, non-commutative)
graded algebra once endowed with the following product:
for  $g\in\cac_n $ and $h\in\cac_m $ let  $gh \in \cac_{n+m} $
the element defined by
\begin{equation}\label{cvpdt}
(gh)_{t_1,\dots,t_{m+n-1}}=g_{t_1,\dots,t_{n}} h_{t_{n},\dots,t_{m+n-1}},
\quad
t_1,\dots,t_{m+n+1}\in\lot.
\end{equation}
In this context, we have the following useful properties.

\begin{proposition}\label{difrul}
The following differentiation rules hold true:
\begin{enumerate}
\item
Let $g,h$ be two elements of $\cac_1 $. Then
\begin{equation}\label{difrulu}
\der (gh) = \der g\,  h + g\, \der h.
\end{equation}
\item
Let $g \in \cac_1 $ and  $h\in \cac_2 $. Then
$$
\der (gh) = \der g\, h + g \,\der h, \qquad
\der (hg) = \der h\, g  - h \,\der g.
$$
\end{enumerate}
\end{proposition}

\begin{proof}
We will just prove (\ref{difrulu}), the other relations being equally trivial:
if $g,h\in\cac_1 $, then
$$
\lc \der (gh) \rc_{ts}
= g_th_t-g_sh_s
=g_t\lp h_t-h_s \rp +\lp  g_t-g_s\rp h_s\\
=g_t \lp \der h \rp_{ts}+ \lp \der g \rp_{ts} h_s,
$$
which proves our claim.

\end{proof}

\vspace{0.2cm}

The iterated integrals of smooth functions on $\lot$ are obviously
particular cases of elements of $\cac$ which will be of interest for
us, and let us recall  some basic  rules for these objects:
consider $f,g\in\cac_1^\infty $, where $\cac_1^\infty $ is the set of
smooth functions from $\lot$ to $\R$. Then the integral $\int dg \,
f$, which will be denoted by
$\cj(dg \,  f)$, can be considered as an element of
$\cac_2^\infty$. That is, for $s,t\in\lot$, we set
$$
\cj_{ts}(dg \,  f)\left(\int  dg f \right)_{ts} = \int_s^t  dg_u f_u.
$$
The multiple integrals can also be defined in the following way:
given a smooth element $h \in \cac_2^\infty$ and $s,t\in\lot$, we set
$$
\cj_{ts}(dg\, h )\equiv
\left(\int dg h \right)_{ts} = \int_s^t dg_u h_{us} .
$$
In particular, the double integral $\cj_{ts}( df^3df^2\,f^1)$ is defined, for
$f^1,f^2,f^3\in\cac_1^\infty$, as
$$
\cj_{ts}( df^3df^2\,f^1)
=\lp \int df^3df^2\,f^1  \rp_{ts}
= \int_s^t df_u^3 \,\cj_{us}\lp  df^2 \, f^1 \rp .
$$
Now, suppose that the $n$th order iterated integral of $df^n\cdots df^2 \,f^1$, still denoted by $\cj(df^n$ $\cdots df^2 \,f^1)$, has been defined for
$f^1,f^2\ldots, f^n\in\cac_1^\infty$.
Then, if $f^{n+1}\in\cac_1^\infty$, we set
\begin{equation}\label{multintg}
\cj_{ts}(df^{n+1}df^n \cdots df^2 f^1)\int_s^t  df_u^{n+1}\, \cj_{us}\lp df^n\cdots df^2 \,f^1\rp\,,
\end{equation}
which defines the iterated integrals of smooth functions recursively.
Observe that a $n$th order integral $\cj(df^n\cdots df^2 df^1)$ could be defined along the same lines.

\medskip

The following relations between multiple integrals and the operator $\der$ will also be useful in the remainder of the paper (see e.g. \cite{GT} for a proof of these elementary facts):
\begin{proposition}\label{dissec}
Let $f,g$ be two elements of $\cac_1^\infty$. Then, recalling the convention
(\ref{cvpdt}), it holds that
$$
\der f = \cj( df), \qquad
\der\lp \cj( dg f)\rp = 0, \qquad
\der\lp \cj (dg df)\rp =  (\der g) (\der f) = \cj(dg) \cj(df),
$$
and, in general,
$$
 \der \lp \cj( df^n \cdots df^1)\rp   =  \sum_{i=1}^{n-1}
\cj\lp df^n \cdots df^{i+1}\rp \cj\lp df^{i}\cdots df^1\rp.
$$
\end{proposition}

%%%%%%%%%%%%%%%%%%%%%%%%%%%%%%%%%%%%%%%%%%%
%%%%%%%%%%%%%%%%%%%%%%%%%%%%%%%%%%%%%%%%%%%

\section{Volterra equations in the Young setting}\label{sec:Young-volterra}

Recall that we wish to solve equation (\ref{eq:sde-laplace}), and we start this program by studying the Young case, i.e. the case of a driving process $x$ which is assumed to be $\ga$-Hölder continuous with $\ga>1/2$. This allows us to introduce most of the general tools used in the sequel, and this section is thus conceived as an introduction to the rough case which will be treated later on. In order to get a feeling of the kind of structure needed in order to deal with Volterra equations, we will start with some heuristic considerations, which are basically justified in case of a smooth driving noise $x$. Then we shall proceed to define rigorously the equation, and solve it in a suitable class of functions. 

\subsection{Heuristic considerations}
\label{sec:heuristic-young}
Assume for the moment that $x$ is a smooth process, in which case equation (\ref{eq:sde-laplace}) is well defined and solvable when $\si$ is a regular coefficient. In order to get an intuition of the natural operators associated to our equation, let us recast it, in quite a redundant way, as a system:
\beq\label{eq:volterra-system1}
\begin{cases}
y_t&=a+\int_0^\infty   \yti_t(\xi)\, \hph(\xi) \, d\xi  \\
\yti_t(\xi)&=\iot e^{-\xi(t-v)} dx_v\, \si(y_v).
\end{cases}
\eeq
Notice that the first relation above is not sufficient in order to determine $\yti$ as a function of $y$. However, the second one defines $\yti$ without ambiguity.

\smallskip

As we already mentioned in the introduction, if one wants to generalize the system we have just written to a non smooth signal $x$, it is now easily seen that the main step is to give a rigorous meaning to the integral $\iot e^{-\xi(t-v)} dx_v\, \si(y_v)$ defining $\yti_t(\xi)$. To this purpose, and having introduced the main tools of algebraic integration in the last section, the first idea one may have in mind is to get a suitable expression of the increments $(\der\yti(\xi))_{ts}\equiv \yti(\xi)_t-\yti(\xi)_s$ for a given $\xi$. And indeed in case of a smooth driving process $x$, invoking equation (\ref{eq:volterra-system1}), those  increments can be written as:
\bean
(\der \yti(\xi))_{ts}&=& \ist e^{-\xi(t-v)} dx_v\, \si(y_v) +  a_{ts}(\xi) \int_0^s e^{-\xi(s-v)} dx_v\, \si(y_v)  \\
&=&\ist e^{-\xi(t-v)} dx_v\, \si(y_v) +  a_{ts}(\xi) \yti_s(\xi),
\eean
where we have set $a_{ts}(\xi)=e^{-\xi(t-s)} -1$. Notice now that the first term $\ist e^{-\xi(t-v)} dx_v \si(y_v)$ above is really similar to what one obtains in the diffusion case, namely an integral of the form $\ist$. However, the second term $a_{ts}(\xi) \yti_s(\xi)$ is a little clumsy for further expansions. Hence, a straightforward idea is to make it disappear by just setting  $(\delti \yti )_{ts}(\xi)=(\der \yti)_{ts}(\xi)-a_{ts}(\xi)\yti_s(\xi)$. Then the last equation can be read as $(\delti \yti )_{ts}(\xi)=\ist e^{-\xi(t-v)} dx_v\, \si(y_v)$, and the system (\ref{eq:volterra-system1}) becomes 
\beq\label{eq:volterra-system2}
\begin{cases}
y_t&=a+\int_0^\infty   \yti_t(\xi)\, \hph(\xi) \, d\xi   \\
(\delti \yti )_{ts}(\xi)&=\ist e^{-\xi(t-v)} dx_v\, \si(y_v),
\end{cases}
\eeq
with the initial condition $\yti_0 \equiv 0$. This very simple fact, together with the nice algebraic properties which will be seen below, converts the elementary operator $\delti$ into the central object in order to solve our Volterra system.

\smallskip

These preliminaries being admitted, we shall essentially focus in the sequel on the process $\yti$, by merging the two equations of the last system into a single one:
\begin{eqnarray}\label{eq:volterra-ytilde}
(\delti \yti)_{ts}(\xi)=\int_s^t e^{-\xi(t-v)} dx_v \, \si\lp a+\int_0^\infty d\eta\, \hphi(\eta)\yti_v(\eta) \rp.
\end{eqnarray}
The original solution process $y$ can then be recovered in an obvious way, and we shall solve the Volterra equation under the form (\ref{eq:volterra-ytilde}).

\subsection{Convolutional increments}
\label{sect:convol-increments}

Let us turn now to the main concern of this section, that is the definition
of a complex $(\cacti_*,\delti)$ which behaves nicely for
the definition of our Volterra problem.

\smallskip

Notice that, due to the fact that $e^{-\xi(t_1-t_2)}$ is nicely bounded
only for $t_1>t_2$, our integration domains will
be of the form $\dn=\dn(\lot)$, where $\dn$ stands for the n-simplex
$$
\cs_n = \{(t_1,\dots,t_n) :  \ell_2\ge t_1 \ge t_2 \ge \cdots \ge t_n  \ge \ell_1\}.
$$
Let then $V$ be a separable Banach space.
In order to define the basic family of continuous increments we will
work with, we first need to specify the (Banach) functional space each $\yti_{ts}(\cdot)$ will belong to, with a special emphasis on the Laplace coordinate. In fact, the calculations to come (see for example Lemma \ref{lem:sigma}) incite us to consider the $\cl^1$-type space induced by the norm
$$\cn[\gti; \cl_\beta(V)]:=\cn[\gti; \cl_{\beta,\hphi}(V)]=\int_0^\infty d\xi |\hphi(\xi)| (1+\xi^ \beta) \norm{\gti(\xi)}_V,$$
where $\beta >0$ is fixed. Then, we define $\cacti_{n, \beta}$ as the space of continuous applications from $\cs_n$ to $\cl_ \beta(V)$.
Observe that an operator
$\der:\cacti_{n,\be}(V)\to\cacti_{n+1,\be}(V)$ can be
defined  just like in (\ref{eq:coboundary}). In particular, if $\Ati\in\cacti_{1,\be}(V)$
and $\Bti\in\cacti_{2,\be}(V)$, the relation (\ref{eq:simple_application}) is still
valid. 

\smallskip

As we have seen at Section \ref{sec:heuristic-young}, a suitable family of operators related to our Volterra equation is given by $\delti:\cacti_{n,\be}(V)\to\cacti_{n+1,\be}(V)$,
defined for any positive $\xi$ by
\begin{equation}\label{defhd}
[ \delti \Ati ]_{t_1\dots t_{n+1}}(\xi)=[ \der \Ati  ]_{t_1\dots t_{n+1}}(\xi)
- a_{t_1t_2}(\xi)\Ati_{t_2\dots t_{n+1}}(\xi),
\quad\mbox{ for }\quad
\Ati\in\cacti_{n,\be}(V),
\end{equation}
where $(t_1\dots t_{n+1})\in\cs_{n+1}$. In the remainder of the paper, we will explicitly write the variable $\xi$ down only when there might be a confusion. Thus, we simply write
$\delti  \Ati=\der \Ati-a\, \Ati$, where we made use of the convention
(\ref{cvpdt}).
As in Section \ref{sec:incr}, one can define, for $n\ge 1$, 
$$
\cz\cacti_{n,\be}(V)=\cacti_{n,\be}(V) \cap \ker(\delti),
\quad\mbox{ and }\quad
\cb\cacti_{n,\be}(V)=\cacti_{n,\be}(V) \cap \mbox{Im}(\delti)
$$

In fact, when $V=\R^k$ or $V=\R^{k,d}$, endowed with their natural Euclidian norms, the convention (\ref{cvpdt}) can be extended according to the following principle:
\begin{lemma}\label{lem:delta-prod-ML}
Let $\Mti\in\cacti_{n,\be}(\R^{k,l})$ and $L\in \cac_m(\R^l)$. Then $\Mti L$, defined by the relation
$$(\Mti L)_{t_1 \ldots t_{m+n-1}}(\xi)=\Mti_{t_1 \ldots t_n}(\xi) \, L_{t_n \ldots t_{m+n-1}},$$
belongs to $\cacti_{m+n-1,\be}(\R^k)$. Moreover, when $n=2$, the following algebraic relations hold true:
$$
\der( \Mtil L )= \der \Mtil \, L- \Mtil\,  \der L ,\quad\mbox{and}\quad
\delti (\Mtil L) = \delti \Mtil  \, L- \Mtil\,  \der L .
$$
\end{lemma}

\begin{proof}
The first part of our claim is an obvious consequence of 
$$\norm{\Mti_{t_1 \ldots t_n}(\xi) \, L_{t_n \ldots t_{m+n-1}}}_{\R^k} \leq \norm{\Mti_{t_1 \ldots t_n}(\xi)}_{\R^{k,l}} \norm{ L_{t_n \ldots t_{m+n-1}}}_{\R^n}.$$
As for the algebraic relations, the first one follows from Proposition \ref{difrul}, while
\bean
\lefteqn{ \delti(\Mti L)_{t_1...t_{m+2}}}\\
&=& \der(\Mti L)_{t_1...t_{m+2}}-a_{t_1t_2} \Mti_{t_2t_3} L_{t_3...t_{m+2}}\\
&=&(\der \Mti)_{t_1t_2t_3} L_{t_3...t_{m+2}}-\Mti_{t_1t_2}(\der L)_{t_2...t_{m+2}}-a_{t_1t_2} \Mti_{t_2t_3} L_{t_3...t_{m+2}}\\
&=& [(\der \Mti)_{t_1t_2t_3}-a_{t_1t_2}\Mti_{t_2t_3}] L_{t_3...t_{m+2}}-(\Mti \, \der L)_{t_1...t_{m+2}}.
\eean
\end{proof}

With these preliminaries in hand, it is now easily shown that the perturbed operators $\delti$ preserve some important properties
of the original coboundary $\delta$:
\begin{proposition}\label{acyc}
$\delti\delti =0$. More precisely, the couple $(\cacti_{*,\be}(V),\delti)$ satisfies $\cz \cacti_{n,\be}(V) =\cb\cacti_{n,\be}(V)$ for all $n \ge 0$.
\end{proposition}

\begin{proof}
This proof is borrowed from \cite[Proposition 3.1]{GT} and is included here for sake of completeness. If $\Fti\in\cacti_{n,\be}(V)$,
according to the fact that
$\der\der=0$ and thanks to Lemma \ref{lem:delta-prod-ML}, we have
\begin{eqnarray*}
\delti\delti \Fti&=&
\lp \der-a \rp [ \lp \der-a \rp \Fti  ]
=\der\der \Fti -\der(a \,\Fti)-a \,\der \Fti+ a \, a \, \Fti\\
&=&-\der a \, \Fti +a \, \der \Fti -a\, \der \Fti+ a\, a\, \Fti
=a\, a\, \Fti-\der a\, \Fti.
\end{eqnarray*}
Furthermore, it is readily checked that
$$
(\delta a)_{tus} = a_{tu}\, a_{us}, \qquad (t,u,s) \in \cs_3,
$$
which gives $\delti\delti \Fti=0$.

\smallskip

The fact that
$\mbox{Im}\delti_{|\cacti_{n+1,\be}(V)}=\ker\delti_{|\cacti_{n+1,\be}(V)}$ can be proved along the
same lines as for the $(\cac_*,\delta)$ complex~\cite{Gu}:
pick $\Ati\in\cacti_{n+1,\be}(V)$ such that
$\delti \Ati=0$, and set $\Bti_{t_1\dots t_n}=\Ati_{t_1\dots t_n s}$, with $s=0$. Then
\begin{eqnarray*}
[\delti \Bti]_{t_1\dots t_{n+1}}&=&
[\der \Bti]_{t_1\dots t_{n+1}s}
+(-1)^{n+1}\Ati_{t_1\dots t_{n+1}}-a_{t_1t_2} \Ati_{t_2\dots t_{n}s}\\
&=&[\delti \Ati]_{t_1\dots t_{n+1}s}+(-1)^{n+1}\Ati_{t_1\dots t_{n+1}}
=(-1)^{n+1}\Ati_{t_1\dots t_{n+1}}.
\end{eqnarray*}
Thus, setting $\Cti=(-1)^{n+1} \Bti$, we get $\delti \Cti=\Ati$.

\end{proof}

The cochain complex $(\cacti_{*,\be}(V),\delti)$ will be the structure at the base of all the constructions in this paper.
Let us also mention at this point that, when the meaning is obvious, we will transpose the notations of Section \ref{sec:incr} to our
convolutional setting. Furthermore, whenever this doesn't lead to an ambiguous situation, we will write $\cacti_{n,\be}$ instead of $\cacti_{n,\be}(V)$.

\smallskip

We will now define an equivalent of the iterated integrals of Section
\ref{cpss} in our convolution context: for two smooth functions $f,g$,
$\ell_1\le s < t \le \ell_2$ and $\xi\ge 0$, define
\begin{equation*}
\cj_{ts}(\hddx g\, f )(\xi)=\int_s^t e^{-\xi(t-v)} dg_v\, f_v,
\end{equation*}
and for $h \in \cac_2^\infty$,
$$\cj_{ts}(\dti g \, h )(\xi) =\int_s^t e^{-\xi(t-v)} dg_v \, h_{vs}.$$
Once these elementary blocks have been defined, the iterated integrals
\begin{equation}\label{eq:not-iter-convol-intg}
\cj( \dti g^1 \dots \dti g^n \, f )
\quad\mbox{ for }\quad
g^1,\dots,g^n,f\in \cac_1^\infty,
\end{equation}
should be defined as functions of several variables, according to the same recursive principle as in Section \ref{cpss}:
$$\cj_{ts}(\dti g^1 \dots \dti g^n \, f)(\xi^1, \dots, \xi^{n})=\int_s^t e^{-\xi^1(t-v)} dg^1_v \, \cj_{vs}(\dti g^2 \dots \dti g^n  \,  f)(\xi^2, \dots, \xi^{n}).$$
In particular, $\cj_{ts}(\dti x \dti x)(\xi,\eta)=\int_s^t e^{-\xi(t-v)} dx_v \int_s^v e^{-\eta(v-w)} dx_w$.

\smallskip

The following relations between $\delti$ and these integrals will be useful for our purposes:
\begin{proposition}\label{prop:dif-convol-intg}
Let $f,g\in\cac_1^\infty$. Then
$$
\delti \lp \cj(\dti g \, f)\rp=0\qquad
\delti\lp \cj(\dti g \, \der f)\rp=\cj( \dti g) \, \der f.
$$
\end{proposition}
\begin{proof}
Straightforward.

\end{proof}

\subsection{Hölder spaces and $\Lati$-map}
\label{sec:holder-spaces}
In the Young setting, it will be enough to let our solution live in some Hölder-type spaces. Indeed, one expects the solution $y$ to (\ref{eq:volterra-system2}) to belong to a space of the form $\cac_1^\beta$ (as defined at Section \ref{sec:incr}) for any $1-\ga<\beta<\gamma$, where $\gamma$ is the Hölder regularity exponent of the noise $x$. Since $\beta+\ga>1$ in this case, the exponentially weighted integrals with respect to $x$ can be interpreted in the Young sense, as will be explained below.

\smallskip

As far as the path $\yti$ alluded to in (\ref{eq:volterra-system2}) is concerned, we also expect his increments $\delti \yti$ to be regular enough. Thus, we shall resort to the following natural Hölder spaces: 
$$\cacti_{2,\be}^\mu:= \{ \yti \in \cacti_{2,\be}: \ \cn[\yti; \cacti_{2,\be}^\mu]:=\sup_{0\leq s < t \leq T} \frac{\cn[\yti_{ts}; \cl_\be]}{\lln t-s\rrn^\mu} < \infty \},$$
$$\cacti_{1,\be}^\mu:= \{ \yti \in \cacti_{1,\be}: \ \delti \yti \in \cacti_{2,\be}^\mu \}.$$
Notice that our definition of the space $\cacti_{1,\be}^\mu$ is based on the twisted operator $\delti$ instead of $\der$.
For any $\hti \in \cacti_{3,\be}$, set, just as in the standard case,
$$\cn[\hti;  \cacti_{3,\be}^{(\ga,\rho)}]:=\sup_{0\leq s <u<t \leq T} \frac{\cn[\hti_{tus}; \cl_\be]}{\lln t-u \rrn^\ga \lln u-s \rrn^{\rho}} ,$$
$$\cn[\hti; \cacti_{3,\be}^\mu ]:= \inf \lcl \sum_i \cn[h_i; \cacti_{3,\be}^{(\rho_i,\mu-\rho_i)}]; \ h=\sum_i h_i, \ 0<\rho_i < \mu\rcl,$$
where the last infimum is taken over all sequences $\{\hti \in \cacti_{3,\be}\}$ such that $h=\sum_i h_i$ and for all choices of the numbers $\rho_i \in (0,\mu)$.
Denote also $\cz \cacti_{k,\be}^\mu := \text{Im}(\delti) \cap \cacti_{k,\be}^\mu$ and observe that the property $\cz \cac_2^\mu =\{0\}$ if $\mu >1$ remains true for $\delti$:
\begin{lemma}\label{lem:zcacti-deux-mu}
If $\mu >1$, then $\cz \cacti_{2,\be}^\mu =\{0\}$.
\end{lemma}
\begin{proof}
Let $\Mti=\delti \fti \in \cz \cacti_{2,\be}^\mu$. Consider the telescopic sum $(\delti \fti)_{ts}=\sum_{i=0}^n e^{- . \, (t-t_{i+1})} (\delti \fti)_{t_{i+1}t_i}$
with respect to the partition $\Pi_{ts}=\{s=t_0 <t_1 <...<t_n=t\}$ of the interval $[s,t]$. Then
$$\cn[\Mti_{ts}; \cl_\be] \leq  \sum_{i=0}^n \cn[(\delti \fti)_{t_{i+1}t_i}; \cl_\be] \leq  \cn[\Mti; \cacti_{2,\be}^\mu] \, \lln t-s \rrn \, \lln \Pi_{ts}\rrn^{\mu-1},$$
which tends to $0$ as the mesh $\lln \Pi_{ts}\rrn$ of the partition decreases to $0$.

\end{proof}
 
\smallskip

As we already mentioned, an essential tool in order to define generalized convolutional integrals is the following inverse of the operator $\delti$:
\begin{proposition}[The convolutional sewing map]\label{th:conv-lambda}
Let $\mu >1, \be >0$. For any $\hti \in \cz \cacti_{3,\be}^\mu$, there exists a unique $\Lati \hti \in \cacti_{2,\be}^\mu$ such that $\delti(\Lati \hti)=\hti$. Furthermore,
\begin{equation}\label{ineqhla}
\cn[\Lati \hti; \cacti_{2,\ka}^\mu] \leq c_\mu \, \cn[\hti; \cacti_{3,\ka}^\mu],
\end{equation}
with $c_\mu=2+2^\mu \sum_{k=1}^\infty \frac{1}{k^\mu}$. This gives rise to a linear continuous map $\Lati:  \cz \cacti_{3,\be}^\mu \rightarrow \cacti_{2,\be}^\mu$ such that $\delti\Lati =\id_{ \cz \cacti_{3,\be}^\mu}$. 
\end{proposition}
\begin{proof}
It follows the same line as the proof of Proposition \ref{prop:Lambda}.

\smallskip

\noindent
{\it Uniqueness.}
Let $\Mti, \Mti' \in \cacti_{2,\be}^\mu$ such that $\delti \Mti=\delti \Mti'=\hti$. In particular, $\delti(\Mti-\Mti')=0$, hence $\Mti-\Mti' \in \cz \cacti_{2,\be}^\mu$. Thanks to lemma \ref{lem:zcacti-deux-mu}, we deduce $\Mti=\Mti'$.

\smallskip

\noindent
{\it Existence.}
As in the standard case, consider $\Bti \in \cacti_{2,\be}$ such that $\delti \Bti=\hti$ and construct
$$\Mti_{ts}^n=\begin{cases}
0 & \mbox{if} \ \pi^n \cap (s,t)=\emptyset\\
(\delti \Bti)_{t t_j^n s} & \mbox{if} \ \pi^n \cap (s,t)=\{t_j^n\}
\end{cases}$$
and if $\pi^n \cap (s,t)=\{t_j^n \leq ... \leq t_l^n\}$,
$$\Mti_{ts}^n=\Bti_{ts}-\Bti_{tt_l^n}-\sum_{i=j}^{l-1} e^{-.\, (t-t_{i+1}^n)} \Bti_{t_{i+1}t_i}-e^{-.\, (t-t_j^n)}\Bti_{t_j^ns},$$
where $\pi^n$ stands for the $n$-dyadic partition of $[0,1]$.\\
It is readily checked that $\Mti^n$ is continuous from $[0,1]^2$ to $\cl_\be$. Moreover, if for instance $\pi^n \cap (s,t)=\{t_{2j}^{n+1} \leq t_{2j+2}^{n+1}\leq ...\leq t_{2l}^{n+1} \}$, with $s < t_{2j-1}^{n+1}$ and $t >t_{2l+1}^{n+1}$, then
$$\Mti_{ts}^{n+1}-\Mti_{ts}^n =(\delti \Bti)_{tt_{2l+1}^{n+1}t_{2l}^{n+1}}+e^{-.\, (t-t_{2j}^{n+1})}(\delti \Bti)_{t_{2j}^{n+1}t_{2j-1}^{n+1}s} 
 +\sum_{i=j}^{l-1} e^{-.\, (t-t_{2i+2}^{n+1})}(\delti \Bti)_{t_{2i+2}^{n+1}t_{2i+1}^{n+1}t_{2i}^{n+1}},$$
and since $\delti \Bti=\hti$, this yields:
$$\cn[\Mti_{ts}^{n+1}-\Mti_{ts}^n; \cl_\be] \leq \frac{\cn[\hti; \cacti_{3,\be}^\mu]}{2^\mu} \lcl 2 \lp \frac{1}{2^\mu}\rp^n +\lp \frac{1}{2^{\mu-1}}\rp^n \rcl.$$
This estimation remains true for the other cases of intersection of $\pi^n$ with $(s,t)$. Using the same arguments as with $\laa$, we thus get the existence of a limit $\Mti$ of $\Mti^n$ in $\cacti_{2,\be}$.

\smallskip

The fact that $\delti \Mti =\hti$ can be proved just as in Proposition \ref{prop:Lambda}, and it is the same for the estimation
$$\cn[\Mti_{ts}^n; \cl_\be]\leq c_\mu \, \cn[\hti; \cacti_{3,\be}^\mu] \, \lln t-s \rrn^\mu.$$

\end{proof}

We also have the following equivalent of Corollary \ref{cor:integration}, which links $\Lati$ with convolutional Riemann sums,  at our disposal:
\begin{corollary}\label{cor:integration-tilde}
For any $1$-increment $\gti \in \cacti_{2,\beta}$ such that $\delti \gti \in \cacti_{3, \beta}^{\mu}$ ($\mu >1$), set $\delti \fti=(\id-\Lati \delti)\gti$. Then
$$(\delti \fti)_{ts} =\lim_{\lln \Pi_{ts} \rrn \rightarrow 0} \sum_{i=0}^n e^{-.(t-t_{i+1})}\gti_{t_{i+1}t_i} \quad \mbox{in} \ \cl_ \beta,$$
where the limit is over any partition $\Pi_{ts} = \{t_0=t,\dots,
t_n=s\}$ of $[t,s]$ whose mesh tends to zero.
\end{corollary}
\begin{proof}
We use the same arguments as in the standard case, starting from
$(\delti \fti)_{ts}=\sum_i e^{-.(t-t_{i+1})}(\delti \fti)_{t_{i+1}t_i}.$ 

\end{proof}

\subsection{Young convolution integral}

Recall that, according to the notations of Section \ref{sect:convol-increments}, the Volterra equation (\ref{eq:volterra-ytilde}) we are interested in can be read as 
\begin{eqnarray} \label{eq:volterra-ytilde-bis}
\yti_0 \equiv 0, \qquad \delti \yti =\cj\lp \dti x \, \si\lp a+\int_0^\infty d\eta \, \hphi(\eta)\yti(\eta)\rp \rp.
\end{eqnarray}
We will now define integrals of the form $\cj(\dti x \, z)$, such as the one appearing in the right hand side of equation (\ref{eq:volterra-ytilde-bis}), when $x,z$ are only $\ga$-Hölder with $\ga>1/2$. This will rely on the following assumption, which is trivially met when $x$ is a smooth path:
\begin{hypothesis}\label{hyp:X1}
Assume that, for some $\ga \in (1/2,1)$, $x$ is a path in $\cac_1^{\ga}(\R^{1,d})$, allowing to define an increment $\xti^1 \in \cacti_{2,\ga}^\ga(\R^{1,d})$ which satisfies $\delti \xti^1=0$.
\end{hypothesis}

\begin{remark}\label{rmk:X1}
The increment $\xti^1$ represents morally the integral $\cj(\hddx x)$, which will be defined as a Wiener integral in the fractional Brownian case (see Section \ref{subsec:application}).  Furthermore, under Hypothesis~\ref{hyp:X1}, the increment $x^1_{ts}\equiv\int_0^\infty \hxu_{ts}(\xi) \, \hph(\xi) \, d\xi$ is well defined as an element of $\cac_2^{\ga}$. The fact that $\xti^1(\cdot)\in\cl_{\ga}$ will be simply ensured by the condition $\int_0^\infty (1+\xi^\ga) |\hph(\xi)|\, d\xi<\infty$.
\end{remark}

\begin{theorem}
\label{th:young}
Let $x$ be a path from $\ott$ to $\R^{1,d}$ satisfying Hypothesis~\ref{hyp:X1}, for a given $\ga\in(1/2,1)$.
Let  $z\in\cac_{1}^{\ga}$, and for $\xi\in\R_+$, set
\beq\label{eq-def-convol-intg-young}
\cj(\hddx x \, z)(\xi)=\hxu(\xi)\,  z+\Lati [\hxu\, \der z ](\xi)=( \id-\Lati \delti  ) [\hxu\, z](\xi).
\eeq
Then
\begin{enumerate}
\item
$\cj(\hddx x \, z)$ is well defined as an element of $\cacti_{2,\ga}^\ga$, and coincides with the usual Riemann integral $\int_s^t e^{-\xi(t-v)} dx_v\, z_v$ when $\ga=1$.
\item
For a constant $c_x>0$, we have, for all $\ell_1 < \ell_2$,
$$
\cn[\cj(\dti x \,z);\cacti_{2,\ga}^\ga([\ell_1,\ell_2])] 
\le
c_x
 \lcl \cn[z;\cac_1^0([\ell_1,\ell_2])]+  \ep^\ga \cn[z;\cac_1^{\ga}([\ell_1,\ell_2])]\rcl,
$$
where
the norms $\cn$ have been defined at Section \ref{sec:holder-spaces}, $\cn[z; \cacti_{1}^0]:=\sup_{\ell_1\leq s \leq \ell_2}\lVert z_s\rVert$ and $\ep=\lln \ell_2-\ell_1\rrn$.
\item
It holds that, for any $\ell_1\le s<t\le \ell_2$,
$$
\cj_{ts}(\hddx x\,  z)
=\lim_{|\Pi_{ts}|\to 0}\sum_{i=0}^n e^{-.(t-t_{i+1})} \hxu_{t_{i+1},t_{i}} \,z_{t_{i}} \quad \mbox{in} \ \cl_\ga,
$$
where the limit is over all partitions $\Pi_{ts} = \{t_0=t,\dots,
t_n=s\}$ of $[s,t]$ as the mesh of the partition goes to zero.
\end{enumerate}
\end{theorem} 

\begin{proof}
(1) In the regular case, the integral $\cj_{ts}(\hddx x\,  z)(\xi)\equiv\int_s^t e^{-\xi(t-v)} dx_v\, z_v$ is defined in the Riemann sense, and it is readily checked that
\beq\label{eq:exp-cj-hdx-z-regular}
\cj_{ts}(\hddx x\,  z) = \hxu_{ts}\,  z_s+ \cj_{ts}(\hddx x\,  \der z),
\eeq
and hence
$$
\cj_{ts}(\hddx x\,  \der z)=\cj_{ts}(\hddx x\,  z)-\hxu_{ts}\,  z_s.
$$
Applying $\delti$ to both sides of this last relation and taking into account Proposition~\ref{prop:dif-convol-intg}, Lemma~\ref{lem:delta-prod-ML} and Hypothesis~\ref{hyp:X1}, we obtain:
$$
\delti\lp  \cj(\hddx x\,  \der z) \rp = - \delti \hxu \, z + \hxu \, \der z
= \hxu \, \der z.
$$
Now, if $\hxu$ and $z$ are $\ga$-Hölder continuous with $\ga>1/2$, $\Lati$ can be applied to the relation above, and one can write $\cj_{ts}(\hddx x\,  \der z)=\Lati(\hxu \, \der z)$. Plugging this equality into (\ref{eq:exp-cj-hdx-z-regular}), we obtain the expression (\ref{eq-def-convol-intg-young}). Thus our integral coincides with the usual one in case of a regular process $x$.

\smallskip

Since $2\ga >1$, the item (2) is a direct consequence of the contraction property (\ref{ineqhla}) of $\Lati$. As for (3), it stems from Corollary \ref{cor:integration-tilde}. 

\end{proof}

\subsection{Volterra equations}
We are now ready to solve equation (\ref{eq:volterra-ytilde-bis}), by interpreting the integral $\cj( \dti x \, \si( a+\int_0^\infty d\eta \, \hphi(\eta)\yti(\eta)) )$ as in Theorem \ref{th:young}. Before stating the main theorem in this direction, let us introduce the subspace $\cacti_{1,\ga}^{0,\ga}$ of $\cacti_{1,\ga}^\ga$ induced by the norm
$$\cn[ \yti; \cacti_{1,\ga}^{0,\ga}]=\cn[\yti; \cacti_{1,\ga}^0]+\cn[\yti; \cacti_{1,\ga}^\ga],$$
where $\cn[\yti; \cacti_{1,\ga}^0]:=\sup_{0\leq s \leq T} \cn[\yti_s; \cl_\ga]$. With this new space in hand, we can prove the following elementary lemma, which will be used throughout the proof of the theorem:

\begin{lemma}\label{lem:sigma}
Let $\si \in \cac_b^2$ and for any $\yti \in \cacti_{1,\ga}^{0,\ga}$, set $y=a+\int_0^\infty d\eta \, \hphi(\eta)\yti(\eta)$. Then $\si(y) \in \cac_1^\ga$ and
$$\cn[\si(y); \cac_1^\ga] \leq c_\si \, \cn[\yti; \cacti_{1,\ga}^{0,\ga}].$$
Moreover, if $\ytiun,\ytide \in \cacti_{1,\ga}^{0,\ga}([\ell_1,\ell_2])$ are such that $\ytiun_{\ell_1}=\ytide_{\ell_1}$, then
$$\cn[\si(\yun)-\si(\yde); \cac_1^{0}([\ell_1,\ell_2])] \leq c_\si \, \ep^\ga \, \cn[\ytiun-\ytide; \cacti_{1,\ga}^{0,\ga}([\ell_1,\ell_2])],$$
$$\cn[\si(\yun)-\si(\yde); \cac_1^{\ga}([\ell_1,\ell_2])] \leq c_\si  \lcl 1+\cn[\ytide; \cacti_{1,\ga}^{0,\ga}([\ell_1,\ell_2])]\rcl\, \cn[\ytiun-\ytide; \cacti_{1,\ga}^{0,\ga}([\ell_1,\ell_2])],$$
with $\ep=\lln \ell_1-\ell_2\rrn$.
\end{lemma}

\begin{proof}
The three inequalities are mostly due to the obvious estimation
$$\lln a_{ts}(\xi)\rrn =\lln e^{-\xi(t-s)}-1\rrn =\lln e^{-\xi(t-s)}-1\rrn^{1-\ga} \lln e^{-\xi(t-s)}-1\rrn^\ga \leq \lln t-s\rrn^\ga \xi^\ga.$$
Indeed, we have for instance
\bean
\cn[ \der(\si(y))_{ts} ]_V & \leq & \norm{\si'}_\infty \int_0^\infty d\xi |\hphi(\xi)| \lln (\der \yti)_{ts}(\xi)\rrn\\
&\leq & \norm{\si'}_\infty \lcl \int_0^\infty d\xi |\hphi(\xi)| | (\delti \yti)_{ts}(\xi)|+\int_0^\infty d\xi | \hphi(\xi)| \lln a_{ts}(\xi)\rrn |\yti_s(\xi)| \rcl\\
&\leq &  \norm{\si'}_\infty \lln t-s\rrn^\ga \lcl \cn[\yti; \cacti_{1,\ga}^\ga]+\cn[\yti; \cacti_{1,\ga}^0]\rcl,
\eean
and therefore $\cn[\si(y); \cac_1^\ga] \leq \norm{\si'}_\infty \, \cn[\yti; \cacti_{1,\ga}^{0,\ga}]$.

\smallskip

The second inequality can be obtained in the same way, after noticing that, for all $s\in [\ell_1,\ell_2]$,
$$\norm{ \si(\yun_s)-\si(\yde_s) }_V \leq \norm{\si'}_\infty \int_0^\infty d\xi | \hphi(\xi)| \norm{ \der(\ytiun-\ytide)_{s\ell_1}(\xi) }_V.$$
As far as the third inequality is concerned, we can invoke the classical estimation
\begin{multline*}
\norm{\der(\si(\yun)-\si(\yde))_{ts}}_V \leq \norm{\si'}_\infty \norm{\der(\yun-\yde)_{ts}}_V +\norm{\si''}_\infty \norm{\der(\yde)_{ts} }_V  \\
\lp \norm{\yun_t-\yde_t}_V+\norm{\yun_s-\yde_s}_V \rp.
\end{multline*}

\end{proof}

We are now in position to prove the

\begin{theorem}
\label{th:young-volterra}
Assume Hypothesis~\ref{hyp:X1} holds true for some $\ga >1/2$, and that $\si\in C^{2,b}$.
Then equation (\ref{eq:volterra-ytilde-bis}) admits a unique solution in $\cacti_{1,\ga}^{0,\ga}$, where the integral $\cj( \dti x \, \si( a+\int_0^\infty d\eta \, \hphi(\eta)\yti(\eta)) )$ stands for the Young convolutional integral introduced in Theorem \ref{th:young}. 
\end{theorem}

\begin{proof}
Let $\ep >0$ (we shall fix this constant retrospectively), $l\in \mathbb{N}$, and suppose that we have already constructed a solution $\yti^l \in \cacti_{1,\ga}^{0,\ga}([0,l\ep])$. If $l=0$, then $\yti^0=\yti^0_0=0$. We mean to extend $\yti^l$ into a solution $\yti^{l+1} \in \cacti_{1,\ga}^{0,\ga}([0,(l+1)\ep])$, by resorting to a fixed point argument.

\smallskip

\noindent
{\it Step 1: Existence of invariant balls.}
Let $\yti \in \cacti_{1,\ga}^{0,\ga}([0,(l+1)\ep])$ such that $\yti_{|[0,l\ep]}=\yti^l$ and set $\zti=\Gamma(\yti)$ the element of $\cacti_{1,\ga}([0,(l+1)\ep])$ defined by $\zti_{|[0,l\ep]} =\yti^l$ and for all $s,t \in [0,(l+1)\ep]$, $(\delti \zti)_{ts}=\cj_{ts} \lp \dti x \, \si(y)\rp$, where, as in Lemma \ref{lem:sigma}, $y=a+\int_0^\infty d\eta \, \hphi(\eta) \yti(\eta)$.

\smallskip

We know from Theorem \ref{th:young} that
$$\cn[\zti; \cacti_{1,\ga}^\ga([l\ep,(l+1)\ep])] \leq c_{x} \lcl \cn[\si(y); \cac_1^0([0,(l+1)\ep])]+\ep^\ga \, \cn[\si(y); \cac_1^\ga([0,(l+1)\ep])] \rcl,$$
which, together with Lemma \ref{lem:sigma}, gives
$$\cn[\zti; \cacti_{1,\ga}^\ga([l\ep,(l+1)\ep])] \leq c^1_{x,\si} \lcl 1+\ep^\ga \cn[\yti; \cacti_{1,\ga}^{0,\ga}([0,(l+1)\ep])] \rcl.$$
If $0\leq s \leq l\ep \leq t \leq (l+1)\ep$, we use the relation $\delti \delti=0$ to deduce
$$0=(\delti \delti \zti)_{t,l\ep, s} =(\delti \zti)_{ts}-(\delti \zti)_{t,l\ep}-e^{-\xi(t-l\ep)} (\delti \zti)_{l\ep,s},$$
and hence
\bean
\cn[ (\delti \zti)_{ts}; \cl_{\ga}] &\leq & \cn[(\delti \zti)_{t,l\ep}; \cl_{\ga}]+\cn[(\delti \zti)_{l\ep,s}; \cl_{\ga}]\\
& \leq & 2\max \lp \cn[\zti; \cacti_{1,\ga}^\ga([l\ep,(l+1)\ep]), \cn[\yti^l; \cacti_{1,\ga}^\ga([0,l\ep])] \rp \lln t-s\rrn^\ga.
\eean
Furthermore, for all $s,t \in [0,(l+1)\ep]$, $\zti_s=(\delti \zti)_{s0}$, and thus 
$$\cn[\zti; \cacti_{1,\ga}^0([0,(l+1)\ep])] \leq \cn[\zti; \cacti_{1,\ga}^{0,\ga}([0,(l+1)\ep])] T^\ga.$$

\smallskip

We are therefore incited to set
\bean
\ep&=&\lp 4c^1_{x,\si} (1+T^\ga)\rp^{-1/\ga}  \\
N_{l+1}&=&\max \lp 2(1+T^\ga) \cn[\yti^l; \cacti_{1,\ga}^\ga([0,l\ep])], 4 c^1_{x,\si} (1+T^\ga)\rp.
\eean
Indeed, for such values, it is readily checked that if $\cn[\yti; \cacti_{1,\ga}^{0,\ga}([0,(l+1)\ep])] \leq N_{l+1}$, then $\cn[\zti; \cacti_{1,\ga}^\ga([0,(l+1)\ep])] \leq \frac{N_{l+1}}{1+T^\ga}$ and $\cn[\zti; \cacti_{1,\ga}^0([0,(l+1)\ep])] \leq \frac{N_{l+1}}{1+T^\ga} T^\ga$, which gives $\cn[\zti; \cacti_{1,\ga}^{0,\ga}([0,(l+1)\ep])] \leq N_{l+1}$. In other words, the ball
$$\cq_{\yti^l,(l+1)\ep}^{N_{l+1}}=\{ \yti \in \cacti_{1,\ga}^{0,\ga}([0,(l+1)\ep]): \ \yti_{|[0,l\ep]}=\yti^l, \ \cn[\yti; \cacti_{1,\ga}^{0,\ga}([0,(l+1)\ep])]\leq N_{l+1} \}$$
is left invariant by $\Gamma$.

\smallskip

The independance of $\ep$ with respect to the initial condition $\yti^l$ allows to repeat the scheme with the same $\ep$ and thus to get a sequence of radii $(N_k)_{k\geq 1}$ such that the sets $\cq_{\yti^k,k\ep}^{N_k}$ are invariant by $\Gamma$. Of course, the definition of the latter mapping has to be adapted (in the natural way) to each of those sets.

\smallskip

\noindent
{\it Step 2: Contraction property.} We will now search for a division of the previous intervals $[l\ep,(l+1)\ep]$ into subintervals $[l\ep, l\ep+\eta],[l\ep +\eta,l\ep+2\eta],\dots$ of the same lenght $\eta$ (possibly depending on $\ep,l$), on which a contraction relation holds. 

\smallskip

For $i\in \{1,2\}$, let $\yti^{(i)} \in \cacti_{1,\ga}^{0,\ga}([0,l\ep +\eta])$ such that $\yti^{(i)}_{|[0,l\ep]}=\yti^l$, $\cn[\yti^{(i)}; \cacti_{1,\ga}^{0,\ga}([0,l\ep+\eta])]\leq N_{l+1}$, and denote $\zti^{(i)}=\Gamma(\yti^{(i)})$, where $\Gamma$ is defined as in Step 1, but restricted to $\cacti_{1,\ga}^{0,\ga}([0,l\ep +\eta])$. According to Theorem \ref{th:young}, 
\begin{multline*}
\cn[\ztiun-\ztide; \cacti_{1,\ga}^{\ga}([l\ep,l\ep+\eta]) \leq c_{\ga,x} \big\{ \cn[\si(\yun)-\si(\yde); \cac_1^0([l\ep,l\ep+\eta])]\\
+\eta^\ga \cn[\si(\yun)-\si(\yde); \cac_1^\ga([l\ep,l\ep+\eta])]\big\},
\end{multline*}
which, together with Lemma \ref{lem:sigma}, implies
$$\cn[\ztiun-\ztide; \cacti_{1,\ga}^{\ga}([l\ep,l\ep+\eta])] \leq c^2_{x,\si} \lcl 1+N_{l+1}\rcl  \eta^\ga \cn[\ytiun -\ytide; \cacti_{1,\ga}^{0,\ga}([l\ep,l\ep+\eta])]. $$

Since the processes $\ytiun-\ytide$, $\ztiun-\ztide$ vanish on $[0,l\ep]$, we can more simply write 
$$\cn[\ztiun-\ztide; \cacti_{1,\ga}^{\ga}([0,l\ep+\eta])] \leq c^2_{x,\si} \lcl 1+N_{l+1}\rcl  \eta^\ga \cn[\ytiun -\ytide; \cacti_{1,\ga}^{\ga}([0,l\ep+\eta])]. $$
Besides, $(\ztiun-\ztide)_s=\delti(\ztiun-\ztide)_{s, l\ep}$, so that $\cn[\ztiun-\ztide; \cacti_{1,\ga}^0([0,l\ep+\eta])] \leq \cn[\ztiun-\ztide; \cacti_{1,\ga}^{0,\ga}([0,l\ep+\eta])] \eta^\ga$. Finally, we get
$$\cn[\ztiun-\ztide; \cacti_{1,\ga}^{0,\ga}([0,l\ep+\eta])] \leq c^2_{x,\si} \lcl 1+N_{l+1}\rcl ( 1+T^\ga )  \eta^\ga \cn[\ytiun -\ytide; \cacti_{1,\ga}^{0,\ga}([0,l\ep+\eta])]. $$
Fix then $\eta=\inf \lp \ep, (2c^2_{x,\si} \lcl 1+N_{l+1}\rcl ( 1+T^\ga ) )^{-1/\ga}\rp$. In this case, $\Gamma$ becomes a strict contraction on the set
$$\{ \yti \in \cacti_{1,\ga}^{0,\ga}([0,l\ep+\eta]): \ \yti_{|[0,l\ep]}=\yti^l, \ \cn[\yti; \cacti_{1,\ga}^{0,\ga}([0,l\ep+\eta])] \leq N_{l+1} \}.$$
Using the stability of $\cq_{\yti^l,(l+1)\ep}^{N_{l+1}}$, we can easily show that the latter set is invariant by $\Gamma$ too (cf Lemma \ref{lem:stability} below). Consequently, there exists a unique fixed point in this set, which we denote by $\yti^{l,\eta}$. Since $\eta$ does not depend on $\yti^l$, the same calculation then remains true on the (invariant) set 
$$\{\yti \in \cacti_{1,\ga}^{0,\ga}([0,l\ep+2\eta]): \ \yti_{|[0,l\ep+\eta]}=\yti^{l,\eta}, \ \cn[\yti; \cacti_{1,\ga}^{0,\ga}([0,l\ep+2\eta])] \leq N_{l+1} \}.$$
Thus, $\yti^{l,\eta}$ can be extended in a solution $\yti^{l,2\eta}$ defined on $[0,l\ep+2\eta]$ and proceeding so until the whole interval $[l\ep,(l+1)\ep]$ is covered, we get the expected extension $\yti^{l+1}$. 

\end{proof}

\begin{lemma}\label{lem:stability}
With the notations of the preceding proof, the set
$$\{ \yti \in \cacti_{1,\ga}^{0,\ga}([0,l\ep+\eta]): \ \yti_{|[0,l\ep]}=\yti^l, \ \cn[\yti; \cacti_{1,\ga}^{0,\ga}([0,l\ep+\eta])] \leq N_{l+1} \}$$
is invariant by $\Gamma$.
\end{lemma}

\begin{proof}
Let $\yti$ an element of this set and $\zti=\Gamma(\yti)$. Set
$$\hat{y}_t=\begin{cases}
\yti_t & \mbox{if} \ t\leq l\ep+\eta\\
e^{-.(t-(l\ep+\eta))}\yti_{l\ep+\eta} & \mbox{if} \ t\in [l\ep+\eta,(l+1)\ep].
\end{cases}
$$
Then $\hat{y}$ is easily shown to be continuous, that is $\hat{y} \in \cacti_{1,\ga}([0,(l+1)\ep])$. Moreover, if $s,t \in [l\ep+\eta,(l+1)\ep]$, $(\delti \yti)_{ts}=0$, whereas if $s \leq l\ep+\eta \leq t$, $(\delti \hat{y})_{ts}=e^{-.(t-(l\ep+\eta))}(\delti \yti)_{l\ep+\eta,s}$, so that $\cn[\hat{y}; \cacti_{1,\ga}^\ga([0,(l+1)\ep])] \leq \cn[\yti; \cacti_{1,\ga}^\ga([0,l\ep+\eta])]$. Since $\cn[\hat{y}; \cacti_{1,\ga}^0([0,(l+1)\ep])] \leq \cn[\yti; \cacti_{1,\ga}^0([0,l\ep+\eta])]$, we deduce $\cn[\hat{y}; \cacti_{1,\ga}^{0,\ga}([0,(l+1)\ep])] \leq \cn[\yti; \cacti_{1,\ga}^{0,\ga}([0,l\ep+\eta])]\leq N_{l+1}$, which means that $\hat{y} \in \cq_{\yti^l,(l+1)\ep}^{N_{l+1}}$. But we know from the first step of the preceding proof that $\cq_{\yti^l,(l+1)\ep}^{N_{l+1}}$ is invariant by $\Gamma$, and so, if $\hat{z}=\Gamma(\hat{y})$, $\cn[\hat{z}; \cacti_{1,\ga}^{0,\ga}([0,(l+1)\ep])] \leq N_{l+1}$. It is now clear that $\zti=\hat{z}_{|[0,l\ep+\eta]}$, which finally leads to $\cn[\zti; \cacti_{1,\ga}^{0,\ga}([0,l\ep+\eta])] \leq \cn[\hat{z}; \cacti_{1,\ga}^{0,\ga}([0,(l+1)\ep])] \leq N_{l+1}$.

\end{proof}

\subsection{Application to fBm} \label{subsec:application}
We now aim at proving that the previous results can be applied to a fractional brownian motion $X=(X^{(1)},\ldots,X^{(n)})$ with Hurst parameter $H> 1/2$. Before we start with this program, let us recall what we mean by fBm in this paper (we refer to \cite{Nu-cours} for further details on this process): for computational sake for the case $1/3<H<1/2$, we will consider $X$ as a centered Gaussian process indexed by $\R$ (even if our equation is indexed by $\ott$), with covariance 
$$R_H(t,s)_{i,j}=E(X^{(i)}_s X^{(j)}_t)=\frac12 \delta_{i,j} (|s|^{2H}+|t|^{2H}-|t-s|^{2H}), \quad s,t\in\R.$$
We assume that the underlying probability space $(\Omega,{\mathcal F},P)$ on which $X$ is defined  is such
that $\Omega$ is  the Banach space of all the continuous funtions
 $C_0(\R;\R^n)$, which vanish at time $0$, endowed with the supremum norm on compact sets. $P$ is the only  probability measure such that the canonical process
$\{X_t; t\in\R\}$ is a $n$-dimensional fBm with parameter $H$ and
the  $\sigma$-algebra ${\mathcal F}$ is the completion of the Borel 
$\sigma$-algebra of $\Omega$ with respect to $P$.

\smallskip

In order to apply our general results to the fBm, we need to define Wiener integrals with respect to $X$. To this purpose, denote by $\ch$ the completion of the $\R^d$-valued step functions $\ce$ with respect to the inner product
$$\left\langle (\textbf{1}_{[0,t_1]},\ldots,\textbf{1}_{[0,t_n]}),
(\textbf{1}_{[0,s_1]},\dots,\textbf{1}_{[0,s_n]})\right\rangle=\sum_{i=1}^{n} R_H(s_i,t_i), \quad s_i,t_i\in\R.
$$
When $H>1/2$, it can be checked that this inner product can be expressed as:
\beq\label{eq:inner-pdt-H}
\lla f,\, g\rra_{\ch}=c_H \sum_{i=1}^{n} \int_\R \int_\R f^{(i)}_u \, g^{(i)}_v \, |u-v|^{2H-2}  \, du dv,
\quad\mbox{with}\quad
c_H=H(2H-1).
\eeq
for all $f,g\in\ch$. It can then be shown that the family of Wiener integrals $\{X(h);\, h\in\ch\}$ with respect to $X$ forms an isonormal Gaussian process, with $E[X(h_1)\, X(h_2)]=\langle h_1,\, h_2\rangle_{\ch}$. 
\smallskip

With these notations and facts in hand, a natural definition of $\Xti^1$ is as follows: for $\xi,s,t\in\R_+$ and $i\le n$, set
\begin{equation}\label{appli:x-tilde-1}
\Xti^{1,(i)}_{ts}(\xi)=\int_s^t e^{-\xi(t-v)}dX_v^{(i)}:= X(h(t,s;\xi)),
\quad\mbox{with}\quad
h_v(t,s;\xi)= e^{-\xi(t-v)} \textbf{1}_{[s,t]}(v) \, e_i,
\end{equation}
where $e_i$ denotes the $i\textsuperscript{th}$ vector of the canonical basis in $\R^n$. We will now show that this process satisfies Hypothesis \ref{hyp:X1}, under some integrability assumptions on $\hphi$. First, it is readily checked that $\delti \Xti^1 =0$. In order to prove that $\Xti^1 \in \cacti_{1,\ga}^\ga$ a.s, we shall use a Garsia-Rodemich-Rumsey (GRR in short in the sequel) type result, which is an extension of the original paper \cite{GRR} in 3 directions: (i) Like in \cite{Gu}, we will get a regularity result for a general function $R$ defined on $\cs_2$, which is not necessarily the increment of a function $f\in\cac_1$. (ii) The conditions on $R$ involve $\delti R$ instead of $\delta R$ (iii) $R$ also depends on the Laplace variable $\xi$. It should be noticed at this point that, for the remainder of the section, $\cs_2$ stands for $\cs_2(\ott)$.

\begin{proposition}\label{prop:g-r-r}
Let $(V,\norm{.})$ a Banach space and fix $\xi \geq 0$. Let $\Rti: \cs_2 \times \mathbb{R}_+ \rightarrow V$ such that $\Rti_{..}(\xi) \in \cac_2(V)$ and define
$$\Uti(\xi)=\iint_{0<v<w<T} \psi\lp \frac{\norm{\Rti_{wv}(\xi)}}{\phi(\lln w-v\rrn )}\rp dvdw,$$
where $\psi, \phi: \mathbb{R}^+\rightarrow \mathbb{R}^+$ are strictly increasing functions and $\phi(0)=0$.
Assume now that there exists some $\Cti(\xi)\geq 0$ such that, for all $s<t$ in $\ott$, 
\beq\label{eq:27}
\sup_{s \leq u \leq t} \norm{(\delti \Rti)_{tus}(\xi)} \leq \psi^{-1}\lp \frac{4\, \Cti(\xi)}{\lln t-s\rrn^2}\rp \phi(t-s).
\eeq
Then, for all $0\leq s\leq t \leq T$,
$$\norm{\Rti_{ts}(\xi)} \leq 8 \int_0^{\lln t-s \rrn}  \psi^{-1}\lp \frac{4\, \Uti(\xi)}{r^2}\rp d\phi(r)+9\int_0^{\lln t-s\rrn}\psi^{-1}\lp \frac{4\, \Cti(\xi)}{r^2}\rp  d\phi(r).$$
\end{proposition}  

\begin{proof}
See Appendix.

\end{proof}

A preliminary step, before we prove the desired continuity result for $\Xti^1$, is to show that this process is at least an element of $\cac_2(\R^{1,d})$ for any fixed $\xi$. This is achieved in the following lemma:
\begin{lemma}\label{lem:c-2-delta}
The process $\Xti^1$ defined by formula (\ref{appli:x-tilde-1}) admits a modification $\Xti^{1,\ast}$ such that, almost surely, $\Xti^{1,\ast}_{..}(\xi) \in \cac_2([0,T])$ for any $\xi \geq 0$.
\end{lemma}

\begin{proof}

Let us lean on the following version of the Kolmogorov continuity criterion: consider a process $\{ \tilde{Z}_{ts}(\xi), \, s,t \in [0,T], \, \xi \in \R^+\}$ living in any of the Wiener chaos associated to $X$, and assume that for all $N \in \N$ and all $(s_1,t_1,\xi_1),(s_2,t_2,\xi_2) \in [0,T]^2 \times [0,N]$,
\begin{equation} \label{kolmogorov}
E\lc \lVert \tilde{Z}_{t_1s_1}(\xi_1)-\tilde{Z}_{t_2s_2}(\xi_2) \rVert^2 \rc \leq c_N \lcl \lln s_1-s_2 \rrn^{\al_1}+\lln t_1-t_2 \rrn^{\al_2}+\lln \xi_1-\xi_2 \rrn^{\al_3} \rcl, 
\end{equation}
for some $\al_1,\al_2,\al_3 >0$. Then $\tilde{Z}$ admits a modication $\tilde{Z}^\ast$ such that, almost surely, $\tilde{Z}^\ast_{..}(\xi)$ is continuous for any $\xi \geq 0$.

\smallskip

To show that $\Xti^1$ actually satisfies (\ref{kolmogorov}), suppose for instance $s_2 <s_1 < t_2 <t_1$. Then
\begin{multline*}
\Xti^1_{t_1s_1}(\xi_1)-\Xti_{t_2s_2}^1(\xi_2)\\
=\int_{s_1}^{t_1} [e^{-\xi_1(t_1-u)}-e^{-\xi_2(t_2-u)}] \, dX_u +\int_{t_2}^{t_1} e^{-\xi_1(t_1-u)} dX_u-\int_{s_2}^{s_1} e^{-\xi_2(t_2-u)} dX_u. 
\end{multline*}
But, on the one hand, relation (\ref{eq:inner-pdt-H}) yields
\begin{multline*}
E\lc \norm{ \int_{t_2}^{t_1} e^{-\xi_1(t_1-u)} dX_u }^2 \rc = c_H \int_{t_2}^{t_1} du \int_{t_2}^{t_1} dv \, e^{-\xi_1(t_1-u)} e^{-\xi_1(t_1-v)} \lln u-v \rrn^{2H-2}\\ \leq c \lln t_1-t_2 \rrn^{2H}
\end{multline*}
and likewise 
\begin{equation*}
E \lc \norm{ \int_{s_2}^{s_1} e^{-\xi_2(t_2-u)}dX_u }^2 \rc \leq c \lln s_1-s_2 \rrn^{2H}.
\end{equation*} 
On the other hand, still evoking relation (\ref{eq:inner-pdt-H}), we have
\begin{multline*}
E \lc \norm{ \int_{s_1}^{t_2} [e^{-\xi_1(t_1-u)}-e^{-\xi_2(t_2-u)}] \, dX_u }^2 \rc \\
\leq c \int_{s_1}^{t_1} du \int_{s_1}^{t_1} dv \lln e^{-\xi_1(t_1-u)}-e^{-\xi_2(t_2-u)} \rrn  \lln u-v \rrn^{2H-2} \lln e^{-\xi_1(t_1-v)}-e^{-\xi_2(t_2-v)} \rrn,
\end{multline*}
which, together with the estimation
\bean
\lln e^{-\xi_1(t_1-u)}-e^{-\xi_2(t_2-u)} \rrn & \leq & \lln e^{-\xi_1(t_1-u)}-e^{-\xi_1(t_2-u)} \rrn +\lln e^{-\xi_1(t_2-u)}-e^{-\xi_2(t_2-u)} \rrn \\
&\leq & \xi_1 \lln t_1-t_2 \rrn +\lln t_2-u \rrn \lln \xi_1-\xi_2 \rrn \\
& \leq & N \lln t_1-t_2 \rrn+T \lln \xi_1-\xi_2 \rrn,
\eean
leads to
\begin{equation*}
E \lc \norm{ \int_{s_1}^{t_2} [e^{-\xi_1(t_1-u)}-e^{-\xi_2(t_2-u)}] \, dX_u }^2 \rc \leq c_N \lcl \lln t_1-t_2\rrn+\lln \xi_1-\xi_2 \rrn \rcl^2.
\end{equation*}

\end{proof}

We are thus in position to apply our general results to the fBm case:
\begin{theorem}\label{thm:3.14}
Let $X$ be a fBm with Hurst parameter $H > 1/2$ and $\ga \in (1/2,H)$ such that $\int_0^\infty d\xi \, |\hphi(\xi)|(1+\xi^\ga) < \infty$. Then the process $\Xti^1$ defined by the Wiener integral (\ref{appli:x-tilde-1}) satisfies Hypothesis \ref{hyp:X1} for $\ga$ a.s. Consequently, if in addition, $\si \in \cac^{2,b}$, the system
$$\begin{cases}
(\delti \Yti)_{ts}=\cj_{ts}\lp\dti X \, \si\lp a+\int_0^\infty d\xi \, \hphi(\xi) \Yti(\xi)\rp \rp\\
\Yti_0 =0
\end{cases} $$
admits a unique solution in $\cacti_{1,\ga}^{0,\ga}$ a.s.
\end{theorem}

\begin{proof}
As mentioned before, we just need to check that $\Xti^1\in\cac_{2,\ga}^{\ga}(\R^{1,d})$. Furthermore, with Lemma \ref{lem:c-2-delta} in hand, and thanks to the fact that $\delti\Xti^1=0$, we can apply Proposition \ref{prop:g-r-r} to $\Xti^1$, with $\psi(x)=x^{2p}$ and $\phi(x)=x^{\ga+1/p}$, to obtain
$\norm{\Xti^1_{ts}(\xi)} \leq c \lln t-s \rrn^\ga (\Uti_{\ga,2p}(\xi))^{1/2p}$, where $p$ is an arbitrary strictly positive number and
\begin{eqnarray} \label{defiuti}
\Uti_{\ga,2p}(\xi)=\int_{T\times T} \frac{\norm{ \Xti^1_{wv}(\xi)}^{2p}}{|w-v|^{2\ga p+2}} dvdw,
\end{eqnarray}
Therefore, we just have to prove that $\cn[\Uti_{\ga,2p}^{1/2p};\cl_\ga]=\int_0^\infty d\xi \, \hphi(\xi)(1+\xi^\ga) (\Uti_{\ga,2p}(\xi))^{1/2p} < \infty $ a.s, since in this case, $\Xti^1 \in \cacti_{2,\ga}$ and $\cn[\Xti^1; \cacti_{2,\ga}^\ga] \leq c \, \cn[\Uti_{\ga,2p}^{1/2p}; \cl_\ga] < \infty$ a.s. In fact, we are going to show that $E[\cn[\Uti_{\ga,2p}^{1/2p};\cl_\ga]] <\infty$.

\smallskip 

Let us start with the Jensen inequality 
$$E[\cn[\Uti_{\ga,2p}^{1/2p};\cl_{\ga}]] \leq \int_0^\infty d\xi |\hphi(\xi)| (1+\xi^\ga) E[\Uti_{\ga,2p}(\xi)]^{1/2p}.$$
Notice then that, as we already mentioned in the proof of Lemma \ref{lem:c-2-delta}, 
\begin{equation}\label{isom-young}
E[\norm{\Xti^1_{wv}(\xi)}^{2p}] \leq c \lln w-v\rrn^{2Hp},
\end{equation}
which leads to
$E[\Uti_{\ga,2p}(\xi)] \leq c \int_{T\times T} \lln w-v\rrn^{2Hp-2\ga p+2}dwdv.$
This means that if we take $\ga \in (1/2, H)$ and $p>1/(H-\ga)$, $E[\Uti_{\ga,2p}(\xi)] \leq M$ for some constant $M$ independent of $\xi$, and as a consequence,
$$E[\cn[\Uti_{\ga,2p}^{1/2p};\cl_\ga]] \leq M^{1/2p} \int_0^\infty d\xi \, |\hphi(\xi)|(1+\xi^\ga).$$
The proof is now easily finished.

\end{proof}

%%%%%%%%%%%%%%%%%%%%%%%%%%%%%%%%%%%%%%%%%%%
%%%%%%%%%%%%%%%%%%%%%%%%%%%%%%%%%%%%%%%%%%%
\section{Rough Volterra equations}\label{sec:rough-volterra}

Our aim is still to solve equation (\ref{eq:volterra-system2}) or (\ref{eq:volterra-ytilde-bis}), in a case where $x$ satisfies Hypothesis~\ref{hyp:X1}, but where we replace the condition $\ga>1/2$ by $\ga>1/3$ only. Like in the Young case, our first task is thus to give a suitable interpretation to the integral in (\ref{eq:volterra-ytilde-bis}), which goes beyond the Young case. We will see that the key to this improvement is to introduce a new class of processes.

\subsection{Convolutional controlled paths}\label{sec:conv-ctrl-path}
As in the Young case, let us start with some heuristic considerations: let us go back for a moment to the Volterra equation under the form 
$$
y_t=a+\iot \phi(t-v) \, dx_v \si(y_v),
$$
and assume that $x$ is a smooth path in $\cac_1^{\infty}$ and $\si,\phi$ are regular coefficients. Then the equation above admits a unique solution $y$, whose increments can be decomposed into:
$$
(\der y)_{ts}=\ist \phi(t-v) \, dx_v \, \si(y_v) +\int_0^s [\phi(t-v)-\phi(t-v)]\, dx_v \, \si(y_v)
= x^1_{ts} \, \si(y_s) + r_{ts},
$$
with
\bean
x^1_{ts}&=&\ist \phi(t-v) \, dx_v  \\
r_{ts}&=&\ist \phi(t-v) \, dx_v \, (\der\si(y))_{vs}+\int_0^s [\phi(t-v)-\phi(t-v)]\, dx_v \, \si(y_v):= r_{ts}^1+r_{ts}^2.
\eean
This is exactly the structure which will be imposed for the solution to our equation, and let us analyze it a little further: if we assume now that $x$ has only a regularity of the form $\cac_1^\ga$ with $\ga>1/3$ and that $y$ is $\cac_1^{\ka}$ for any $\ka<\ga$, then we expect $x^1$ to be an element of $\cac_2^\ga$, under some regularity conditions on $\phi$ (which will in fact be assumed to be a differentiable kernel). As far as the remainder term $r$ is concerned, we expect it to inherit the Hölder regularity of $y$ and $x$ for $r^1$, and the regularity of $\phi$ for $r^2$. Hence, the remainder term $r$ should be an element of $\cac_2^{2\ka}$. It is also worth recalling from Remark~\ref{rmk:X1} that, if $x$ is a path allowing to apply Fubini's theorem, then $x^1$ should satisfy:
\beq
x^1_{ts}=\int_0^\infty \hxu_{ts}(\xi) \, \hph(\xi) \, d\xi.
\eeq
It is thus natural to formulate the following assumption on our driving process $x$:
\begin{hypothesis}\label{hyp:X1-bis}
Assume that, for some $\ga \in (1/3,1/2)$ and $\be >0$, $x$ allows to define a process $\xti^1 \in \cacti_{2,\be}^\ga(\R^{1,n})$ such that $\delti \xti^1=0$. Set then
$$x^1_{ts}=\int_0^\infty \xti_{ts}^1(\xi)\hphi(\xi) d\xi.$$
\end{hypothesis}

\begin{remark}\label{rmk:4.1}
Notice that, contrary to Hypothesis \ref{hyp:X1}, the index $\be$ and the exponent $\ga$ may be different here. In fact, for some computational reasons that will arise in the proof of Theorem \ref{main-theorem-rough}, we shall be prompted to take $\be=1$. Therefore, from now on, let us only focus on the spaces $\cacti_{i,1}^\ga(\R^{1,n})$ ($i\in \{ 1,2,3\}$), that we more simply denote by $\cacti_i^\ga(\R^{1,n})$.
\end{remark}

Fix an interval $I=[a,b] \subset [0,T]$ and denote $\ep=|I|=b-a$. With the above considerations in mind, the natural spaces to work with in order to solve equation (\ref{eq:volterra-system2}) can be defined as follows:
\begin{definition}
A path $y\in\cac_1(I;\mathbb{R}^k)$ is said to be a convolutional process controlled by $x^1$ (with regularity $\ka$) if $\der y$ can be decomposed into:
\beq\label{eq:weak-X1-dcp}
 (\der y)_{ts}= (x^1_{ts} \, \zeta_s)^\ast + r_{ts},
 \quad\mbox{with}\quad
 \zeta\in\cac_1^{\ka}(I;\R^{n,k}), \mbox{ and } r\in \cac_2^{2\ka}(I;\mathbb{R}^k).
\eeq
Denote the space of such controlled paths by $\cq^\ka(I;\R^k)$ and for any $h \in \R^k$, write $\cq_h^\ka(I;\R^k)$ $=\{ y \in \cq(I;\R^k): \ y_a=h\}$ . Then the norm associated to $\cq^\ka(I;\R^k)$ is
$$
\cn [y;\cq^\ka]=\cn [y;\cac_1^{\ka}]+\cn [\zeta;\cac_1^{0}]+\cn [\zeta;\cac_1^{\ka}]
+\cn [r;\cac_2^{2\ka}].
$$ 
Notice that if $1/3 < \ka < \ga$, $\cq^\ka \subset \cac_1^\ga$ and 
$$\cn[y; \cac_1^\ka(I)] \leq c_x \ep^{\ga-\ka} \, \cn[y; \cq^\ka(I)].$$
\end{definition}

In fact, as in the Young case, we shall focus on the form (\ref{eq:volterra-ytilde}) of the original equation which also involves a process $\yti$ indexed by the Laplace variable $\xi$. In this setting, the same reasoning as above applied to $\yti$ leads to the introduction of the following spaces:
\begin{definition}\label{def:controlled-path-q}
A path $\yti \in \cacti_{1}(I;\mathbb{R}^k)$ is said to be a process controlled by $\hxu$ (with regularity $\ka$) if $\delti \yti$ can be decomposed into:
$$
 (\delti \yti)_{ts}=(\xti_{ts}^1 \, \zeta_s)^\ast+\rti_{ts},
 \quad\mbox{with}\quad
 \zeta\in\cac_1^{\ka}(I;\R^{n,k}), \mbox{ and } \rti\in \cacti_{2}^{2\ka}(I;\mathbb{R}^k),
$$
where we recall that we have set $\cacti_{i}^{\nu}:=\cacti_{i,1}^{\nu}$ according to Remark \ref{rmk:4.1}. Denote this second space of controlled paths by $\Qti^\ka(I;\R^k)$ and for any $\hti \in \cl_1$, write $\Qti_{\hti}^\ka(I;\R^k) =\{\yti \in \Qti^\ka(I;\R^k): \ \yti_a=\hti \}$. Then the norm associated to $\Qti^\ka(I;\R^k)$ is
$$
\cn[\yti; \Qti^\ka]=\cn[\yti;\cacti_{1}^\ka]+\cn[\zeta; \cac_1^{0}]+\cn[\zeta;\cac_1^\ka]+\cn[\rti; \cacti_{2}^{2\ka}].$$ 
\end{definition}

\smallskip

It is then readily shown that the space of controlled processes is stable by composition with a smooth enough function:
\begin{proposition}\label{cp:weak-phi}
Let $z\in\cq^\ka(I;\R^k)$ with decomposition (\ref{eq:weak-X1-dcp}),
$\si \in   C^{2,b}(\R^k;\R^l)$ and set $\hz=\si(z)$.
Then $\hz\in\cq^\ka(I;\R^l)$, and it can be decomposed into
$$
\der \hz= (x^1 \, \hat\zeta)^\ast  +\hat r,
$$
with 
$$
\hat\zeta_s= \zeta_s( D \si(z_s))^\ast
,\quad
\hat r_{ts}=   D \si(z_s)\, r_{ts}  + \lc \der(\si(z))_{ts}- D \si(z_s)(\der z)_{ts} \rc, 
$$
where $ D \si$ stands for the matrix-valued coefficient $( \frac{\partial \si^i}{\partial x^j} )_{1\leq i\leq l,1\leq j\leq k}$, and the norm of $\hz$ as a convolutional controlled process can be bounded as:
\begin{equation}\label{eqn:deux-etape}
\cn[\hz;\cq^\ka(I;\R^l)]\le
c_{\si}\lcl 1+\cn[z;\cq^\ka(I;\R^k)]^2  \rcl.
\end{equation}
Furthermore, if $\zun, \zde \in \cq^\ka(I)$ are such that $\zun_a=\zde_a$, then
\begin{equation}\label{deux-etape-bis}
\cn[\si(\zun)-\si(\zde);\cq^\ka(I;\R^l)] \leq c_{\si,\zun,\zde} \,  \cn[\zun-\zde; \cq^\ka(I;\R^k)],
\end{equation}
where
$$c_{\si,\zun,\zde} \leq c_\si \lcl 1+\cn[\zun; \cq^\ka(I;\R^k)]+\cn[\zde; \cq^\ka(I;\R^k)] \rcl^2.$$
\end{proposition}

\begin{proof}
It is exactly the same as the proof of \cite[Proposition 4]{Gu}, replacing $\der X$ with $x^1$.

\end{proof}

Finally, let us mention that, in the remainder of the article, we will write $\cq^\ka(I)$ and $\Qti_\be^\ka(I)$ instead of $\cq^\ka(I;V)$ and $\Qti_\be^\ka(I;V)$ whenever this does not lead to an ambiguous situation.

\subsection{Integration of controlled processes}
We now aim at giving a precise sense to the integral $\cj_{ts}(\dti x \, \si(y))$ which appears in (\ref{eq:volterra-ytilde-bis}) and stands for $\int_s^t e^{-.(t-u)} dx_u \, \si(y_u)$ in case of smooth processes. As $\ga < 1/2$, we can no longer resort to Young's interpretation. In fact, in order to define this integral, we will rely, as usual in the rough path theory, on the a priori existence of  some Levy area type process adapted to our problem (notice that the following hypothesis covers Hypothesis \ref{hyp:X1-bis}):

\begin{hypothesis}\label{hyp:X2}
Assume that, for some $\ga \in (1/3,1/2)$, $x$ allows to define three processes $\xti^1 \in \cacti_{1}^\ga(\R^{1,n})$, $\xti^2 \in \cacti_{2}^{2\ga}(\R^{n,n})$ and $\xti^3 \in \cacti_{3}^{3\ga}(\R^{n,n})$ satisfying $\delti \xti^1 =0$ and
$$(\delti \xti^2)_{tus}=\xti^1_{tu}\otimes  x^1_{us}+\xti^3_{ts}.$$
\end{hypothesis}

\begin{remark}\label{rmk:4.5}
In case of a smooth process $x$, the increment $\hxd$ represents now the double iterated integral $\ist e^{-\xi(t-v)} dx_v \otimes x^1_{vs}$, which can also be written, with a slight adaptation of the notations of Section \ref{sect:convol-increments}, as the (partially) integrated Levy area
$$\xti_{ts}^2(\xi)=\int_0^\infty d\eta \, \hphi(\eta) \cj_{ts}(\dti x \otimes \dti x)(\xi,\eta).$$
As for $\xti^3$, it is given in this case by 
$$
\xti^3_{tus}(\xi)=\int_u^t e^{-\xi(t-v)} dx_v  \otimes (\der x^1)_{vus}
=\int_u^t e^{-\xi(t-v)} dx_v \otimes \int_s^u [\phi(v-w)-\phi(u-w)] dx_w.
$$
\end{remark}

\smallskip

The last ingredient we need before we can  integrate convolutional controlled processes with respect to the increment $\dti x$ is a matrix equivalent of Lemma \ref{lem:delta-prod-ML}: if $A,B \in \R^{k,l}$, denote $A \cdot B=\mbox{Tr}(AB^*)$. Obviously, $\lln A \cdot B \rrn \leq \norm{A}\norm{B}$, and hence:
\begin{lemma}\label{lem:prod-scal-matrices}
If $\Mti \in \cacti_{2}(\R^{k,l})$ and $L\in \cac_m(\R^{k,l})$, then $\Mti \cdot L \in \cacti_{m+1}(\R)$ and
$$\delti(\Mti \cdot L)=\delti \Mti \cdot L-\Mti \cdot \der L.$$
\end{lemma}

\smallskip

Here is now the natural way to integrate convolutional controlled processes in our context:

\begin{proposition}\label{prop:intg-controlled-proc}
For two given coefficients $\ga,\ka$ such that $1/3<\ka<\ga$,
let $x$ be a process satisfying Hypothesis \ref{hyp:X2}. Furthermore,  let
$z\in \cq^\ka(I;\R^n)$ with decomposition
\begin{equation}\label{eq:dcp-m}
(\der z)_{ts}=(x_{ts}^1 \zeta_s)^* +r_{ts},
\quad\mbox{ where }\quad
\zeta\in\cac_1^\ka(I;\R^{n,n}), \, r\in\cac_2^{2\ka}(I;\R^n).
\end{equation}
Define $\Ati$ by $\Ati_0=\hti$ (where $\hti \in \cl_1$) and
\begin{equation}\label{eq:dcp-mdx}
(\delti \Ati)_{ts} =\xti^1 z+\xti^2 \cdot\zeta^\ast+\Lati(\xti^1 r+\xti^2 \cdot (\der \zeta)^\ast-\xti^3 \cdot\zeta^\ast).
\end{equation}
Finally, set
 \begin{equation*} 
 \cj(\dti x \, z) = \delti \Ati. 
 \end{equation*}
Then:
\begin{enumerate}
\item
$\Ati$ is well-defined as an element of $\Qti^\ka(I;\R)$, and $\cj_{ts}(\dti x \, z)(\xi)$ coincides with the integral $\ist e^{-\xi(t-v)} dx_v \, z_v$ in case of two smooth functions $x$ and $z$.
\item
The semi-norm of $\Ati$ in $\Qti^\ka(I;\R)$ can be estimated as
\begin{equation}\label{eqn:trois-etape}
\cn[\Ati; \Qti^\ka(I;\R)]\le
c_{x}
\lcl (\cn[z; \cac_1^{0}(I;\R^n)] + \ep^{\ga-\ka}\cn[z; \cq^\ka(I;\R^n)]\rcl,
\end{equation}
for a positive constant $c_{x}$ depending only on $x$. 
\item
It holds
\begin{equation}\label{eq:rsums-imdx}
 \cj_{ts}(\dti x \, z)
=\lim_{|\Pi_{ts}|\to 0}\sum_{i=0}^n
\lc \hxu_{t_{i+1}, t_{i}} \, z_{t_{i}}
+ \hxd_{t_{i+1}, t_{i}} \cdot \zeta_{t_{i}}^\ast \rc \quad \mbox{in} \ \cl_1,
\end{equation} for any $\ell_1\le s<t\le \ell_2$,
where the limit is over all partitions $\Pi_{ts} = \{t_0=t,\dots,
t_n=s\}$ of $[s,t]$ as the mesh of the partition goes to zero.
\end{enumerate}
\end{proposition}

\begin{remark}
It is certainly possible to state and prove continuity results for our extended integral in terms of a sequence $x^n$ converging to $x$ in the sense of convolutional controlled processes. We did not go into these considerations for sake of conciseness.
\end{remark}

\begin{proof}[Proof of Proposition \ref{prop:intg-controlled-proc}]
{\it (1)} If $z,x$ are two smooth functions, then $\cj(\hddx x \, z)$ can be defined as a Riemann integral, and as in the Young case, one can write:
$$
\cj_{ts}(\hddx x \, z)(\xi)=\ist e^{-\xi(t-v)} dx_v \, z_v= \hxu(\xi) \, z_s + 
\ist e^{-\xi(t-v)} dx_v \, (\der z)_{vs}.
$$
Plugging the decomposition (\ref{eq:dcp-m}) for $(\der z)_{vs}$ into this last expression, and observing that, thanks to some elementary matrix manipulations, we have  $dx_v \, (x_{vs}^1 \, \zeta_s)^\ast= (dx_v \otimes  x_{vs}^1)   \cdot  \zeta_s^\ast$, we end up with:
\bean
\cj_{ts}(\hddx x \, z)(\xi)&=&\hxu(\xi) \, z_s +
\ist e^{-\xi(t-v)} dx_v \, \lc (x_{vs}^1 \, \zeta_s)^\ast +r_{vs} \rc  \\
&=&\hxu_{ts}(\xi) \, z_s + \lc \ist e^{-\xi(t-v)} dx_v \otimes  x_{vs}^1 \rc  \cdot  \zeta_s^\ast
+ \ist e^{-\xi(t-v)} dx_v \, r_{vs}  \\
&=&\hxu_{ts}(\xi) \, z_s + \hxd_{ts}(\xi) \cdot  \zeta_s^\ast + \ist e^{-\xi(t-v)} dx_v \, r_{vs},
\eean
or otherwise stated:
\beq\label{eq:rel-intg-hdx-r}
\cj_{ts}(\hddx x \, r)=\cj_{ts}(\hddx x \, z)
-\hxu _{ts} z_s - \hxd_{ts} \cdot  \zeta_s^\ast.
\eeq
In order to analyze the term $\cj_{ts}(\hddx x \, r)$, let us apply, like in the Young case, $\hdx$ to both members of the equality above. This gives, owing to Proposition~\ref{prop:dif-convol-intg}, Lemma~\ref{lem:prod-scal-matrices}, Hypothesis~\ref{hyp:X2}, and using the decomposition (\ref{eq:dcp-m}):
\bean
\delti\lp  \cj(\hddx x \, r) \rp&=& \hxu \, \der z - \hdx \xti^2 \cdot \zeta^\ast 
+ \xti^2 \cdot (\der \zeta)^\ast  \\
&=& \hxu \, r + \xti^2 \cdot (\der \zeta)^\ast - \xti^3 \cdot \zeta^\ast
\eean
When all these terms have a Hölder regularity greater than 1, we are now in a shape to apply the operator $\Lati$, which gives:
$$
\cj(\hddx x \, r)=\Lati \lp  \hxu \, r +\xti^2 \cdot (\der \zeta)^\ast - \xti^3 \cdot \zeta^\ast \rp.
$$
Plugging this equality back into (\ref{eq:rel-intg-hdx-r}), we have proved the relation
$$
\cj(\hddx x \, z)=\hxu \, z + \hxd \cdot \zeta^\ast 
+\Lati\lp  \hxu\, r + \hxd \cdot (\der\zeta)^\ast - \xti^3\cdot \zeta^\ast \rp,
$$
in case of some regular functions $x$ and $z$.

\smallskip

\noindent {\it (2)} 
Let us analyze the two terms of the remainder $\Rti$ of $\Ati$ defined by (\ref{eq:dcp-mdx}), namely:
$$\Rti= \xti^2 \zeta +\Lati(\hxu\, r + \hxd \cdot (\der\zeta)^\ast - \xti^3\cdot \zeta^\ast).$$
For the first term, we have
$$\cn[\xti^2 \cdot \zeta^\ast; \cacti_{2}^{2\ka}(I)] \leq \ep^{2(\ga-\ka)} \cn[\xti^2; \cacti_{2}^{2\ga}(I)] \, \cn[\zeta; \cac_1^{0}(I)].$$
As for the second term, we use the contraction property (\ref{ineqhla}) of $\Lati$ to deduce:
\begin{multline*}
\cn[\Lati(\hxu\, r + \hxd \cdot (\der\zeta)^\ast - \xti^3\cdot \zeta^\ast); \cacti_{2}^{2\ka}(I)]\\
 \leq  \ep^\ga \cn[\Lati(\hxu\, r + \hxd \cdot (\der\zeta)^\ast - \xti^3\cdot \zeta^\ast); \cacti_{2}^{2\ka+\ga}(I)]
 \leq  C  \ep^\ga \lp I+II+III \rp,
\end{multline*}
with
$$I=\cn[\xti^1  r; \cacti_{2}^{2\ka+\ga}(I)] \leq \cn[\xti^1;\cacti_{2}^\ga(I)] \, \cn[r; \cac_2^{2\ka}(I)],$$
$$II=\cn[\xti^2 \cdot (\der \zeta)^\ast; \cacti_{2}^{2\ka+\ga}(I)] \leq \ep^{\ga-\ka} \cn[\xti^2;\cacti_{2}^{2\ga}(I)] \, \cn[\zeta; \cac_1^\ka(I)],$$
$$III=\cn[\xti^3 \cdot \zeta^\ast; \cacti_{2}^{2\ka+\ga}(I)] \leq \ep^{2(\ga-\ka)} \cn[\xti^3; \cacti_{2}^{3\ga}(I)] \, \cn[\zeta; \cac_1^0(I)].$$
Thus, we get $\cn[\Rti; \cacti_{2}^{2\ka}(I)] \leq \ep^{2(\ga-\ka)} \, \cn[z; \cq^\ka(I)]$. Besides, we already mentioned that $\cn[z; \cac_1^\ka(I)] \leq c_x \ep^{\ga-\ka} \, \cn[z; \cq^\ka(I)]$. The estimation (\ref{eqn:trois-etape}) easily follows.

\smallskip

\noindent {\it (3)}  Remark that
$$\delti \Ati =(\id-\Lati \delti)(\xti^1 z+\xti^2 \cdot \zeta^\ast).$$
The result then stems from Corollary \ref{cor:integration-tilde}.

\end{proof}

In order to define the term $\cj(\dti x \, \si(y))$ in our Volterra equation, we need the following multidimensional version of the previous proposition:

\begin{definition}
We say that $z\in \cac_1(\R^{k,l})$ is controlled by $x^1$ (with regularity $\ka$) if $z=(z^{(1)}, \ldots, z^{(l)})$, with $z^{(i)}\in \cq^\ka(\R^{k})$. Denote $\cq^\ka(\R^{k,l})$ this set of controlled processes, and define, for any $z\in \cq^\ka(\R^{k,l})$,
$$\cn[z;\cq^\ka(\R^{k,l})]=\lp\sum_{i=1}^l \cn[z^{(i)};\cq^\ka(\R^k)]^2\rp^{1/2}.$$
Let us also introduce the set $\Qti^\ka(\R^{k,l})$ along the same principle, together with the norm
$$\cn[\zti;\Qti^\ka(\R^{k,l})]=\lp\sum_{i=1}^l \cn[\zti^{(i)};\Qti^\ka(\R^k)]^2\rp^{1/2}.$$
\end{definition}

\begin{corollary}
If $z \in \cq^\ka(\R^{n,l})$, the process $\cj(\dti x \, z)$ (with values in $\R^{l}$) defined by
$$\cj(\dti x \, z)^{(i)}=\cj(\dti x \, z^{(i)}) \quad \mbox{for all} \ i\in \{1,\ldots, l\},$$
belongs to $\Qti^\ka(\R^{l})$. Moreover, the conclusions of Proposition \ref{prop:intg-controlled-proc} still hold in this context. In particular,
\begin{multline*}
\cn[\cj(\dti x \, z); \Qti^\ka([\ell_1,\ell_2];\R^{l})]\\ \le
c_x
\lcl (\cn[z; \cac_1^{0}([\ell_1,\ell_2];\R^{n,l})] + \lln \ell_2-\ell_1\rrn^{\ga-\ka}\cn[z; \cq^\ka([\ell_1,\ell_2];\R^{n,l})]\rcl.
\end{multline*}
\end{corollary}

\subsection{Localized controlled processes}

In order to get a global solution for our rough Volterra system, we still have to perform a technical step. Indeed, like in the Young case, we will solve the equation by patching solutions defined on small intervals, and this patching procedure will involve a localization of some convolutional paths around a certain smooth increment $f$, which represents in general an initial condition. The current section is thus devoted to adapt our previous definitions and propositions to this localized setting. Notice that we assume, throughout the section, that $x$ satisfies Hypothesis \ref{hyp:X2}.

\smallskip

Fix thus an interval $I=[a,b]$ and denote $\ep=b-a$. The following subsets of $\cq^\ka(I;\R^k)$ will come into play:

\begin{definition}
Let $f \in \cac_2^1(I;\R^k)$. A process $y \in \cac_1^\ga(I;\R^k)$ is said to be $\ka$-weakly controlled around $f$ if 
\begin{equation}\label{decompo-controlled-around}
(\der y)_{ts}-f_{ts}=(x^1_{ts} \zeta_s^y)^\ast+r_{ts}^y, \ \mbox{with} \ \zeta^y \in \cac_1^\ka(I;\R^{n,k})\ \mbox{and} \ r^y \in \cac_2^{2\ka}(I;\R^k).
\end{equation}

\smallskip

Denote $\ca_{f,h}^\ka(I;\R^k)$ the set of $\ka$-weakly controlled around $f$ processes such that $z_a=h$, and for any $y \in \ca_{f,h}^\ka(I;\R^k)$, define its semi-norm by:
$$\cm[y;\ca_{f,h}^\ka(I;\R^k)]=\cn[\zeta^y;\cac_1^0(I)]+\cn[\zeta^y;\cac_1^\ka(I)]+\cn[r^y;\cac_2^{2\ka}(I)]+\cn[y;\cac_1^\ka(I)].$$ 

\end{definition}

The following elementary facts are worth noticing:
obviously, $\ca_{0,h}^\ka(I)=\cq_h^\ka(I)$, the norm $\cm[.;\ca_{0,h}^\ka(I;\R^k)]$ coincides with $\cn[.;\cq^\ka(I;\R^k)]$ and for any $f \in \cac_2^1(I;\R^k)$, $\ca_{f,h}^\ka(I) \subset \cq_h^\ka(I)$. The important point in our localization around $f$ is precisely that this latter increment does not play any role in the computation of $\cm[y;\ca_{f,h}^\ka(I;\R^k)]$ (thus the new notation $\cm$, instead of $\cn$, for the norm of $y$).

\smallskip

Let us now see how the spaces $\ca_{f,h}^\ka(I)$ pop out naturally when one integrates a convolutional controlled process.
\begin{proposition}\label{prop:passage-yti-y}
Let $\yti \in \Qti_{\hti}^\ka(I;\R^k)$ with decomposition $\delti \yti=(\xti^1 \zeta^{\yti})^\ast+\rti^{\yti}$. Set $y=a_0+\int_0^\infty d\xi \, \hphi(\xi) \, \yti(\xi)$. Then $y\in \ca_{f,h}^\ka(I;\R^k)$, with $f_{ts}=\int_0^\infty d\xi \, \hphi(\xi) \, a_{ts}(\xi) e^{-\xi(s-a)}\hti(\xi)$ and $h=a_0+\int_0^\infty d\xi \, \hphi(\xi) \, \hti(\xi)$. Moreover,
\begin{equation}\label{passage-yti-y}
\cm[y;\ca_{f,h}^\ka(I)] \leq c_x \lcl \cn[\yti;\Qti^\ka(I)]+\ep^{1-\ka} \cn[\hti;\cl_1] \rcl.
\end{equation}
\end{proposition}

\begin{proof}
If $s <t \in I$, write
\bean
(\der y)_{ts} &=& \int_0^\infty d\xi \, \hphi(\xi) (\der \yti)_{ts}(\xi)\\
&=& \int_0^\infty d\xi \, \hphi(\xi) (\delti \yti)_{ts}(\xi)+\int_0^\infty d\xi \, \hphi(\xi) a_{ts}(\xi) \yti_s(\xi)\\
&=& (x_{ts}^1 \zeta_s^{\yti})^\ast+\int_0^\infty d\xi \, \hphi(\xi) \rti_{ts}^{\yti}(\xi)+\int_0^\infty d\xi \, \hphi(\xi)a_{ts}(\xi) (\delti \yti)_{sa}(\xi)+f_{ts}.
\eean
Set $\zeta^y_s=\zeta^{\yti}_s$, $r^y_{ts}=\int_0^\infty d\xi \, \hphi(\xi) \lcl \rti^{\yti}_{ts}(\xi)+a_{ts}(\xi)(\delti \yti)_{sa}(\xi) \rcl $. Then
$$\cn[r^y;\cac_2^{2\ka}] \leq c \lcl \cn[\rti^{\yti};\cacti_{2}^{2\ka}] +\cn[\yti;\cacti_{1}^\ka] \rcl \leq c \, \cn[\yti;\Qti^\ka],$$
and $\norm{(\der y)_{ts} } \leq \norm{f_{ts} }+\lln t-s \rrn^\ga \cn[x^1;\cac_2^\ga] \cn[\zeta^{\yti};\cac_1^0]+\lln t-s\rrn^{2\ka} \cn[r^y;\cac_2^{2\ka}]$. But $\norm{f_{ts}} \leq \lln t-s \rrn \cn[\hti;\cl_1]$, hence $\cn[y;\cac_1^\ka ] \leq \ep^{1-\ka} \cn[\hti;\cl_1]+c_x \cn[\yti;\Qti^\ka]$, and (\ref{passage-yti-y}) is thus proved.

\end{proof}

An analog of Proposition \ref{cp:weak-phi} concerning the composition of a localized controlled process with a smooth function is the following:
\begin{proposition}\label{prop:passage-y-sigmay}
Let $y \in \ca_{f,h}^\ka(I)$, and consider a function $\si \in \cac^{3,b }$. Then $\si(y) \in \ca_{D\si(h)f,\si(h)}^\ka(I)$ and we have the following bound on the norm of $\si(y)$:
\begin{multline}\label{passage-y-sigmay}
\cm[\si(y);\ca_{D\si(h)f,\si(h)}^\ka(I)] \\
\leq c_{x,\si} \lcl 1+\cm[y;\ca_{f,h}^\ka(I)]^2+\ep^{1-\ka} \cm[y;\ca_{f,h}^\ka(I)] \cn[f;\cac_2^1(I)]+\ep^{1-\ka} \cn[f;\cac_2^1(I)] \rcl.
\end{multline}
Moreover, if $\yun, \yde \in \ca_{f,h}^\ka(I)$,
\begin{multline}\label{passage-y-sigmay-contraction}
\cn[\si(\yun)-\si(\yde);\cq^\ka(I)] \\
\leq c_{x,\si} \cn[\yun-\yde;\cq^\ka(I)] \big\{ 1+\cm[\yun;\ca_{f,h}^\ka(I)]^2+\cm[\yde;\ca_{f,h}^\ka(I)]^2\\
+\ep^{1-\ka} \cn[f;\cac_2^1(I)] ( 1+\cn[\yun;\cac_1^\ka(I)]+\cn[\yde;\cac_1^\ka(I)] ) \big\}.
\end{multline}
\end{proposition}

\begin{proof}
If $s,t \in I$, write
\begin{eqnarray}
\der(\si(y))_{ts} &=& \int_0^1 d\la \, D\si(y_s+\la(\der y)_{ts}) (\der y)_{ts} \nonumber\\
& =& \int_0^1 d\la \, D\si(y_s+\la(\der y)_{ts})f_{ts}+\int_0^1 d\la \, D\si(y_s+\la(\der y)_{ts})((x_{ts}^1 \zeta^y_s)^\ast +r^y_{ts}) \nonumber\\
&=& D\si(h)f_{ts}+(x^1_{ts} \zeta_s^{\si(y)})^\ast +r_{ts}^{\si(y),1}+r_{ts}^{\si(y),2},\label{decompo-sigma-around}
\end{eqnarray}
with $\zeta_s^{\si(y)}=\zeta_s^y D\si(y_s)^\ast$, $$r_{ts}^{\si(y),1}=\int_0^1 d\la \lc D\si(y_s+\la(\der y)_{ts})-D\si(y_s)\rc (x_{ts}^1 \zeta_s^y)^\ast+\int_0^1 d\la \, D\si(y_s+\la(\der y)_{ts}) r^y_{ts},$$ 
$$r_{ts}^{\si(y),2}=\int_0^1 d\la \lc D\si(y_s+\la(\der y)_{ts})-D\si(y_s)\rc f_{ts}+ \lc D\si(y_s)-D\si(y_a)\rc f_{ts}.$$
By standard computations, 
$$\cn[\zeta^{\si(y)};\cac_1^0]+\cn[\zeta^{\si(y)};\cac_1^\ka]+\cn[r^{\si(y),1};\cac_2^{2\ka}] \leq c_{x,\si} \lcl 1+\cm[y;\ca_{f,h}^\ka ]^2 \rcl.$$
Besides,
$$\norm{r_{ts}^{\si(y),2}} \leq \norm{D^2\si }_\infty \cn[y;\cac_1^\ka ] \cn[f;\cac_2^1] \lcl \lln t-s \rrn^{1+\ka}+\lln s-a \rrn^\ka \lln t-s \rrn \rcl,$$
and hence $\cn[r^{\si(y),2}; \cac_2^{2\ka}] \leq c_\si \cm[y;\ca_{f,h}^\ka ] \cn[f;\cac_2^1] \ep^{1-\ka}$. Finally, going back to decomposition (\ref{decompo-sigma-around}), we obtain:
\begin{multline*}
\norm{ \der(\si(y))_{ts} }\\
\leq \norm{D\si}_\infty \lln t-s \rrn \cn[f;\cac_2^1]+\lln t-s \rrn^\ka \cn[x^1;\cac_2^\ga] \cm[y;\ca_{f,h}^\ka] \norm{D\si}_\infty+\lln t-s \rrn^{2\ka} \cn[r^{\si(y)};\cac_2^{2\ka}],
\end{multline*}
so that $\cn[\si(y);\cac_1^\ka] \leq c_{\si,x} \lcl \ep^{1-\ka} \cn[f;\cac_2^1]+\cm[y;\ca_{f,h}^\ka]+\cn[r^{\si(y)};\cac_2^{2\ka}] \rcl$, which achieves the proof of (\ref{passage-y-sigmay}).

\smallskip

As for (\ref{passage-y-sigmay-contraction}), we have, with the notations (\ref{decompo-sigma-around}),
$$\der(\si(\yun)-\si(\yde))_{ts}=(x^1_{ts} [\zeta_s^{\si(\yun)}-\zeta_s^{\si(\yde)}])^\ast+[r_{ts}^{\si(\yun),1}-r_{ts}^{\si(\yde),1}]+[r_{ts}^{\si(\yun),2}-r_{ts}^{\si(\yde),2}].$$
If we refer now to the proof of \cite[Proposition 4]{Gu}, we effortlessly get
\begin{multline*}
\cn[\zeta^{\si(\yun)}-\zeta^{\si(\yde)};\cac_1^0(I)]+\cn[\zeta^{\si(\yun)}-\zeta^{\si(\yde)};\cac_1^\ka(I)]+\cn[r^{\si(\yun),1}-r^{\si(\yde),1};\cac_2^{2\ka}(I)]\\
\leq c_{x,\si} \lcl 1+\cm[\yun;\ca_{f,h}^\ka(I)]^2+\cm[\yde;\ca_{f,h}^\ka(I)]^2\rcl \cn[\yun-\yde;\cq^\ka(I)].
\end{multline*}
As far as $r^{\si(\yun),2}-r^{\si(\yde),2}$ is concerned, notice that
\begin{multline*}
\norm{r_{ts}^{\si(\yun),2}-r_{ts}^{\si(\yde),2}}
\leq \lln t-s \rrn \cn[f;\cac_2^1(I)]\\ \Big\{ \int_0^1 d\la \, \norm{D\si(\yun_s+\la(\der \yun)_{ts})-D\si(\yun_s)-D\si(\yde_s+\la(\der \yde)_{ts})+D\si(\yde_s)} \\
+\norm{D\si(\yun_s)-D\si(\yun_a)-D\si(\yde_s)+D\si(\yde_a) } \Big\}.
\end{multline*}
Some standard computations (see e.g. \cite[Lemma 3.1]{QT} for further details), using differentiations along the path 
$$
a(\mu,\nu)= \yun_s + \mu (\yun_t-\yun_s)+\nu (\yde_s-\yun_s)
+\mu\nu (\yde_t-\yde_s-\yun_t+\yun_s)
$$ 
defined for $\mu,\nu\in\ou$, then lead to 
\begin{multline*}
\cn[r^{\si(\yun),2}-r^{\si(\yde),2};\cac_2^{2\ka}(I)]\\
\leq c_\si \ep^{1-\ka} \cn[f;\cac_2^1(I)] \lcl 1+\cn[\yun;\cac_1^\ka(I)]+\cn[\yde;\cac_1^\ka(I)] \rcl \cn[\yun-\yde;\cac_1^\ka(I)].
\end{multline*}
Inequality (\ref{passage-y-sigmay-contraction}) easily follows.

\end{proof}

\smallskip

Observe again that $\ca_{f,h}^\ka$ is a subset of $\cq^{\ka}$, which means that, for any path $z\in\ca_{f,h}^\ka$, the integral $\cj(\dti x \, z)$ is defined thanks to Proposition \ref{prop:intg-controlled-proc}. In the particular context of a process $z\in\ca_{f,h}^\ka$, the bounds on this generalized integral can be improved as follows:
\begin{proposition}\label{prop:intg-local-ctrl-ps}
If $z \in \ca_{f,h}^\ka(I;\R^n)$, then the semi-norm of the process $\zti$ in $\Qti^\ka(I;\R)$ defined by $\zti_a=\hti \in \cl_1$ and $\delti \zti =\cj(\dti x \, z)$ can be estimated as
\begin{equation}\label{passage-sigmay-zti-1d}
\cn[\zti;\Qti^\ka(I)] \leq c_x \lcl \cn[z;\cac_1^0(I)]+\ep^{\ga-\ka}\cm[z;\ca_{f,h}^\ka(I)]+\ep^{1-\ka} \cn[f;\cac_2^1(I)]\rcl.
\end{equation}
\end{proposition}

\begin{proof}
According to Proposition \ref{prop:intg-controlled-proc}, $\zti$ can be decomposed as a controlled process, with $\zeta^{\zti}=z$ and $\rti^{\zti}=\rti^{\zti,1}+\rti^{\zti,2}$, where
$$\rti^{\zti,1}=\xti^2 \cdot  \zeta^z \quad \mbox{and} \quad \rti^{\zti,2}=\Lati(\xti^1 r^z+\xti^2 \cdot \der \zeta^z-\xti^3 \cdot \zeta^z+\xti^1 f).$$
First, since $(\der z)_{ts}=f_{ts}+x_{ts}^1 \zeta_s^z+r^z_{ts}$,
$$\cn[\zeta^{\zti};\cac_1^\ka(I)]=\cn[z;\cac_1^\ka(I)] \leq c_x \lcl \ep^{1-\ka}\cn[f;\cac_2^1(I)] +\ep^{\ga-\ka} \cm[z;\ca_{f,h}^\ka(I)] \rcl.$$
As for the remainder term, we have $\cn[\rti^{\zti,1};\cacti_{2}^{2\ka}] \leq c_x \ep^{2(\ga-\ka)} \cm[z;\ca_{f,h}^\ka(I)]$, while, thanks to the contraction property (\ref{contraction}),
$$\cn[\rti^{\zti,2};\cacti_{2}^{2\ka}(I)] \leq c_x \lcl \ep^\ga \cm[z;\ca_{f,h}^\ka(I)]+\ep^{1+\ga-2\ka}\cn[f;\cac_2^1(I)] \rcl.$$
Finally, $\cn[(\delti \zti)_{ts};\cl_1] \leq c_x \lcl \lln t-s \rrn^\ga \cm[z;\ca_{f,h}^\ka(I)]+\lln t-s \rrn^{1+\ga} \cn[f;\cac_2^1(I)]\rcl$, hence
$$\cn[\zti;\cacti_{1}^\ka(I)]\leq c_x \lcl \ep^{\ga-\ka} \cm[z;\ca_{f,h}^\ka(I)]+\ep^{1+\ga-\ka}\cn[f;\cac_2^1(I)]\rcl,$$
which achieves the proof of (\ref{passage-sigmay-zti-1d}).
\end{proof}

\begin{remark}
If $f \in \cac_2^1(I;\R^{k,l})$ and $h \in \R^{k,l}$, we can define $\ca_{f,h}^\ka(I;\R^{k,l})$ along the same lines as $\cq^\ka(I;\R^{k,l})$. If $z \in \ca_{f,h}^\ka(I;\R^{n,l})$, inequality (\ref{passage-sigmay-zti-1d}) remains true, that is
\begin{equation}\label{passage-sigmay-zti}
\cn[\zti;\Qti^\ka(I;\R^l)] \leq c_x \lcl \cn[z;\cac_1^0(I;\R^{n,l})]+\ep^{\ga-\ka}\cm[z;\ca_{f,h}^\ka(I;\R^{n,l})]+\ep^{1-\ka} \cn[f;\cac_2^1(I;\R^{n,l})]\rcl,
\end{equation}
where $\zti$ is defined analogously to Proposition \ref{prop:intg-local-ctrl-ps}.
\end{remark}

\subsection{Rough Volterra equations}

We are now in position to prove the main result of this section:

\begin{theorem}\label{main-theorem-rough}
Let $\ga \in (1/3,1/2)$ and $1/3 < \ka < \ga$. Assume $x$ satisfies Hypothesis \ref{hyp:X2} and $\si \in \cac^{3,\textbf{\textit{b}} }(\R^{1,d}; \R^{n,d})$. Then Equation (\ref{eq:volterra-ytilde-bis}) admits a unique solution in $\Qti_{0}^\ka([0,T];\R^{d})$.
\end{theorem}

\begin{proof}
As in the Young case, the solution we are looking for is seen as a fixed point of some naturally defined application $\Gamma$. The fixed point argument is then divided into two steps: we first establish the invariance of some well-chosen balls of $\Qti_1^\ka$, and then show a contraction property on these balls.

\smallskip

{\it Step 1: invariant balls.} Fix a positive integer $N$ and consider a sequence of intervals $I^N_n=[l^N_n,l^N_{n+1}]$ with $l_0^N=0$ and $\ep_n =\ep_n^N=l^N_{n+1}-l^N_n=\frac{1}{N+n}$, so that $[0,T]$ is covered by a finite union of $(I^N_n)_{n\geq 0}$. Introduce also a sequence of balls 
$$B_n^{\hti_n}=\{ \yti \in \Qti^\ka(I^N_n): \ \yti_{l_n^N}=\hti_n, \, \zeta^{\yti}_{l_n^N}=\si(h_n), \,   \cn[\yti; \Qti^\ka(I^N_n)] \leq (N+n)^{\alpha_2} \},$$
where $\hti_n \in \cl_1$ is such that $\cn[\hti_n; \cl_1] \leq (N+n)^{\alpha_1}$. We are thus given a control over both $\delti \yti$ and the initial condition $\yti_{l^N_n}$. If $\yti \in B_n^{\hti_n}$, $\zti:=\Gamma(\yti)$ is the path in $\cacti_{1}^\ga(I^N_n)$ defined by the two conditions: $\zti_{l^N_n}=\hti_n$ and for all $s,t \in I^N_n$, $(\delti \zti)_{ts}=\cj_{ts}(\dti x \, \si(y))$, with $y=a_0+\int_0^\infty d\xi \, \hphi(\xi)\yti(\xi)$. With these notations, and using the previous propositions, we are going to prove the existence of two constants $\alpha_1,\alpha_2 >0$ such that the sets $B_n^{\hti_n}$ are invariant by $\Gamma$ and the following property holds: 
\begin{center}
(H) \ \ \ \ If $\yti \in B_n^{\hti_n}$, then $\cn[\yti_{l^N_{n+1}}; \cl_1]\leq (N+n+1)^{\alpha_1}$. 
\end{center}
Thanks to (H), the local solutions can then be patched together, as we shall see at the end of the proof.

\smallskip

Let $\yti \in B_n^{\hti_n}$, $\zti =\Gamma(\yti)$. As in Proposition \ref{prop:passage-yti-y}, denote $f^n_{ts}=\int_0^\infty d\xi \, \hphi(\xi) a_{ts}(\xi) e^{-\xi(s-a)}$ $\hti_n(\xi)$ 
and $h_n=a_0+\int_0^\infty d\xi \, \hphi(\xi)\hti_n(\xi)$. In order to estimate $\cn[\zti;\cq^\ka(I^N_n)]$, use successively (\ref{passage-sigmay-zti}), (\ref{passage-y-sigmay}) and (\ref{passage-yti-y}), together with the fact that
$$\cn[D\si(h_n)f^n;\cac_2^1(I^N_n)] \leq c_\si \cn[f^n;\cac_2^1(I^N_n)] \leq c_\si \cn[\hti_n;\cl_1],$$
to get 
\begin{multline}\label{eqn:resultat}
\cn[\zti;\Qti^\ka(I^N_n)] \leq c_{\si,x}^1 \big\{ 1+\ep_n^{\ga-\ka} \cn[\yti;\Qti^\ka(I^N_n)]^2+\ep_n^{\ga-\ka+2(1-\ka)} \cn[\hti_n;\cl_1]^2\\
+\ep_n^{1+\ga-2\ka} \cn[\yti;\Qti^\ka(I^N_n)] \cn[\hti_n;\cl_1]+\ep_n^{1-\ka} \cn[\hti_n;\cl_1] \big\}.
\end{multline}

If one desires to stay in the ball $B_n^{\hti_n}$ after applying $\Gamma$, one is naturally led to consider the system
\begin{equation}\label{systeme-alpha}
\begin{cases}
2\alpha_2-(\ga-\ka) < \al_2\\
2\al_1-(\ga-\ka+2(1-\ka)) < \al_2 \\
\al_1+\al_2-(1+\ga-2\ka) < \al_2\\
\al_1-(1-\ka) < \al_2,
\end{cases}
\end{equation}
which reduces to
$$\begin{cases}
\al_2 < \ga-\ka \\
\al_1-1 < \al_2 -\ka.
\end{cases}$$
In fact, for some reasons that will arise soon, we should add the conditions $\al_2 < \frac{\ga-\ka}{2}$ and $\al_1 -1>\al_2-\ga$, which turn the previous system into
$$\begin{cases}
0 <\al_2 < \frac{\ga-\ka}{2} \\
\al_2-\ga <\al_1-1 < \al_2 -\ka.
\end{cases}$$

\smallskip

Notice that the conditions above can be easily met (and are assumed to be met in the sequel) whenever $\ka<\ga$. Now, going back to (\ref{eqn:resultat}), we get $\cn[\zti; \Qti^\ka(I^N_n)] \leq 6c_{x,\si}^1 \,(N+n)^{\al_3}$, where $\al_3$ stands for the maximum of the left members of the system (\ref{systeme-alpha}). As $\al_3 < \al_2$, we can pick $N$ sufficiently large such that for any $n \geq 0$, $(N+n)^{\al_2-\al_3} \geq 6c_{x,\si}^1$, and so $\cn[\zti;\Qti^\ka(I_n^N)]\leq (N+n)^{\al_2}$.

\smallskip

It remains to analyze the condition (H). But $\yti_{l^N_{n+1}}=e^{-.\, \ep_n} \yti_{l^N_n}+(\delti \yti)_{l^N_{n+1}l^N_n}$, so that
$$\cn[\yti_{l^N_{n+1}}; \cl_1] \leq (N+n)^{\al_1}+\cn[\yti; \cacti_{1}^\ga(I^N_n)] \frac{1}{(N+n)^\ga} \leq (N+n)^{\al_1}+c_x^2 \, (N+n)^{\al_2-\ga}.$$
Now notice that, when $m\to\infty$, we have $\frac{c_x^2 \, m^{\al_2-\ga}}{(m+1)^{\al_1}-m^{\al_1}} \sim \frac{c_x^2}{\al_1}m^{\al_2-\ga-(\al_1-1)}$. But remember that we have assumed $\al_1-1 >\al_2-\ga$, so that if $N$ is large enough, the last equivalent yields $\frac{c_x^2 \, (N+n)^{\al_2-\ga}}{(N+n+1)^{\al_1}-(N+n)^{\al_1}} \leq 1$ for any $n \geq 0$. Hence $\cn[\yti_{l_{n+1}^N};\cl_1] \leq (N+n+1)^{\al_1}$, which achieves the first step.

\smallskip

\noindent 
{\it Step 2: contraction property.} The contraction argument is now easy to settle. Indeed, if $\yti^{(i)} \in B_n^{\hti_n}$ and $\zti^{(i)} =\Gamma(\yti^{(i)})$, then, owing to relation (\ref{passage-sigmay-zti}), we have
\begin{multline*}
\cn[\ztiun-\ztide; \Qti_1^\ka(I^N_n)] \\
\leq c \lcl \cn[\si(\yun)-\si(\yde); \cac_1^0(I^N_n)]+\ep_n^{\ga-\ka} \, \cn[\si(\yun)-\si(\yde); \cq^\ka(I^N_n)] \rcl.
\end{multline*}
But $\cn[\si(\yun)-\si(\yde); \cac_1^0(I^N_n)] \leq \ep_n^\ga \, \cn[\si(\yun)-\si(\yde); \cq^\ka(I^N_n)]$, and the previous relation, together with (\ref{passage-y-sigmay-contraction}) and (\ref{passage-yti-y}), gives $\cn[\ytiun-\ytide;\Qti^\ka(I^N_n)] \leq c_{\si,x} \, J_{N+n} \, \cn[\ytiun-\ytide;\Qti_1^\ka(I^N_n)]$, with
$$J_n=n^{-(\ga-\ka)} \lcl 1+n^{2\al_2}+n^{-2(1-\ka)} n^{2\al_1}+n^{-(1-\ka)} n^{\al_1} \lcl 1+n^{\al_2}+n^{-(1-\ka)} n^{\al_1} \rcl \rcl.$$

It is finally readily checked that the two conditions $2\al_2-(\ga-\ka)<0$ and $\al_1<\al_2+1-\ka$ entail $ \lim_{N\to \infty }J_N=0$. Therefore, here again, we just have to take $N$ sufficiently large for the contraction argument to work on the balls $B_n^{\hti_n}$, $n\geq 0$.

\smallskip

\noindent 
{\it Step 3: patching solutions.}
The construction of the announced solution $\yti \in \Qti_{0}^\ka([0,T])$ reduces now to a patching argument. Let us make it precise.

\smallskip

First, define a sequence $(\yti^n,\zeta^{\yti^n})_{n\geq 0}$ by the recursive condition: $(\yti^0,\zeta^{\yti^0}) \in \Qti^\ka(I_0^N)$ is the fixed point of $\Gamma$ in $B_0^0$ and for any $n \geq 1$, $(\yti^n,\zeta^{\yti^n}) \in \Qti^\ka(I_n^N)$ is the fixed point of $\Gamma$ in $B_n^{\yti_{l_n^N}^{n-1}}$. This construction is allowed by the first part. Then set, for any $t\in [0,T]$,
$$\yti_t=\sum_{n=0}^{N_T} \yti_t^n \, \textbf{1}_{I_n^N}(t) \quad , \quad \zeta_t^{\yti}=\sum_{n=0}^{N_T} \zeta_t^{\yti^n} \, \textbf{1}_{I_n^N}(t),$$
where $N_T$ stands for the lowest integer such that $\sum_{n=0}^{N_T} |I_n^N | \geq T$.

\smallskip

If $l_{k-1}^N <s \leq l_k^N < \ldots < l_{k'}^N \leq t < l_{k'+1}^N$, use the relation
\begin{equation}\label{decomposition-intervalles-successifs}
(\delti \yti)_{ts}=e^{-\cdot (t-l_k^N)}(\delti \yti)_{l_k^N s}+(\delti \yti)_{t l_{k'}^N}+\sum_{i=k}^{k'-1} e^{-\cdot (t-l_{i+1}^N)}(\delti \yti)_{l_{i+1}^Nl_i^N},
\end{equation}
together with $\delti \xti^1=0$, to deduce $(\delti \yti)_{ts}=\xti^1_{ts} \zeta_s^{\yti}+\rti_{ts}^{\yti}$, where $\rti_{ts}^{\yti}=\rti_{ts}^{\yti,1}+\rti_{ts}^{\yti,2}$,
$$\rti_{ts}^{\yti,1}=\xti_{tl_k^N}^1 \lc \zeta_{l_k^N}^{\yti^k}-\zeta_s^{\yti^{k-1}}\rc+\sum_{i=k+1}^{k'} \xti^1_{tl_i^N} \lc \zeta_{l_i^N}^{\yti^i}-\zeta_{l_{i-1}^N}^{\yti^{i-1}} \rc,$$
$$\rti_{ts}^{\yti,2}=e^{-\cdot (t-l_k^N)} \rti_{l_k^N s}^{\yti^{k-1}}+\rti_{t l_{k'}^N}^{\yti^{k'}}+\sum_{i=k}^{k'-1} e^{-\cdot (t-l_{i+1}^N)} \rti_{l_{i+1}^n l_i^N}^{\yti^i}.$$
Owing to the regularity of each $\zeta^{\yti^k}$, this proves that $(\yti,\zeta^{\yti})$ actually belongs to $\Qti_{0}^\ka([0,T])$.

\smallskip

Finally, let us go back to the decomposition (\ref{decomposition-intervalles-successifs}) to deduce
$$
(\delti \yti)_{ts}= e^{-\cdot (t-l_k^N)} \cj_{l_k^N s}(\dti x \, \si(y))+\cj_{t l_{k'}^N}(\dti x \, \si(y))+\sum_{i=k}^{k'-1} e^{-\cdot (t-l_{i+1}^N)} \cj_{l_{i+1}^Nl_i^N}(\dti x \, \si(y)).
$$
Furthermore, invoking the fact that $\delti \lp \cj(\dti x \, z)\rp =0$, we obtain:
$$
\cj_{t l_{k'-1}^N}(\dti x \, \si(y))=\cj_{t l_{k'}^N}(\dti x \, \si(y))+
e^{-\cdot (t-l_{k'}^N)} \cj_{l_{k'}^Nl_{k'-1}^N}(\dti x \, \si(y)),
$$
and hence
$$
(\delti \yti)_{ts}= e^{-\cdot (t-l_k^N)} \cj_{l_k^N s}(\dti x \, \si(y))+\cj_{t l_{k'-1}^N}(\dti x \, \si(y))+\sum_{i=k}^{k'-2} e^{-\cdot (t-l_{i+1}^N)} \cj_{l_{i+1}^Nl_i^N}(\dti x \, \si(y)).
$$
Iterating this procedure, we end up with the relation $(\delti \yti)_{ts}=\cj_{ts}(\dti x \, \si(y))$ for all $s,t\in\ott$, which proves that $y$ is a global solution to equation (\ref{eq:volterra-ytilde-bis}).

\end{proof}

\subsection{Application to the Brownian case}\label{sec:appli-usual-brownian}
We now intend to show that the previous results can be applied to a (classical) brownian motion $X=(X^{(1)},\ldots,X^{(n)})$ with values in $\R^{1,n}$. In other words, we shall consider the processes $\Xti^1$, $\Xti^2$, $\Xti^3$ defined in a natural way, according to Remark \ref{rmk:4.5}, by
$$\Xti^1_{ts}=\int_s^t e^{-.(t-v)}dX_v, \ \Xti^2_{ts}=\int_s^t e^{-.(t-v)} dX_v \otimes X^1_{vs}, \ \Xti^3_{tus}=\int_u^t e^{-.(t-v)}dX_v \otimes (\der X^1)_{vus},$$
where $X^1_{ts}=\ist \phi(t-v) \, dX_v$, and where all the stochastic integrals above are understood in the Itô sense. We thus have to prove that those processes satisfy the required regularity conditions.
 
\smallskip

As far as $\Xti^1$ is concerned, we can use the same proof as in the Young case, and the following regularity result is easily shown:
\begin{lemma}
If $\hphi$ is such that $\int_0^\infty d\xi \, |\hphi(\xi)|(1+\xi) < \infty $, then, for any $\ga \in (1/3,1/2)$, $\Xti^1 \in \cacti_{1}^\ga([0,T];\R^{1,n})$ a.s.
\end{lemma}

\begin{proof}
The same trick as in the proof of Lemma \ref{lem:c-2-delta} leads to the existence of a continuous version of $\Xti^1(\xi)$ for any fixed $\xi$. Now, just as in the Young case (see the proof of Theorem \ref{thm:3.14}), it is readily checked that $\delti\Xti^1=0$, and thus one is allowed to write, for any $p>0$, 
\begin{equation}\label{majo-xtilde-1}
\norm{ \Xti^1_{ts}(\xi)} \leq C \, \lln t-s\rrn^\ga (\Uti_{\ga,2p}(\xi))^{1/2p},\ \mbox{where} \ \Uti_{\ga,2p}(\xi)=\iint_{0 <s<t<T} \frac{\norm{\Xti_{wv}^1(\xi)}^{2p}}{\lln w-v\rrn^{2\ga p+2}} dvdw.
\end{equation}
Our claim is thus easily proved by replacing (\ref{isom-young}) with the usual Itô isometry property.

\end{proof}

Consider now $\Xti^3$, and notice that, up to a Fubini-type theorem, this process can be written for all $\xi\geq 0$ as
\begin{equation}\label{decompo-xtilde-3}
\Xti^3_{tus}(\xi)=\int_0^\infty d\eta \, \hphi(\eta) \Xdeti^4_{tu}(\xi,\eta)\otimes \Xti^1_{us}(\eta),
\end{equation}
with $\Xdeti^4_{tu}(\xi,\eta)=\int_u^t e^{-\xi(t-v)}a_{vu}(\eta)dX_v$. The issue then consists in studying the regularity of $\Xdeti^4$. To this end, we shall resort to a GRR-type argument, which requires the introduction of a new incremental operator $\deldeti$ acting on the space $\cacdeti_2$ of applications on $\mathcal{S}_2$ with values in the space of two-variables functions. This operator should send $\cacdeti_2$ into the space $\cacdeti_3$ of applications on $\cs_3$ with values in the space of two-variables functions.

\smallskip

In order to define $\deldeti$, observe that for all $\xi,\eta\geq 0$, 
\begin{multline}\label{eq:50}
\Xdeti_{ts}^4(\xi,\eta)-\Xdeti_{tu}^4(\xi,\eta) e^{-\eta(u-s)}-e^{-\xi(t-u)} \Xti_{us}^4(\xi,\eta)\\
= \int_u^t e^{-\xi(t-v)}\lc a_{vs}(\eta)-a_{vu}(\eta)e^{-\eta(u-s)}\rc dX_v  =  \Xti_{tu}^1(\xi) a_{us}(\eta),
\end{multline}
the regularity of which is known. This simple relation yields naturally the following:
\begin{definition}
If $\Rdeti \in \cacdeti_2$, let $\deldeti \Rdeti$ the element of $\cacdeti_3$ defined by the relation
$$(\deldeti \Rdeti)_{tus}(\xi,\eta)=(\der \Rdeti)_{tus}(\xi,\eta)-a_{tu}(\xi)\Rdeti_{us}(\xi,\eta)-\Rdeti_{tu}(\xi,\eta)a_{us}(\eta)$$
for any $\xi,\eta \geq 0$.
\end{definition}
With such a definition, the above relation (\ref{eq:50}) can be written as: $(\deldeti \Xdeti^4)(\xi,\eta) =\Xti^1_{tu}(\xi)  a_{us}(\eta)$. Furthermore, we have the following equivalent of Proposition \ref{prop:g-r-r}, whose proof is postponed to the appendix for sake of readability:

\begin{proposition}\label{prop:4.19}
Let $(V, \norm{.})$ a Banach space and fix $\xi,\eta \geq 0$. Let $\Rdeti \in \cacdeti_2(T;E)$ such that $\Rdeti_{..} (\xi,\eta) \in \cac_2(V)$, and set
$$\Udeti(\xi,\eta)=\iint_{0<s<t<T} \psi\lp \frac{\norm{ \Rdeti_{ts}(\xi,\eta)}}{\phi(\lln t-s\rrn)} \rp dtds,$$
where $\psi,\phi: \mathbb{R}^+ \rightarrow \mathbb{R}^+$ are strictly increasing functions and $\phi(0)=0$. Assume now that there exists some $\Cdeti(\xi,\eta)\geq 0$ such that, for all $\ell_1 <\ell_2 \in [0,T]$, 
$$\sup_{\ell_1 \leq u\leq \ell_2} \norm{(\deldeti \Rdeti)_{\ell_2 u \ell_1}(\xi,\eta)}\leq \psi^{-1}\lp \frac{4\, \Cdeti(\xi,\eta)}{\lln \ell_2-\ell_1\rrn^2}\rp \phi\lp \lln \ell_2-\ell_1 \rrn\rp.$$
Then, for all $s,t \in T$,
$$\norm{ \Rdeti_{ts}(\xi,\eta)} \leq c \int_0^{\lln t-s\rrn} \lc \psi^{-1}\lp \frac{4\, \Udeti(\xi,\eta)}{r^2}\rp + \psi^{-1}\lp \frac{4\, \Cdeti(\xi,\eta)}{r^2}\rp \rc d\phi(r).$$
\end{proposition}

\smallskip

It is now possible to give some regularity results for the increment $\Xti^3$:
\begin{lemma}
If $\hphi$ is such that $\int_0^\infty d\xi \, |\hphi(\xi)|(1+\xi) < \infty$, then, for any $\ga \in (1/3,1/2)$, $\Xti^3 \in \cacti_{3}^{3\ga}([0,T];\R^{n,n})$ a.s.
\end{lemma}

\begin{proof}
We will apply of course Proposition \ref{prop:4.19} to $\Xdeti^4$, with $\psi(x)=x^{2p'}$ and $\phi(x)=x^{2\ga+1/p'}$. The same arguments as in the proof of Lemma \ref{lem:c-2-delta} enable to assert that, a.s.,  $\Xdeti^4_{..}(\xi,\eta) \in \cac_2(\R^{1,n})$ for all $\xi,\eta \geq 0$. To find out what $\Cdeti(\xi,\eta)$ should be, remember that $(\deldeti \Xdeti^4)_{\ell_2 u\ell_1}(\xi,\eta)=\Xti^1_{\ell_2 u}(\xi)a_{ul_1}(\eta)$. Hence, according to (\ref{majo-xtilde-1}),
\bean
\norm{(\deldeti \Xdeti^4)_{\ell_2 u\ell_1}(\xi,\eta)} &\leq & c \, \lln l_2-u\rrn^\ga (\Uti_{\ga,2p}(\xi))^{1/2p} \eta^\ga \lln u-\ell_1 \rrn^\ga\\
&\leq & c \, \lln \ell_1-\ell_2\rrn^{2\ga} \lc (\Uti_{\ga,2p}(\xi))^{1/2p} \eta^\ga \rc.
\eean
Therefore, set
\begin{equation}\label{notation-cdeti-udeti}
\Cdeti_{\ga,2p,2p'}(\xi,\eta)=(\Uti_{\ga,2p}(\xi))^{2p'/2p} \eta^{2p' \ga}, \quad \Udeti_{2\ga,2p'}(\xi,\eta)=\iint_{0 <v<w<T} \frac{\norm{ \Xdeti^4_{wv}(\xi,\eta)}^{2p'}}{\lln w-v\rrn^{4p'\ga+2}} dvdw
\end{equation}
and with these notations,
$$\norm{ \Xdeti_{ts}^4(\xi,\eta)} \leq c \, \lln t-s\rrn^{2\ga} \lcl (\Udeti_{\ga,2p'}(\xi,\eta))^{1/2p'}+(\Cdeti_{\ga,2p,2p'}(\xi,\eta))^{1/2p'} \rcl.$$

Going back to (\ref{decompo-xtilde-3}) and estimating $\Xti_{us}^1(\eta)$ with
$\norm{ \Xti^1_{us}(\eta)} \leq c \lln u-s\rrn^\ga (\Uti_{\ga,2p''}(\eta))^{1/2p''}$ for some $p'' >0$, we get
$$\norm{ \Xti_{tus}^3(\xi)} \leq c  \lln t-u\rrn^{2\ga}\lln u-s\rrn^\ga \Rti_{\ga,2p,2p',2p''}(\xi),$$
where
\begin{eqnarray}\label{eq:52}
\Rti_{\ga,2p,2p',2p''}(\xi) &=& \int_0^\infty d\eta \, |\hphi(\eta)|(\Udeti_{2\ga,2p'}(\xi,\eta))^{1/2p'} (\Uti_{\ga,2p''}(\eta))^{1/2p''}  \nonumber\\
& & \hspace{3cm} +\int_0^\infty d\eta \, |\hphi(\eta)|(\Cdeti_{\ga,2p,2p'}(\xi,\eta))^{1/2p'} (\Uti_{\ga,2p''}(\eta))^{1/2p''}  \nonumber\\
&=& \Rti^1_{\ga,2p',2p''}(\xi)+\Rti^2_{\ga,2p,2p',2p''}(\xi).
\end{eqnarray}
To prove that $\Xti^3 \in \cacti_{1}^{3\ga}$ a.s, it is now sufficient to show that $\cn[\Rti_{\ga,2p,2p',2p''}; \cl_1] <\infty $ a.s, which will be seen as a consequence of $E[\cn[\Rti_{\ga,2p,2p',2p''}; \cl_1]] <\infty$.

\smallskip

In order to prove this latter relation, use first succesively Schwarz and Jensen inequalities to obtain
\beq\label{eq:53}
E[ (\Udeti_{2\ga,2p'}(\xi,\eta))^{1/2p'} (\Uti_{\ga,2p''}(\eta))^{1/2p''}] \leq E[\Udeti_{2\ga,2p'}(\xi,\eta)]^{1/2p'}E[\Uti_{\ga,2p''}(\eta)]^{1/2p''}.
\eeq 
To estimate the first term in the right hand side above, we resort to the fact that
$$E[\norm{ \Xdeti_{wv}^4(\xi,\eta)}^2] =\int_v^w e^{-2\xi(w-u)}a_{uv}^2(\eta) du \leq  \eta^2 \lln w-v\rrn^2.
$$
Furthermore, $\Xdeti_{wv}^4(\xi,\eta)$ is a random variable in the second chaos of the Brownian motion, on which all the $L^p$-norms are equivalent. Thus $E[\norm{ \Xdeti_{wv}^4(\xi,\eta)}^{2p'}]\le \eta^{2p'} \lln w-v\rrn^{2p'}$ for any $p'\ge 1$, which yields:
$$E[\Udeti_{2\ga,2p'}(\xi,\eta)]^{1/2p'} \leq c \, \eta \lp \iint_{0<v<w<T} \lln w-v\rrn^{2p'-4\ga p'-2} dvdw \rp^{1/2p'}.$$
Hence, if we take $p'$ such that $p'-2\ga p'-1 >0$, that is $p' >1/(1-2\ga)$, then the quantity $\Udeti_{2\ga,2p'}(\xi,\eta)$ can be bounded as $E[\Udeti_{2\ga,2p'}(\xi,\eta)]^{1/2p'}$  $\leq C \, \eta$. As for the second term of (\ref{eq:53}), we have (remember that $n$ stands for the dimension of $\Xti^1$)
$$E[\norm{ \Xti_{wv}^1(\eta)}^2] =n\int_v^w e^{-2\eta(w-u)}du \leq n \lln w-v\rrn,$$
so that the same kind of arguments as for $\Xti^4$ yield 
$$
E[\Uti_{\ga,2p''}(\eta)] \leq C \iint_{0<v<w<T} \lln w-v\rrn^{p''-2\ga p''-2}dvdw.
$$ 
By choosing $p''> 2/(1-2\ga)$, we get $E[\Uti_{\ga,2p''}(\eta)] \leq c$. Consequently , recalling that $\Rti^1$ is defined at equation (\ref{eq:52}), one gets:
\begin{multline}\label{eq:54}
E[\cn[\Rti_{\ga,2p',2p''}^{1};\cl_1]] \leq c \int_0^\infty d\xi \, |\hphi(\xi)|(1+\xi) \int_0^\infty d\eta \, |\hphi(\eta)| \, \eta\\
\leq c \lp \int_0^\infty d\xi \, |\hphi(\xi)|(1+\xi)\rp^2.
\end{multline}

As far as $\Rti_{\ga,2p,2p',2p''}^2$ is concerned, we use the definition of $\Cdeti_{\ga,2p,2p'}$, together with the previous estimation of $E[\Uti_{\ga,2p''}(\eta)]$, to assert that, if $p> 2/(1-2\ga)$,
\bean
E[\Cdeti_{\ga,2p,2p'}(\xi,\eta)^{1/2p'} \Uti_{\ga,2p''}(\eta)^{1/2p''} ] &=& \eta^\ga E[\Uti_{\ga,2p}(\xi)^{1/2p} \Uti_{\ga,2p''}(\eta)^{1/2p''} ] \\
&\leq & \eta^\ga E[\Uti_{\ga,2p}(\xi)]^{1/2p} E[\Uti_{\ga,2p''}(\eta)]^{1/2p''}\ \leq C \ \, \eta^\ga.
\eean
Hence,
\begin{multline}\label{eq:55}
E[\cn[\Rti_{\ga,2p,2p',2p''}^2 ;\cl_1]] \leq c\int_0^\infty d\xi \, |\hphi(\xi)|(1+\xi)\int_0^\infty d\eta \, |\hphi(\eta)| \, \eta^\ga\\
 \leq c \lp \int_0^\infty d\xi \, |\hphi(\xi)|(1+\xi) \rp^2.
\end{multline}
Putting together the estimates (\ref{eq:54}) and (\ref{eq:55}), we end up with 
$E[\cn[\Rti_{\ga,2p,2p',2p''}; \cl_1]] <\infty$, which ends the proof.

\end{proof}

It remains to analyze the regularity of $\Xti^2$. To this purpose, we will apply Proposition \ref{prop:g-r-r} again, which means that both the moments of $\Xti^2$ and $\delti \Xti^2 =\Xti^1 X^1+\Xti^3$ have to be controlled. We first have to check the following property:
\begin{lemma}\label{lem:4.21}
If $\int_0^\infty d\eta \, |\hphi(\eta)| < \infty$, then, a.s., $\Xti^2_{..}(\xi) \in \cac_2(\R^{n,n})$ for any $\xi \geq 0$.
\end{lemma}
\begin{proof}
This is the same Kolmogorov-type argument as the one used in Lemma \ref{lem:c-2-delta}. The details are left to the reader.

\end{proof}

\smallskip

Let us state now the regularity result for $\Xti^2$:
\begin{lemma}
If $\hphi$ is such that $\int_0^\infty d\xi \, |\hphi(\xi)|(1+\xi) < \infty$, then, for any $\ga \in (1/3,1/2)$, $\Xti^2 \in \cacti_{2}^{2\ga}([0,T]; \R^{n,n})$ a.s.
\end{lemma}

\begin{proof}
Invoking the fact that $\delti \Xti^2 =\Xti^1 X^1+\Xti^3$ and the previous estimations of $\Xti^1$ and $\Xti^3$, we deduce $\norm{(\delti \Xti^2)_{\ell_2 u\ell_1}(\xi)} \leq c \lln \ell_2-\ell_1\rrn^{2\ga} \Dti(\xi)^{1/4p_5}$, with
$$\Dti(\xi)^{1/4p_5} =\Rti_{\ga,2p_0,2p_1,2p_2}(\xi)+(\Uti_{\ga,2p_3}(\xi))^{1/2p_3} \int_0^\infty d\eta \, |\hphi(\eta)| (\Uti_{\ga,2p_4}(\eta))^{1/2p_4}.$$

We are thus ready to apply Proposition \ref{prop:g-r-r} to $\Xti^2$ with $\psi(x)=x^{4p_5}$ and $\phi(x)=x^{2\ga+1/2p_5}$ to get $\norm{ \Xti^2_{ts}(\xi)} \leq c \, \lln t-s\rrn^{2\ga} \lcl \Vti_{2\ga,4p_5}(\xi)^{1/4p_5}+\Dti(\xi)^{1/4p_5} \rcl$, where we set
$$\Vti_{2\ga,4p_5}(\xi)=\iint_{0<v<w<T} \frac{\norm{\Xti_{wv}^2(\xi)}^{4p_5}}{\lln w-v\rrn^{8\ga p_5+2}} dvdw,$$
and so $\cn[\Xti^2; \cacti_{1}^{2\ga}] \leq c \lcl \cn[\Vti_{2\ga,4p_5}^{1/4p_5}; \cl_1]+\cn[\Dti^{1/4p_5}; \cl_1] \rcl $.

\smallskip

The fact that $\cn[\Dti^{1/4p_5}; \cl_1] < \infty$ a.s has been shown while studying the regularities of $\Xti^1$ and $\Xti^3$, for some well-chosen $p_0,p_1,p_2,p_3,p_4$.

\smallskip

To conclude with, let us prove that $E[\cn[\Vti_{2\ga,4p_5}^{1/4p_5};\cl_1]] < \infty$: notice that $$E[\cn[\Vti_{2\ga,4p_5}^{1/4p_5};\cl_1]] \leq \int_0^\infty d\xi \, |\hphi(\xi)|(1+\xi)E[\Vti_{2\ga,4p_5}(\xi)]^{1/4p_5},$$ so that the issue consists in estimating $E[\norm{\Xti^2_{wv}(\xi)}^{4p_5}]$. To this end, observe that
$$E\lc \norm{\Xti^2_{wv}(\xi)}^{4p_5}\rc \leq c\lcl E\lc | \Xti_{wv}^{2,(1,1)}(\xi)|^{4p_5}\rc +E\lc | \Xti_{wv}^{2,(1,2)}(\xi)|^{4p_5}\rc \rcl,$$
where $\Xti_{wv}^{2,(i,j)}(\xi)$ is defined as $\ist e^{-\xi(t-v)} dX_v^{(i)}\, X_{vs}^{1,(j)}$. But, thanks to the Burkholder-Davis-Gundy inequality, we know that
\bean
E[| \Xti^{2,(1,1)}_{wv}(\xi)|^{4p_5}] &\leq& c E\lc \lp \int_v^w (e^{-\xi(w-s)} X^{1,(1)}_{sv})^2 ds \rp^{2p_5} \rc  \\ 
&\leq& c \lln w-v \rrn^{2p_5-1} \int_v^w E\lc (X^{1,(1)}_{sv})^{4p_5} \rc ds,
\eean
and
\bean
E[(X^{1,(1)}_{sv})^2] &=& E\lc \lp \int_0^\infty d\eta \hphi(\eta)\int_v^s e^{-\eta(s-t)} dX^{(1)}_t\rp^2\rc\\
&=& \int_v^s \lp \int_0^\infty d\eta \hphi(\eta)e^{-\eta(s-t)}\rp^2 dt \ \leq \ \lln s-v\rrn \lp \int_0^\infty d\eta \hphi(\eta)\rp^2,
\eean
which gives $E[(X^{1,(1)}_{sv})^{4p_5}] \leq c \lln s-v\rrn^{2p_5}$ and thus $E[(\Xti^{2,(1,1)}_{wv}(\xi))^{4p_5}] \leq c \lln w-v\rrn^{4p_5}$. 

\smallskip
In fact, this reasoning remains true for $E[(X^{1,(1,2)}_{sv})^{4p_5}]$, so that finally $E[\norm{\Xti^{2}_{wv}(\xi)}^{4p_5}] \leq c \lln w-v\rrn^{4p_5}$.
If we take $p_5>1/(1/2-\ga)$, then we get
$$E[\Vti_{2\ga,4p_5}(\xi) ] \leq \iint_{0<v<w<T} \lln w-v\rrn^{4p_5-8\ga p_5-2}dwdv \leq M  < \infty,$$
which leads to the announced claim $E[\cn[\Vti_{2\ga,4p_5}^{1/4p_5};\cl_1]] < \infty$, provided $\int_0^\infty d\xi \, |\hphi(\xi)|(1+\xi) < \infty$.

\end{proof}

We are now able to write Theorem \ref{main-theorem-rough} in the Brownian setting:

\begin{theorem}
Let $X=(X^{(1)},\ldots,X^{(n)})$ a standard Brownian motion on $[0,T]$ with values in $\R^{1,n}$. Introduce coefficients $\ga \in (1/3,1/2)$, $\ka \in (1/3,\ga)$ and assume that $\int_0^\infty d\xi \, |\hphi(\xi)|(1+\xi) < \infty$. If $\si\in \cac^{3,\textbf{\textit{b}}}(\R^{1,d};\R^{n,d})$, then, a.s, the system
$$\begin{cases}
\Yti_0  =  0\\
(\delti \Yti)_{ts} =  \cj_{ts}\lp \dti X \, \si\lp a+\int_0^\infty d\eta \, \hphi(\eta) \, \Yti(\eta)\rp\rp
\end{cases} $$
admits a unique solution in $\Qti^\ka([0,T],\R^{1,d})$.

\end{theorem}

%%%%%%%%%%%%%%%%%%%%%%%%%%%%%%%%%%%%%%%%%%%
%%%%%%%%%%%%%%%%%%%%%%%%%%%%%%%%%%%%%%%%%%%

\section{Application to a fBm with Hurst parameter $H \in (1/3,1/2)$}\label{sec:applic-rough-fbm}
This section is devoted to prove that Hypothesis \ref{hyp:X2} is fulfilled for a $n$-dimensional fractional Brownian motion with Hurst parameter $H\in (1/3,1/2)$. More specifically, we will construct a stochastic vector $$(X,\Xti^1,\Xti^2,\Xti^3) \in L^1(\Omega;\cac_1^\ga(\R^{1,n}) \times \cacti_{2}^\ga(\R^{1,n}) \times \cacti_{2}^{2\ga}(\R^{n,n}) \times \cacti_{3}^{3\ga}(\R^{n,n}))$$ lying above $X$ (in some rough path sense) such that:
\begin{itemize}
\item $X$ is a fBm with Hurst parameter $H$,
\item $\ga \in (1/3,H)$,
\item almost surely, $(\Xti^1,\Xti^2,\Xti^3)$ satisfies Hypothesis \ref{hyp:X2}, that is
$$\delti \Xti^1=0, \quad \delti \Xti^2=\Xti^1 \otimes X^1+\Xti^3,\quad \mbox{where} \ X^1_{ts}=\int_0^\infty d\xi \, \hphi(\xi) \, \Xti^1_{ts}(\xi).$$

\end{itemize}

\smallskip

As mentioned in the introduction, to this end, we shall resort to an approximation of the fBm introduced by Unterberger in \cite{Un}. Let us recall first briefly the definition of this approximation in the one-dimensional case.

\smallskip

All the processes we deal with in the sequel are defined on the same complete probability space $(\Omega,\cf,P)$. As shown in \cite{Un}, a simple explicit decomposition of  the covariance of the fBm allows to introduce  an analytic process $X^{+'}$ on the complex half-plane $\Pi^+=\lcl x+iy\in \mathbb{C}: y >0\rcl$ such that, if $X^{-'}$ is defined on $\Pi^-$ by $X_w^{-'}=\bar X_{\overline{w}}^{+'}$, then, for all $z,w \in \Pi^+$,
\beq\label{eq:58}
E\lc X_z^{+'}X_w^{+'}\rc =E\lc X_{\overline{z}}^{-'}X_{\overline{w}}^{-'}\rc =0, \quad E\lc X_z^{+'}X_{\overline{w}}^{-'}\rc =\frac{H(1-2H)}{2\cos(\pi H)} (-i(z-\overline{w}))^{2H-2}.
\eeq
The process $X_z^{+'}$ has to be interpreted as an analytic approximation of the derivative of the fBm, and the simple expression (\ref{eq:58}) for its covariance function is at the core of our further calculations. If one desires to construct an approximation of the fBm itself, just pick, for $t\in\R$ and any $\ep >0$, a continuous path $\ga_{\ep,t}:[0,1]\rightarrow \Pi^+$ such that $\ga_{\ep,t}(0)=i\ep$ and $\ga_{\ep,t}(1)=t+i\ep$. Set then $X_t^{+,\ep}=\int_{\ga_{\ep,t}}X_z^{+'} \, dz$. Likewise, a process $X_t^{-,\ep}$ can be defined as $X_t^{-,\ep}=\int_{\gati_{\ep,t}}X_z^{-'} \, dz$, where $\gati_{\ep,t}:[0,1]\rightarrow \Pi^-$ is such that $\gati_{\ep,t}(0)=-i\ep$ and $\gati_{\ep,t}(1)=t-i\ep$. Of course, $X_t^{-,\ep}=\bar X_t^{+,\ep}$, and the (real) approximation we shall work with is finally defined as
\beq\label{eq:59}
X_t^\ep=2 \,\mbox{Re}(X_t^{+,\ep})=X_t^{+,\ep}+X_t^{-,\ep}.
\eeq
The next proposition, borrowed from \cite{Un}, gives a first relation between the approximation we have just recalled and the usual fBm indexed by $\R$:
\begin{proposition}\label{prop:convergence-covariance}
Let $X^\ep$ be the process constructed above, given by relation (\ref{eq:59}). For all $s,t \in \R$, we have
$$\lim_{\ep \to 0}E\lc X_t^\ep X_s^\ep \rc =\frac{1}{2}\lcl |t|^{2H}+|s|^{2H}-\lln t-s\rrn^{2H}\rcl .$$
\end{proposition}
This statement in law will be improved at Theorem \ref{thm:5.2}. Just notice for the moment that, for any fixed $\ep>0$, $\lcl \Xep_t,t\in[0,T]\rcl$ is a smooth process and $(\Xep)'_t=X_{t+i\ep}^{+'}+X_{t-i\ep}^{-'}$.

\smallskip

In a natural way, the $n$-dimensional analog of our analytic approximation is a process $\Xep:=(X^{\ep,(1)}, \ldots, X^{\ep,(n)})$, where the components $X^{\ep,(i)}$ are constructed from independent copies of $X^{+'}$. We can then introduce the following smooth integral (in the Riemann sense) processes associated to $\Xep$:
\begin{align}\label{eq:59b}
&\Xtiunep_{ts}(\xi)=\int_s^t e^{-\xi(t-u)} d\Xep_u, \quad \Xunep_{ts}=\int_0^\infty d\xi \, \hphi(\xi) \, \Xtiunep_{ts}(\xi),  \\
&\Xtideep_{ts}(\xi)=\int_s^t e^{-\xi(t-u)} d\Xep_u \otimes \Xunep_{us}, \quad \Xtitrep_{tus}(\xi)=\int_u^t e^{-\xi(t-v)} d\Xep_v \otimes (\der \Xunep)_{vus}.  \nonumber
\end{align}
The main result of this section, which entails in particular our Theorem \ref{thm:1.1}, can be summarized as follows:
\begin{theorem}\label{thm:5.2}
Assume that $\hphi$ is such that $\int_0^\infty d\xi \, | \hphi(\xi)| (1+\xi^{2}) < \infty$. Then, for any $\ga \in (1/3,H)$, the sequence of processes $(\Xep,\Xtiunep,\Xtideep,\Xtitrep)$ converges in $L^1(\Omega; \cac_1^\ga \times \cacti_{2}^\ga \times \cacti_{2}^{2\ga} \times \cacti_{3}^{3\ga})$, as $\ep$ tends to $0$, to a process $(X,\Xti^1,\Xti^2,\Xti^3)$ . Furthermore, $X$ has the same law as a fBm with Hurst parameter $H$ and $\Xti^1,\Xti^2,\Xti^3$ satisfy Hypothesis 3.
\end{theorem}
 The proof of this theorem will be carried out in the sections below, and the main step in this process will be to prove that $(\Xep,\Xtiunep,\Xtideep,\Xtitrep)$ is a Cauchy sequence.  This is achieved once the following stronger statement is proved:
\begin{multline}\label{eq:60}
E \Big[ \cn[ \Xep-\Xet;\cac_1^\ga]+\cn[ \Xtiunep-\Xtiunet;\cacti_{1}^\ga ]\\
+\cn[ \Xtideep-\Xtideet;\cacti_{2}^{2\ga}]+\cn[ \Xtitrep-\Xtitret;\cacti_{3}^{3\ga}] \Big] \leq c \, \ep^\al,
\end{multline}
with $\al >0$. To do so, let us fix $0<\eta <\ep$ and follow the same lines as in the Brownian case (see Section \ref{sec:appli-usual-brownian}), which means that the issue mainly consists in estimating the moments of any order of the processes at stake.

\smallskip

For the sake of clarity, the proofs of the lemmas to come are carried over to the appendix. Let us also introduce the notation
\beq\label{eq:61}
\Xde=\Xep-\Xet, \quad \Xti^{i,\Delta}=\Xti^{i,\ep}-\Xti^{i,\eta}, \ i\in \lcl 1,2,3\rcl.
\eeq

\subsection{Estimation of the first order integrals} 
Our approximation results will stem from the association of Proposition \ref{prop:g-r-r} and the following lemma concerning Wiener integrals of analytic functions:

\begin{lemma}\label{lem:lemme-un}
Let $0\leq s <t \leq T$ and $f_{ts}$ an analytic function in a neighbourhood of
$$\Pi_{(s,t)}=\{ z=a+ib \in \mathbb{C}: \ a\in [s,t], \, b \in [-(t-s),t-s] \}$$
such that the restriction $f_{ts_{|[s,t]}}$ takes value in $\R$. Suppose that $f_{ts}$ is bounded on $\Pi_{(s,t)}$. Then, for any $\al \in (0,2H)$, 
$$E \lc\left\Vert\int_s^t f_{ts}(u) \, dX_u^\ep-\int_s^t f_{ts}(u) \, dX_u^\eta \right\Vert^2 \rc  \leq c_\al \norm{f}_{\infty, \Pi_{(s,t)}}^2 \lln t-s \rrn^{2H-\al} \lln \ep-\eta\rrn^\al,$$
where the constant $c_\al$ does not depend on $s,t,\ep,\eta$.
\end{lemma}

\smallskip

With this lemma in hand, our approximation result for the first order integrals based on $X^\ep$ can be written as:
\begin{proposition}\label{prop:5.2}
Let $\Xep$ and $\Xtiunep$ be the increments defined by (\ref{eq:59b}). Then there exists a constant $\al>0$ such that
$$
E \Big[ \cn[ \Xep-\Xet;\cac_1^\ga]+\cn[ \Xtiunep-\Xtiunet;\cacti_{2}^\ga ] \Big] \leq c \, \ep^\al,
$$
for any $0<\eta<\ep$.
\end{proposition}

\begin{proof}
Recall our notation (\ref{eq:61}) for the differences of increments based on $X$. For the estimation of $\Xde$, use the classical Garsia-Rumsey-Rodemich inequality to deduce $\norm{ (\der\Xde)_{ts} }$ $\leq c\lln t-s \rrn^\ga (U_{\ga,2p}^\Delta)^{1/2p}$, where
$$U_{\ga,2p}^\Delta =\iint_{0<v<w<T} \frac{\norm{ (\der\Xde)_{wv}}^{2p}}{|w-v|^{2\ga p+2}} dvdw,$$
which leads to $E\big[ \cn[\Xde;\cac_1^\ga]\big] \leq c \, E[U_{\ga,2p}^\Delta]^{1/2p}$. But, according to Lemma \ref{lem:lemme-un} (take $f_{wv}=1$), we have
$$E[\norm{(\der \Xde)_{wv}}^2]=E[ \norm{ (\Xep_w-\Xep_v)-(\Xet_w-\Xet_v)}^2]\leq c \lln \ep-\eta \rrn^{\al_1} \lln w-v\rrn^{2H-\al_1},$$
for an arbitrary constant $\al_1 \in (0,2H)$. Hence, since $X^+$ is a Gaussian process, we also obtain $E[\norm{(\der \Xde)_{wv}}^{2p}]\leq c \lln \ep-\eta \rrn^{\al_1p} \lln w-v\rrn^{2Hp-\al_1p}$, so that, if $\al_1 \in (2\ga,2H)$ and $p >2/(2H-2\ga-\al_1)$, $E\big[ \cn[\Xde;\cac_1^\ga]\big] \leq c \lln \ep-\eta\rrn^{\al_1 /2}$.

\smallskip

In order to estimate $\cn[\Xtiunde;\cacti_{1,\be}^\ga ]$, notice that we obviously have $\delti \Xtiunde =0$, which, thanks to Proposition \ref{prop:g-r-r}, gives $\norm{ \Xtiunde_{ts}(\xi)} \leq c \lln t-s\rrn^\ga (\Uti_{\ga,2p}^\Delta(\xi))^{1/2p}$, where
$$\Uti_{\ga,2p}^\Delta(\xi)=\iint_{0<v<w<T}\frac{\norm{\Xtiunde_{wv}(\xi)}^{2p}}{|w-v|^{2\ga p+2}} dvdw.$$
Therefore $E\big[ \cn[\Xtiunde;\cacti_{1}^\ga]\big] \leq c \int_0^\infty d\xi \, |\hphi(\xi)| (1+\xi) \, E[\Uti_{\ga,2p}^\Delta(\xi)]^{1/2p}$, and invoking again Lemma \ref{lem:lemme-un} with $f_{wv}(u)=e^{-\xi(w-u)}$, we get $E[ \norm{\Xtiunde_{wv}(\xi)}^{2p}] \leq c \lln \ep-\eta \rrn^{\al_1 p}\lln w-v\rrn^{2Hp-\al_1 p}$. Thus, just as in the case of $\Xde$, we end up with
$$E\big[\cn[\Xtiunde;\cacti_{1}^\ga ]\big] \leq c \lln \ep-\eta \rrn^{\al_1/2} \int_0^\infty d\xi \, |\hphi(\xi)|(1+\xi) \leq c \lln \ep-\eta \rrn^{\al_1/2},$$
which finishes the proof.

\end{proof}

\subsection{Estimation of the second order integrals} 
We now proceed to the estimation of the increments $\Xtideep$ and $\Xtitrep$, starting with the second one: 
\begin{proposition}\label{prop:5.3}
Let $\Xtitrep$ be the increment defined at (\ref{eq:59b}). Then there exists a constant $\al>0$ such that
$$
E \Big[ \cn[ \Xtitrep-\Xtitret;\cacti_3^{3\ga}] \Big] \leq c \, \ep^\al,
$$
for any $0<\eta<\ep$.
\end{proposition}

\begin{proof}
As in the Brownian case, write $\Xtitrep$ as
$$\Xtitrep_{tus}(\xi)=\int_0^\infty d\mu \, \hphi(\mu) \, \Xdetiquep_{tu}(\xi,\mu) \otimes \Xtiunep_{us}(\mu),$$
with $\Xdetiquep_{tu}(\xi,\mu)=\int_u^t e^{-\xi(t-v)}a_{vu}(\mu) \, d\Xep_v$. Then the increment $\Xtitrde$ defined at equation (\ref{eq:61}) satisfies:
\begin{eqnarray}\label{eq:63}
\Xtitrde_{tus}(\xi)&=&\int_0^\infty d\mu \, \hphi(\mu) \, \Xdetiqude_{tu}(\xi,\mu)\otimes \Xtiunep_{us}(\mu)+\int_0^\infty d\mu \, \hphi(\mu) \, \Xdetiquet_{tu}(\xi,\mu) \otimes \Xtiunde_{us}(\mu) \nonumber\\
&:=& ^{\sharp}\mbox{ }\negthickspace\negmedspace\Xtitrde_{tus}(\xi)+^{\flat}\negthickspace\negmedspace\Xtitrde_{tus}(\xi).
\end{eqnarray}
We will now bound these last two terms separately.

\smallskip

Let us start by controlling $^{\sharp}\negmedspace\Xtitrde$: notice that $(\deldeti\Xdetiqude)(\xi,\mu)=\Xtiunde(\xi)a(\mu)$, and thus, using the same arguments and notations (\ref{notation-cdeti-udeti}) as in the Brownian case, we get
$$\norm{ \Xdetiqude_{tu}(\xi,\mu)} \leq c \lln t-u\rrn^{2\ga} \lcl (\Udeti_{2\ga,2p'}^\Delta(\xi,\mu))^{1/2p'}+(\Cdeti_{\ga,2p,2p'}^\Delta(\xi,\mu))^{1/2p'} \rcl.$$
As a consequence, $\norm{\Xdetiqude_{tu}(\xi,\mu)\otimes \Xtiunep_{us}(\mu)} \leq c \lln t-u\rrn^{2\ga}\lln u-s\rrn^\ga \Rdeti_{\ga,2p,2p',2p''}^{\Delta,\ep}(\xi,\mu)$, where
\begin{eqnarray}\label{eq:64}
\Rdeti_{\ga,2p,2p',2p''}^{\Delta,\ep}(\xi,\mu)&=& \Udeti_{2\ga,2p'}^\Delta(\xi,\mu)^{1/2p'}\Uti_{\ga,2p''}^\ep(\mu)^{1/2p''}+\Cdeti_{\ga,2p,2p'}^\Delta(\xi,\mu)^{1/2p'}\Uti_{\ga,2p''}^\ep(\mu)^{1/2p''} \nonumber\\
&:=& ^1\negmedspace\Rdeti_{\ga,2p',2p''}^{\Delta,\ep}(\xi,\mu)+^2\negthickspace\Rdeti_{\ga,2p,2p',2p''}^{\Delta,\ep}(\xi,\mu),
\end{eqnarray}
which leads to
$$E\big[ \cn[^{\sharp}\negmedspace\Xtitrde;\cacti_{3,\be}^{3\ga} ]\big]
\leq \sum_{i=1}^2 \int_0^\infty d\xi \, |\hphi(\xi)|(1+\xi^\be)\int_0^\infty d\mu \, |\hphi(\mu)| \, E\big[ ^i\negmedspace\Rdeti_{\ga,2p',2p''}^{\Delta,\ep}(\xi,\mu)\big].$$

\smallskip

In order to estimate the first term of the latter sum, start with
\begin{equation}\label{eq:65}
E[ \Udeti_{2\ga,2p'}^\Delta(\xi,\mu)^{1/2p'}\Uti_{\ga,2p''}^\ep(\mu)^{1/2p''}] \leq E[\Udeti_{2\ga,2p'}^\Delta(\xi,\mu)]^{1/2p'} E[\Uti_{\ga,2p''}^\ep(\mu)]^{1/2p''}.
\end{equation}
Then use Lemma \ref{lem:lemme-un} to assert that, for any $\la \in (0,1)$ and any $\al \in (0,2H)$,
$$E\lc \norm{ \int_s^t e^{-\xi(t-u)}a_{us}(\mu) \, d\Xep_u-\int_s^t e^{-\xi(t-u)}a_{us}(\mu) \, d\Xet_u }^2\rc \leq c \, \mu^{2\la} \lln \ep-\eta \rrn^\al \lln t-s\rrn^{(2H-\al)+2\la}.$$
Indeed, if $z\in \Pi_{(s,t)}$,
$$\lln a_{zs}(\mu)\rrn =\lln e^{-\mu(z-s)}-1 \rrn \leq 2 \mu^\la \lln z-s\rrn^\la \leq 2\mu^\la \lln t+i(t-s)-s \rrn^\la \leq c \, \mu^\la \lln t-s \rrn^\la.$$
Accordingly, $E[\norm{\Xdetiqude_{wv}(\xi,\mu)}^{2p'}] \leq c \, \mu^{2p'\la_1}\lln \ep-\eta \rrn^{\al_2p'}\lln w-v\rrn^{(2H-\al_2+2\la_1)p'}$, and then, if we take $(\la_1,\al_2)\in (0,1)\times (0,2H)$ such that $2(H+\la_1)-\al_2-4\ga >0$ and $p'>2/(2(H+\la_1)-\al_2-4\ga)$, we get $E[\Udeti_{2\ga,2p'}^\Delta(\xi,\mu)]^{1/2p'} \leq c \, \mu^{\la_1}\lln \ep-\eta \rrn^{\al_2/2}$.

\smallskip

To deal with $E[\Uti_{\ga,2p''}^\ep(\mu)]$ in (\ref{eq:65}), consider the following estimation (recall that the proofs of all the lemmas in this section are postponed to the appendix):
\begin{lemma}\label{lem:lemme-trois}
Let $0<s<t<T$, $\mu \geq 0$. Then
$$E\lc \left\Vert \int_s^t e^{-\mu(t-u)} d\Xep_u \right\Vert^2\rc \leq c \, \lln t-s\rrn^{2H},$$
where the constant $c$ does not depend on $s,t,\mu,\ep$.
\end{lemma} 

Therefore, $E[\norm{\Xtiunep_{wv}(\mu)}^{2p''}]\leq c \, \lln w-v\rrn^{2Hp''}$, which, by taking $p''>1/(H-\ga)$, gives $E[\Uti_{\ga,2p''}^\ep(\mu)]\leq c$. We can thus assert that
$E[ ^1\negmedspace\Rdeti_{\ga,2p',2p''}^{\Delta,\ep}(\xi,\mu)] \leq c \, \mu^{\la_1}\lln \ep-\eta\rrn^{\al_2/2}$.

\smallskip

As far as $E[ ^2\negmedspace\,\Rdeti_{\ga,2p',2p''}^{\Delta,\ep}(\xi,\mu)]$ in (\ref{eq:64}) is concerned, go back to the definition of $\Cdeti_{\ga,2p,2p'}^\Delta$, together with the previous estimations of $E[\Uti_{\ga,2p}^\Delta(\xi)]$ and $E[\Uti_{\ga,2p''}^\ep(\mu)]$, to deduce
$$E\big[ ^2\negmedspace\Rdeti_{\ga,2p',2p''}^{\Delta,\ep}(\xi,\mu)\big] \leq  \mu^\ga E\big[ \Uti_{\ga,2p}^\Delta(\xi)\big]^{1/2p} E\big[\Uti_{\ga,2p''}^\ep(\mu)\big]^{1/2p''} \leq c \, \mu^\ga \lln \ep-\eta\rrn^{\al_1/2}.$$
Putting together the estimates on $^1\negmedspace\,\Rdeti$ and $^2\negmedspace\,\Rdeti$, we thus have proved that
\begin{eqnarray}\label{eq:66}
E\big[ \cn[^{\sharp}\negmedspace\Xtitrde;\cacti_{3}^{3\ga}]\big] &\leq& \lp \int_0^\infty d\xi \, |\hphi(\xi)|(1+\xi)\rp^2 \lcl \lln \ep-\eta\rrn^{\al_1/2}+\lln \ep-\eta\rrn^{\al_2/2}\rcl  \nonumber\\
 &\leq&c \lcl \lln \ep-\eta\rrn^{\al_1/2}+\lln \ep-\eta\rrn^{\al_2/2}\rcl.
\end{eqnarray}

Going back to equation (\ref{eq:63}), let us deal with the term $E[ \cn[^{\flat}\negthickspace\Xtitrde;\cacti_{3}^{3\ga}]]$. But in this latter case, it is readily checked  that the previous reasoning and bound (\ref{eq:66}) remain true by inverting the roles of $\Xdeti^4$ and $\Xti^1$, thanks to Lemma \ref{lem:lemme-un} and invoking the following lemma:
\begin{lemma}\label{lem:lemme-quatre}
Let $0<s<t<T$, $\xi,\mu \geq 0$. Then, for any $\la\in (0,1)$,
$$E\lc \norm{ \int_s^t e^{-\xi(t-u)}a_{us}(\mu)\, d\Xet_u}^2\rc \leq c \, \mu^{2\la}\lln t-s\rrn^{2H+2\la},$$
where the constant $c$ does not depend on $s,t,\xi,\mu,\eta$.
\end{lemma} 

This remark allows us to finally plug the bounds on $E[ \cn[^{\flat}\negthickspace\Xtitrde;\cacti_{3}^{3\ga}]]$, $E[ \cn[^{\sharp}\negthickspace\Xtitrde;\cacti_{3}^{3\ga}]]$ back into equation (\ref{eq:63}), and claim that the expected relation
$$E\big[\cn[\Xtitrde;\cacti_{3}^{3\ga}]\big] \leq c \lln \ep-\eta\rrn^\al \quad \mbox{for some} \ \al>0$$
holds true.

\end{proof}

The upper bound for $\Xtideep$ can be written in a similar way as for the previous cases:
\begin{proposition}\label{prop:5.4}
Let $\Xtideep$ be the increment defined at (\ref{eq:59b}). Then there exists a constant $\al>0$ such that
$$
E \Big[ \cn[ \Xtideep-\Xtideet;\cacti_2^{2\ga}] \Big] \leq c \, \ep^\al,
$$
for any $0<\eta<\ep$.
\end{proposition}

\begin{proof}
Here again, we will proceed as in the Brownian case of Section \ref{sec:appli-usual-brownian}, and we shall apply Proposition \ref{prop:g-r-r}. This means that we must control both the regularity of $\delti \Xtidede$ and the moments of $\Xtidede_{wv}$. However, $\delti \Xtidede=\Xtiunde\otimes \Xunep+\Xtiunet \otimes X^{1,\Delta}+\Xtitrde$, so that the previous estimations of Propositions \ref{prop:5.3} and \ref{prop:5.4} easily lead to $\norm{(\delti \Xtidede)_{tus}(\xi)} \leq \lln t-s \rrn^{2\ga} \Dti^{\Delta,\ep,\eta}(\xi)$, where $\Dti^{\Delta,\ep,\eta}$ satisfies $E[\cn[\Dti^{\Delta,\ep,\eta};\cl_1]] \leq c\lln \ep-\eta\rrn^\al$ for some $\al >0$. We have thus obtained that $\norm{\Xtidede_{ts}(\xi)}\leq c \lln t-s\rrn^{2\ga}\{ \Vti_{2\ga,4p}^\Delta(\xi)^{1/4p}+\Dti^{\Delta,\ep,\eta}(\xi) \}$, with
$$\Vti_{2\ga,4p}^\Delta(\xi)=\iint_{0<v<w<T} \frac{\norm{\Xtidede_{wv}(\xi)}^{4p}}{\lln w-v\rrn^{8\ga p+2}} dvdw,$$
and hence $E\big[\cn[\Xtidede;\cacti_{2}^{2\ga}]\big] \leq c \{ E\big[ \cn[(\Vti_{2\ga,4p}^\Delta)^{1/4p};\cl_1]\big]+\lln \ep-\eta \rrn^\al \}$.

\smallskip

To study $E\big[ \cn[(\Vti_{2\ga,4p}^\Delta)^{1/4p};\cl_1]\big]$, we first give a bound on the second moments of the increment $\Xtidede_{ts}(\xi)$:
\begin{lemma}\label{lem:lemme-cinq}
Let $0<s<t<T$, $\xi\geq 0$. Then, for any $\la\in (0,2H)$, there exists $\ka >0$ such that
$$E[\norm{\Xtidede_{ts}(\xi)}^2] \leq c  \, \ep^\ka  \lln t-s\rrn^{4H-\la} (1+\xi^2),$$
where the constant $c$ does not depend on $\xi,s,t,\ep,\eta$.
\end{lemma}
Since $\Xtidede_{ts}(\xi)$ is an element of the second chaos associated to the Gaussian process $X^+$, we can easily deduce from the previous lemma that $E[\norm{\Xtidede_{wv}(\xi)}^{4p}]\leq c \, \ep^{2p\ka} (1+\xi^2)^{2p} \lln w-v\rrn^{8Hp-2p\la}$. Thus, if we pick $\la \in (0,4(H-\ga))$ and $p>1/(4(H-\ga)-\la)$,
\begin{eqnarray*}
E\big[ \cn[(\Vti_{2\ga,4p}^\Delta)^{1/4p};\cl_1]\big] &\leq & \int_0^\infty d\xi \, |\hphi(\xi)|(1+\xi) \, E[\Vti_{2\ga,4p}^\Delta(\xi)]^{1/4p}\\
&\leq & c \, \ep^{\ka/2}\int_0^\infty d\xi \, |\hphi(\xi)|(1+\xi^{2}) \ \leq \ c \, \ep^{\ka/2}
\end{eqnarray*}
and the expected result $E[ \cn[\Xtidede;\cacti_{2}^{2\ga}]] \leq c \, \ep^\al$ holds true.

\end{proof}

\smallskip

We can now conclude this section with the proof of our main theorem.

\begin{proof}[Proof of Theorem \ref{thm:5.2}]
Putting together Propositions \ref{prop:5.2}, \ref{prop:5.3} and \ref{prop:5.4}, it is readily checked that $(\Xep,\Xtiunep,\Xtideep,\Xtitrep)$ is a Cauchy sequence in the Banach space $L^1(\Omega;\cac_1^\ga \times \cacti_{2}^\ga \times \cacti_{2}^{2\ga}\times \cacti_{3}^{3\ga})$.

\smallskip

The fact that the process $\lcl X_t,t\in [0,T]\rcl$ has the same law as a fBm with Hurst parameter $H$ is a direct consequence of Proposition \ref{prop:convergence-covariance}.

\smallskip

Finally, it is readily checked that the algebraic relations $\delti \Xtiunep=0$ is preserved as $\ep$ tends to $0$, by taking $L^1(\Omega)$-limits on both sides of the equality. The same kind of limit can be also taken for the relation $\delti \Xtideep=\Xtiunep \otimes X^{1,\ep}+\Xtitrep$,  provided one can prove that $\Xtiunep$ and $X^{1,\ep}$ are in fact Cauchy sequences in $L^2(\Omega)$. But this is achieved by a slight elaboration of Proposition \ref{prop:5.2}, thanks to the fact that $X^{1,\ep}$ and $\Xtiunep$ are Gaussian processes.

\end{proof}

%%%%%%%%%%%%%%%%%%%%%%%%%%%%%%%%%%%%%%%%%%%
%%%%%%%%%%%%%%%%%%%%%%%%%%%%%%%%%%%%%%%%%%%

\section{Appendix}

\subsection{Proofs of the GRR type propositions}
This section gathers the proofs of all the general results we need for the regularity of the stochastic processes handled in this article.

\begin{proof}[Proof of Proposition \ref{prop:g-r-r}]
This is an adaptation of Stroock's proof of the (classical) Garsia-Rodemich-Rumsey inequality (see \cite{St}).

\smallskip

Let $s,t \in [0,T]$ and notice that, for any sequence of decreasing times $(s_k)\in (s,t)$, 
\begin{equation}\label{proof-g-r-r-1}
(\delti \Rti)_{s_k s_{k+1}s}(\xi)=\Rti_{s_k s}(\xi)-\Rti_{s_ks_{k+1}}(\xi)-e^{-\xi(s_k-s_{k+1})}\Rti_{s_{k+1}s}(\xi),
\end{equation}
so that $\norm{ \Rti_{s_k s}(\xi)} \leq \norm{ \Rti_{s_{k+1}s}(\xi)}+\norm{\Rti_{s_ks_{k+1}}(\xi)}+\norm{(\delti \Rti)_{s_ks_{k+1}s}(\xi)}$ and by iteration,
\begin{equation}\label{decompo-g-r-r}
\norm{ \Rti_{s_0 s}(\xi)} \leq \norm{ \Rti_{s_{n+1}s}(\xi)}+\sum_{k=0}^n \lc \norm{\Rti_{s_k s_{k+1}}(\xi)}+\norm{(\delti \Rti)_{s_ks_{k+1}s}(\xi)} \rc.
\end{equation}

For all $v \geq s$, set $I(v)=\int_s^v \psi\lp\frac{\norm{\Rti_{v}(\xi)|}}{\phi(v-u)}\rp du$, and define the sequence $(s_k)$ as follows. First, fix $s_0 \in (s,t)$ arbitrarily. Next, given $s_k \in (s,t)$, write $s_k=s+\la_k$ ($\la_k \in (0,t-s)$) and define $\al_k < \la_k$ by the relation $2\phi(\al_k)=\phi(\la_k)$. Then set $s_{k+1}=s+\la_{k+1}$, where $\la_{k+1} \in (0,\al_k)$ is such that
\beq\label{eq:69}
I(s+\la_{k+1})\leq \frac{2\, \Uti(\xi)}{\al_k} \quad \mbox{and} \quad \psi\lp \frac{\norm{\Rti_{s_k,s+\la_{k+1}}(\xi)}}{\phi(\la_k-\la_{k+1})}\rp \leq \frac{ 2\, I(s_k)}{\al_k}.
\eeq
Such an element always exists since if we call $A_k \ (\mbox{resp.} \ B_k) \subset (0,\al_k)$ the set on which the first (resp. the second) inequality fails, we have
$\Uti(\xi) \geq \int_{A_k} I(s+u)\, du > \frac{2\, \Uti(\xi)}{\al_k} \mu(A_k)$, whereas
$$I(s_k)=\int_0^{\la_k} \psi\lp \frac{\norm{\Rti_{s_k,s+u}(\xi)}}{\phi(\la_k-u)}\rp du \geq \int_{B_k} \psi\lp \frac{\norm{\Rti_{s_k,s+u}(\xi)}}{\phi(\la_k-u)}\rp du > \frac{2\, I(s_k)}{\al_k} \mu(B_k).$$
The last two inequalities yield $\mu(A_k)<\al_k/2$ and $\mu(B_k)<\al_k/2$,
and thus $\mu(A_k \cup B_k) < \al_k$. It is then clear that $(s_k)$ decreases to $s$. 

\smallskip

Observe now that
\begin{multline*}
\phi(\la_k-\la_{k+1}) \leq \phi(\la_k)=2\phi(\al_k)=4\lp \phi(\al_k)-\frac{\phi(\al_k)}{2}\rp\\
\leq 4\lp \phi(\al_k)-\frac{\phi(\la_{k+1})}{2}\rp=4(\phi(\al_k)-\phi(\al_{k+1})).
\end{multline*}
Plugging this observation into equation (\ref{eq:69}), we end up with:
\begin{eqnarray*}
\norm{\Rti_{s_ks_{k+1}}(\xi)} &\leq& \phi(\la_k-\la_{k+1})\psi^{-1}\lp \frac{2\, I(s_k)}{\al_k}\rp \leq 4(\phi(\al_k)-\phi(\al_{k+1}))\psi^{-1}\lp \frac{4\, \Uti(\xi)}{\al_k \al_{k-1}}\rp\\
&\leq& 4\int_{\al_{k+1}}^{\al_k} \psi^{-1}\lp \frac{4\, \Uti(\xi)}{r^2}\rp d\phi(r),
\end{eqnarray*}
where we have used the fact that $\psi^{-1}$ is an increasing function. Besides, condition (\ref{eq:27}) entails:
\begin{eqnarray*}
\norm{(\delti \Rti)_{s_ks_{k+1}s}(\xi)} &\leq&  \psi^{-1}\lp \frac{4\, \Cti(\xi)}{\la_k^2}\rp \phi(\la_k)\leq  4 \,\psi^{-1}\lp \frac{4\, \Cti(\xi)}{\la_k^2}\rp (\phi(\al_k)-\phi(\al_{k+1})) \\
 &\leq&  4\int_{\al_{k+1}}^{\al_k} \psi^{-1}\lp \frac{4\, \Cti(\xi)}{r^2}\rp d\phi(r).
\end{eqnarray*}
As $\Rti \, (\xi) \in \cac_2$, we get, by letting $n$ tend to infinity in (\ref{decompo-g-r-r}),
$$\norm{\Rti_{s_0s}(\xi)} \leq 4\int_0^{\lln t-s \rrn} \lc \psi^{-1}\lp \frac{4\, \Uti(\xi)}{r^2}\rp +\psi^{-1}\lp \frac{4\, \Cti(\xi)}{r^2}\rp \rc d\phi(r).$$
In the same way, we find 
$$\norm{\Rti_{ts_0}(\xi)} \leq 4\int_0^{\lln t-s \rrn} \lc \psi^{-1}\lp \frac{4\, \Uti(\xi)}{r^2}\rp +\psi^{-1}\lp \frac{4\, \Cti(\xi)}{r^2}\rp \rc d\phi(r).$$
Write now 
\begin{equation} \label{proof-g-r-r-2}
\Rti_{ts}(\xi)=\Rti_{ts_0}(\xi)+e^{-\xi(t-s_0)}\Rti_{s_0s}(\xi)+(\delti \Rti)_{ts_0s}
\end{equation}
to deduce
\bean
\norm{ \Rti_{ts}(\xi)} & \leq & \norm{ \Rti_{ts_0}(\xi)}+\norm{\Rti_{s_0s}(\xi)}+\norm{(\delti \Rti)_{ts_0s}(\xi)}\\
&\leq & 8\int_0^{\lln t-s \rrn} \lc \psi^{-1}\lp \frac{4\, \Uti(\xi)}{r^2}\rp +\psi^{-1}\lp \frac{4\, \Cti(\xi)}{r^2}\rp \rc d\phi(r)+\norm{(\delti \Rti)_{ts_0s}(\xi)},
\eean
and observe that
$$\norm{(\delti \Rti)_{ts_0s}(\xi)} \leq \psi^{-1}\lp \frac{4\, \Cti(\xi)}{\lln t-s\rrn^2}\rp \phi(\lln t-s\rrn) \leq \int_0^{\lln t-s\rrn} \psi^{-1}\lp \frac{4\, \Cti(\xi)}{r^2}\rp d\phi(r),$$
which achieves the proof.

\end{proof}

\smallskip

We also need to prove a slight extension of the previous proposition to functions indexed by two Laplace variables:

\begin{proof}[Proof of Proposition \ref{prop:4.19}]
It follows the same lines as the proof of Proposition \ref{prop:g-r-r}. Relation (\ref{proof-g-r-r-1}) has to be replaced with
$$(\deldeti \Rdeti)_{s_k s_{k+1}s}(\xi,\eta)=\Rdeti_{s_ks}(\xi,\eta)-\Rdeti_{s_ks_{k+1}}(\xi,\eta)e^{-\eta(s_{k+1}-s)}-e^{-\xi(s_k-s_{k+1})} \Rdeti_{s_{k+1}s}(\xi,\eta),$$
which leads to the expected estimation
$$\norm{ \Rdeti_{s_k s}(\xi,\eta)}\leq \norm{ \Rdeti_{s_{k+1}s}(\xi,\eta)}+\norm{ \Rdeti_{s_k s_{k+1}}(\xi,\eta)}+\norm{ (\deldeti \Rdeti)_{s_k s_{k+1}s}(\xi,\eta)},$$
whereas (\ref{proof-g-r-r-2}) becomes
$$\Rdeti_{ts}(\xi,\eta)=e^{-\eta(s_0-s)}\Rdeti_{ts_0}(\xi,\eta)+e^{-\xi(t-s_0)} \Rdeti_{s_0s}(\xi,\eta)+(\deldeti \Rdeti)_{ts_0s}(\xi,\eta)$$
and thus $\norm{ \Rdeti_{ts}(\xi,\eta)} \leq \norm{ \Rdeti_{ts_0}(\xi,\eta)}+\norm{ \Rdeti_{s_0s}(\xi,\eta)}+\norm{ (\deldeti \Rdeti)_{ts_0s}(\xi,\eta)}$.

\end{proof}

\subsection{Proofs of the complex analysis lemmas}
We will prove in this section Lemmas \ref{lem:lemme-un}, \ref{lem:lemme-trois} and \ref{lem:lemme-quatre}. The key ingredients for those proofs are the following elementary estimations:

\begin{lemma}\label{lem:lemme-de-base}
Let $0\leq s <t \leq T$ and $\ga_{(s,t)}$ the three-part path in $\Pi^+$
\begin{equation}\label{definition-gamma-s-t}
\ga_{(s,t)}=[s,s+i(t-s)] \cup [s+i(t-s),t+i(t-s)] \cup [t+i(t-s),t].
\end{equation}
Then, for any $\ep >0$ and any $\al \in (0,1)$, 
\begin{equation}\label{estimation-noyau-simple}
\int_{\ga_{(s,t)}} \lln dz \rrn \int_{\overline{\ga_{(s,t)}}} \lln dw \rrn \lln -i(z-w)+\ep \rrn^{\al-2} \leq c \lln t-s \rrn^\al,
\end{equation}
where the constant $c$ does not depend on $s,t,\ep$. Moreover, for any $\eta >0$ and $\la \in (0,\al)$,
\begin{equation}\label{estimation-noyau-double}
\int_{\ga_{(s,t)}} \lln dz \rrn \int_{\overline{\ga_{(s,t)}}} \lln dw \rrn \lln (-i(z-w)+\ep)^{\al-2}-(-i(z-w)+\eta)^{\al-2} \rrn \leq c \lln t-s \rrn^{\al-\la} \lln \ep-\eta\rrn^\la,
\end{equation}
where the constant $c$ does not depend on $s,t,\ep,\eta$.
\end{lemma}

\begin{proof}
Denote $\ga^1_{(s,t)}=[s,s+i(t-s)]$, $\ga^2_{(s,t)}=[s+i(t-s),t+i(t-s)]$ and $\ga_{(s,t)}^3=[t+i(t-s),t]$ (see the figure below), so that
$$\int_{\ga_{(s,t)}} \lln dz \rrn \int_{\overline{\ga_{(s,t)}}} \lln dw \rrn \lln -i(z-w)+\ep \rrn^{\al-2}
=\sum_{1\leq j,k\leq 3}\int_{\ga_{(s,t)}^j} \lln dz \rrn \int_{\overline{\ga_{(s,t)}^k}} \lln dw \rrn \lln -i(z-w)+\ep \rrn^{\al-2}.$$

\begin{figure}[!h]\label{figure-un}
\centering
\begin{pspicture}(0,0)(10,10)
%\psgrid
\psline{->}(1,1)(1,9) 
\psline{->}(0.2,5)(9,5) 
\psline[linestyle=dashed,ArrowInside=->,ArrowInsidePos=0.5](2.4,5)(2.4,8)(5.4,8)(5.4,5) 
\psline[linestyle=dashed,ArrowInside=->,ArrowInsidePos=0.5](2.4,5)(2.4,2)(5.4,2)(5.4,5) 
%\psline[ArrowInside=->,ArrowInsidePos=0.5](2.4,5)(5.4,5)  
\psdots[dotstyle=+](1,8)
\psdots[dotstyle=+](1,2)  
\rput(0,2){$-i(t-s)$}
\rput(0.2,8){$i(t-s)$}    
\rput(2.2,4.8){$s$} 
\rput(5.6,4.8){$t$} 
\rput(2.7,6.5){$\ga^1$}   
\rput(3.8,7.7){$\ga^2$}  
\rput(5.1,6.5){$\ga^3$} 
\end{pspicture}
\caption{Contours of integration}
\end{figure}
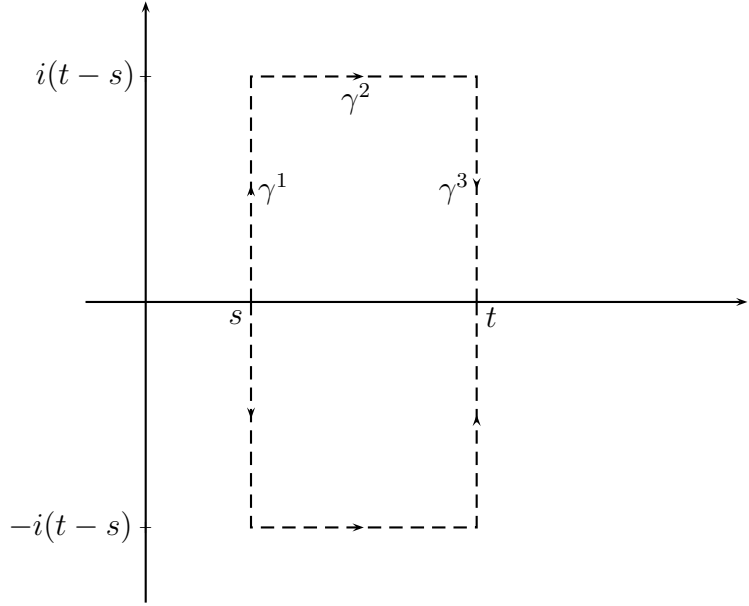

Let us now estimate each term of the latter sum. Notice first that, for any $\al>0$, the integral $\int_0^1  ( u+v)^{\al-2} dudv$ is finite. This allows to obtain:

\smallskip

\noindent
{\it Case} $j=k=1$:
\begin{align*}
&\int_{\ga_{(s,t)}^1} \lln dz \rrn \int_{\overline{\ga_{(s,t)}^1}} \lln dw \rrn \lln -i(z-w)+\ep \rrn^{\al-2}= \int_0^{t-s} \int_0^{t-s}  ( u+v+\ep)^{\al-2}dudv\\
&\leq  \int_0^{t-s}\int_0^{t-s} ( u+v)^{\al-2} dudv  
 \leq  \lln t-s\rrn^\al \int_0^1 \int_0^1  ( u+v)^{\al-2} dudv 
 \leq c \, \lln t-s\rrn^\al.
\end{align*}

\smallskip

\noindent
{\it Case} $j=k=2$:
\bean
\int_{\ga_{(s,t)}^2} \lln dz \rrn \int_{\overline{\ga_{(s,t)}^2}} \lln dw \rrn \lln -i(z-w)+\ep \rrn^{\al-2} &=&  \int_0^{t-s}\int_0^{t-s} \lln 2(t-s)+\ep-i(u-v)\rrn^{\al-2} dudv \\
&\leq &  \lln t-s\rrn^\al.
\eean

\smallskip

\noindent
{\it Case} $j=k=3$:
\begin{align*}
& \int_{\ga_{(s,t)}^3} \lln dz \rrn \int_{\overline{\ga_{(s,t)}^3}} \lln dw \rrn \lln -i(z-w)+\ep \rrn^{\al-2} \\
&=  \int_0^{t-s} \int_0^{t-s} \lln (t-s-u)+(t-s-v)+\ep\rrn^{\al-2}dudv 
\leq  \int_0^{t-s} \int_0^{t-s} \lln u+v\rrn^{\al-2} dudv \\ 
&\leq \ c \lln t-s\rrn^\al.
\end{align*}

\smallskip

\noindent
{\it Case} $j=1,k=2$:
\bean
\int_{\ga_{(s,t)}^1} \lln dz \rrn \int_{\overline{\ga_{(s,t)}^2}} \lln dw \rrn \lln -i(z-w)+\ep \rrn^{\al-2} &=&  \int_0^{t-s}\int_0^{t-s} \lln (u+t-s)-iv\rrn^{\al-2} du dv\\
&\leq &  \lln t-s\rrn^\al.
\eean

\smallskip

\noindent
{\it Case} $j=1,k=3$:
\begin{multline*}
\int_{\ga_{(s,t)}^1} \lln dz \rrn \int_{\overline{\ga_{(s,t)}^3}} \lln dw \rrn \lln -i(z-w)+\ep \rrn^{\al-2} \\
=  \int_0^{t-s}\int_0^{t-s} \lln -i(s-t)+(u-v+t-s)\rrn^{\al-2} dudv 
 \leq \ \lln t-s \rrn^\al.
\end{multline*}

%\begin{align*}
%&\int_{\ga_{(s,t)}^1} \lln dz \rrn \int_{\overline{\ga_{(s,t)}^3}} \lln dw \rrn \lln -i(z-w)+\ep \rrn^{\al-2} \\
%&=  \int_0^{t-s}\int_0^{t-s} \lln -i(s-t)+(u-v+t-s)\rrn^{\al-2} dudv 
% \leq \ \lln t-s \rrn^\al.
%\end{align*}

\smallskip

\noindent
{\it Case} $j=2,k=3$:
\begin{multline*}
\int_{\ga_{(s,t)}^2} \lln dz \rrn \int_{\overline{\ga_{(s,t)}^3}} \lln dw \rrn \lln -i(z-w)+\ep \rrn^{\al-2}  \\ 
=  \int_0^{t-s}\int_0^{t-s} \lln 2(t-s)-v-i(s-t+u)\rrn^{\al-2} dudv 
\leq   \lln t-s \rrn^\al.
\end{multline*}

It is clear that the other cases can be dealt with in the same way, which achieves the proof of (\ref{estimation-noyau-simple}). As for (\ref{estimation-noyau-double}), notice that, if $z\in \ga_{(s,t)}$ and $w \in \overline{\ga_{(s,t)}}$,
\begin{multline*}
\lln (-i(z-w)+\ep)^{\al-2}-(-i(z-w)+\eta)^{\al-2} \rrn, \\
\leq (2\sup_{x\in [\eta,\ep]} \lln -i(z-w)+x \rrn^{\al-2})^{1-\la} (\lln \al-2\rrn \sup_{x\in [\eta,\ep]} \lln -i(z-w)+x \rrn^{\al-3})^\la \lln \ep-\eta \rrn^\la\\
\leq c \lln -i(z-w)+\eta \rrn^{(\al-\la)-2} \lln \ep-\eta\rrn^\la,
\end{multline*}
so that the result is a consequence of (\ref{estimation-noyau-simple}).

\end{proof}

\begin{remark}
As mentioned in \cite{Un}, one of the interesting features of the complex analysis approach for the stochastic calculus with respect to fBm is that simple deformations of contour like (\ref{definition-gamma-s-t}) allow to transform very singular kernels like $(u-v)^{\al-2}$ into a much more tractable term of the form $(u+v)^{\al-2}$.
\end{remark}

We can now begin with the proof of our lemmas, after introducing an additional notation:
for the sake of conciseness, we shall henceforth denote 
\beq\label{eq:75}
K_{\ep_1,\ep_2}^\al(x,y)=(-i(x-y)+\ep_1+\ep_2)^\al.
\eeq

\begin{proof}[Proof of Lemma \ref{lem:lemme-un}]
Write the left-hand-side as
\begin{multline*}
A_{ts}(\ep,\eta):= E\lc \norm{ \int_s^t f_{ts}(u)\, dX^{\ep}_u-\int_s^t f_{ts}(u)\, dX^{\eta}_v }^2 \rc \\
= n\int_s^t \int_s^t f_{ts}(u)f_{ts}(v)\, E \lc \lp (X^{\ep,(1)})_u'-(X^{\eta,(1)})_u' \rp \lp (X^{\ep,(1)})'_v-(X^{\eta,(1)})_v'\rp \rc du dv.
\end{multline*}
Notice then that
\begin{align*}
&E\lc (X^{\ep,(1)})'_u (X^{\ep,(1)})'_v \rc\\
&=E\lc \lp X^{'+,(1)}(u+i\ep)+X^{'-,(1)}(u-i\ep)\rp \lp X^{'+,(1)}(v+i\ep)+X^{'-,(1)}(v-i\ep)\rp\rc \\
&= c_H \lcl K_{\ep,\ep}^{2H-2}(u,v)+K_{\ep,\ep}^{2H-2}(v,u) \rcl,
\end{align*}
where $c_H=\frac{H(1-2H)}{2\cos(\pi H)}$, and
$$E\lc (X^{\ep,(1)})'_u (X^{\eta,(1)})_v'\rc =E\lc (X^{\eta,(1)})_u'(X^{\ep,(1)})_v' \rc 
=c_H \lcl K_{\ep,\eta}^{2H-2}(u,v)+K_{\ep,\eta}^{2H-2}(v,u) \rcl.$$
Using these identities, one can decompose $A_{ts}(\ep,\eta)$ into a sum of terms whose prototype is:
$$
A^1_{ts}(\ep,\eta)=\int_s^t du\int_s^t dv \, f_{ts}(u) f_{ts}(v) [K_{\ep,\ep}^{2H-2}(u,v)-K_{\ep,\eta}^{2H-2}(u,v)].
$$ 
Let us focus then on the estimation of this last term: by a deformation of contour, we get (remember that $\ga_{(s,t)}$ is defined by (\ref{definition-gamma-s-t}))
$$A_{ts}^1(\ep,\eta)=\int_{\ga_{(s,t)}} dz \int_{\overline{\ga_{(s,t)}}} dw \, f_{ts}(z) f_{ts}(w) [K_{\ep,\ep}^{2H-2}(z,w)-K_{\ep,\eta}^{2H-2}(z,w)],$$
hence, owing to (\ref{estimation-noyau-double}), it is easily seen that:
\begin{multline*}
\lln A_{ts}^1(\ep,\eta) \rrn \leq \norm{f}_{\infty,\Pi_{(s,t)}}^2 \int_{\ga_{(s,t)}} \lln dz\rrn \int_{\overline{\ga_{(s,t)}}} \lln dw \rrn \lln  K_{\ep,\ep}^{2H-2}(z,w)-K_{\ep,\eta}^{2H-2}(z,w)\rrn\\
\leq c \norm{f}_{\infty,\Pi_{(s,t)}}^2 \lln \ep-\eta \rrn^\al \lln t-s \rrn^{2H-\al}.
\end{multline*}
Clearly, this argument remains true for the other terms composing $A_{ts}(\ep,\eta)$, which achieves the proof.

\end{proof}

\begin{proof}[Proof of Lemma \ref{lem:lemme-trois}]
We have
$$E \lc \norm{\int_s^t e^{-\mu(t-u)} d\Xep_u }^2 \rc =n \int_s^t du \int_s^t dv \, e^{-\mu(t-u)}e^{-\mu(t-v)} E[(\Xepun)'_u(\Xepun)'_v],$$
with $E[(\Xepun)'_u(\Xepun)'_v]=c_H \lcl K_{\ep,\ep}^{2H-2}(u,v)+K_{\ep,\ep}^{2H-2}(v,u)\rcl$, which gives, by an argument of symetry,
$$E \lc \norm{\int_s^t e^{-\mu(t-u)} d\Xep_u }^2 \rc =c \int_s^t du \int_s^t dv \, e^{-\mu(t-u)}e^{-\mu(t-v)} K_{\ep,\ep}^{2H-2}(u,v).$$
In the latter integral, deform the line $[s,t]$ into $\ga_{(s,t)}$ for $u$ and $\overline{\ga_{(s,t)}}$ for $v$. The result is then a consequence of (\ref{estimation-noyau-simple}).
\end{proof}

\begin{proof}[Proof of Lemma \ref{lem:lemme-quatre}]
It goes along the same lines as the previous proof, taking into account the fact that if $z \in \ga_{(s,t)} \cup \overline{\ga_{(s,t)}}$,
$$\lln a_{zs}(\mu)\rrn \leq \lln e^{-\mu(z-s)}-1\rrn^\la \leq \mu^\la \lln z-s \rrn^\la \leq \mu^\la \lln s+i(t-s)-s \rrn^\la \leq c \, \mu^\la \lln t-s \rrn^\la.$$

\end{proof}

\subsection{Proof of Lemma \ref{lem:lemme-cinq}}
As one might expect, the estimation of the variance of the convolutional Levy area $\Xti_{ts}^{2,\Delta}$ (defined by equation (\ref{eq:61})) gives rise to more intricate calculations. We shall distinguish the diagonal and non-diagonal terms, respectively denoted by $\Xti_{ts}^{2,\Delta,(1,1)}$ and $\Xti_{ts}^{2,\Delta,(1,2)}$, and use the fact that $$E [ \norm{\Xti_{ts}^{2,\Delta}(\xi)}^2] \leq c \lcl E [ \norm{\Xti_{ts}^{2,\Delta,(1,1)}(\xi)}^2] +E [ \norm{\Xti_{ts}^{2,\Delta,(1,2)}(\xi)}^2 ] \rcl.$$

\smallskip

\noindent
\textbf{The diagonal term $\mathbf{\Xti_{ts}^{2,\Delta,(1,1)}}$:} For the sake of conciseness, denote $\Xep:=\Xepun$. The following property of $\Xep$ will be useful:
\begin{lemma}
Let $\ep,\eta >0$, $0\leq s ,t \leq T$. Then, for any $\al \in (0,2H)$,
\begin{equation}\label{majo-de-base}
E[ (\Xep_t-\Xep_s )^2 ] \leq c_1 \lln t-s \rrn^{2H} \quad , \quad E[ (\Xep_t-\Xet_t )^2 ] \leq c_2 \lln \ep-\eta \rrn^{\al},
\end{equation} 
\begin{equation} \label{majo-double-increment}
E[ ( \lcl \Xep_t-\Xep_s\rcl^2-\lcl \Xet_t-\Xet_s \rcl^2 )^2 ] \leq c_3 \lln t-s \rrn^{4H-\al} \lln \ep-\eta \rrn^{\al^2/4H},
\end{equation}
where the constants $c_1,c_2,c_3$ do not depend on $\ep,\eta,s,t,\al$.
\end{lemma}
\begin{proof}
The first (resp. second) inequality is a direct consequence of Lemma \ref{lem:lemme-trois} (resp. Lemma \ref{lem:lemme-un}). Those two inequalities account for (\ref{majo-double-increment}). Indeed, we have, on the one hand, 
$$E[ ( \lcl \Xep_t-\Xep_s\rcl^2-\lcl \Xet_t-\Xet_s \rcl^2 )^2 ] \leq c \lcl E[ (\Xep_t-\Xep_s)^4] +E[ (\Xet_t-\Xet_s)^4 ] \rcl\leq c \lln t-s \rrn^{4H}.$$
On the other hand, using Hölder's inequality,
\begin{multline}\label{eq:78}
E[ ( \lcl \Xep_t-\Xep_s\rcl^2-\lcl \Xet_t-\Xet_s \rcl^2 )^2 ]\\
 =E[ \lcl (\Xep_t-\Xet_t)-(\Xep_s-\Xet_s)\rcl^2\lcl \Xep_t-\Xep_s+\Xet_t-\Xet_s \rcl^2] \\
\leq c \sqrt{E[ \lcl (\Xep_t-\Xet_t)-(\Xep_s-\Xet_s)\rcl^4] },
\end{multline}
and thus
\begin{eqnarray}\label{eq:79}
E[ ( \lcl \Xep_t-\Xep_s\rcl^2-\lcl \Xet_t-\Xet_s \rcl^2 )^2 ]
&\leq& c \sqrt{ E[ (\Xep_t-\Xet_t)^4]+E[ (\Xep_s-\Xet_s)^4 ] } \nonumber\\
&\leq& c \lln \ep-\eta \rrn^{\al}.
\end{eqnarray}
Hence for any $\la \in (0,1)$, thanks to (\ref{eq:78}) and (\ref{eq:79}), we end up with
\begin{align*}
&E\lc ( \lcl \Xep_t-\Xep_s\rcl^2-\lcl \Xet_t-\Xet_s \rcl^2 )^2 \rc \\
&=E^{\la} \lc ( \lcl \Xep_t-\Xep_s\rcl^2-\lcl \Xet_t-\Xet_s \rcl^2 )^2 \rc 
E^{1-\la} \lc ( \lcl \Xep_t-\Xep_s\rcl^2-\lcl \Xet_t-\Xet_s \rcl^2 )^2 \rc \\
&\leq c \lln t-s \rrn^{4H\la} \lln \ep-\eta \rrn^{\al(1-\la)},
\end{align*}
which gives the result if we take $\la=1-\frac{\al}{4H}$.

\end{proof}

\smallskip

Recall from (\ref{eq:59b}) that we have $\Xti_{ts}^{2,\ep,(1,1)}(\xi)=\int_0^\infty d\mu \, \hphi(\mu) \, \Xdeti_{ts}^{2,\ep,(1,1)}(\xi,\mu)$, with
\beq\label{eq:80}
\Xdeti_{ts}^{2,\ep,(1,1)}(\xi,\mu)= 
\int_s^t du \, e^{-\xi(t-u)} dX_u^\ep \int_s^u e^{-\mu(u-v)} \, dX_v^{\ep}.
\eeq
Our main effort will of course concern the estimation of $\Xdeti_{ts}^{2,\Delta,(1,1)}$. In the absence of exponential weights like $e^{-\xi(t-u)}$, namely in the case of the usual Levy area $A_{ts}^{\ep}=\int_s^t  dX_u^\ep \int_s^u  dX_v^{\ep}$, the strategy is obvious (see \cite{Un}): one can compute explicitly $A_{ts}^{\ep}=\frac12 (X_t^\ep-X_s^{\ep})^2$, from which all the useful bounds can be easily deduced. The situation is less simple here due to our exponential weights, but we will try to mimic the classical situation with a natural trick: integrate the exponential weights by parts.

\smallskip

More specifically, for a fixed $\ep>0$, since $X^\ep$ is a smooth process, it is easily derived from equation (\ref{eq:80}) that 
\begin{multline*}
\Xdeti_{ts}^{2,\ep,(1,1)}(\xi,\mu)=\frac{1}{2} \lcl (\Xep_t)^2-2e^{-\mu(t-s)}\Xep_s \Xep_t+ e^{-\xi(t-s)}(\Xep_s)^2 \rcl  \\
+ \Rdeti_{ts}^{\ep,1}(\xi,\mu)+\Rti_{ts}^{\ep,2}(\xi)+\Rdeti_{ts}^{\ep,3}(\xi,\mu),
\end{multline*}
where the remainders $\Rdeti^{\ep,1},\Rti^{\ep,2}$ and $\Rdeti^{\ep,3}$ are defined by:
\bean
\Rdeti_{ts}^{\ep,1}(\xi,\mu)&=&-\mu \int_s^t du \, e^{-\xi(t-u)} (\Xep)_u' \int_s^u dv \, e^{-\mu(u-v)}\Xep_v \\ 
\Rti_{ts}^{\ep,2}(\xi)&=&-\frac{\xi}{2}\int_s^t du \, e^{-\xi(t-u)}(\Xep_u)^2  \\
\Rdeti_{ts}^{\ep,3}(\xi,\mu)&=&(\xi-\mu)  \Xep_s \int_s^t du \, e^{-\xi(t-u)}e^{-\mu(u-s)} \Xep_u.
\eean
A little more elementary algebraic manipulations yield then:
\begin{multline}\label{decompo-xdeti-1-2}
\Xdeti_{ts}^{2,\Delta,(1,1)}(\xi,\mu)=\frac{1}{2} \lcl (\Xep_t-\Xep_s)^2-(\Xet_t-\Xet_s)^2 \rcl +(1-e^{-\mu(t-s)}) \lcl \Xep_s\Xep_t-\Xet_s\Xet_t \rcl\\
+\frac{1}{2} (e^{-\xi(t-s)}-1) \lcl (\Xep_s)^2-(\Xet_s)^2 \rcl+\Rdeti_{ts}^\Delta(\xi,\mu)
:=\sum_{j=1}^{4} K_j,
\end{multline}
where, of course, $\Rdeti_{ts}^\Delta(\xi,\mu)=\Rdeti_{ts}^{\Delta,1}(\xi,\mu)+\Rti_{ts}^{\Delta,2}(\xi)+\Rdeti_{ts}^{\Delta,3}(\xi,\mu)$. Recall that we have to estimate $E[(\Xti_{ts}^{2,\Delta,(1,1)}(\xi))^2]$, and according to (\ref{decompo-xdeti-1-2}), we will treat the different terms $K_j$ separately.

\smallskip

\noindent
{\it Study of} $K_1:$ The expected value $E[K_1^2]$ can be bounded easily thanks to  (\ref{majo-double-increment}), applied for some fixed $\al \in (0,2H)$. 

\smallskip

\noindent
{\it Study of} $K_3:$ By Hölder's inequality, we have
\bean
\lefteqn{E\lc \lp (e^{-\xi(t-s)}-1) [(\Xep_s)^2-(\Xet_s)^2 ] \rp^2 \rc}\\
 &\leq & \xi^2 \lln t-s \rrn^2 E\lc (\Xep_s-\Xet_s)^2(\Xep_s+\Xet_s)^2 \rc\\
&\leq & \xi^2 \lln t-s \rrn^2 \sqrt{E[(\Xep_s-\Xet_s)^2]} \sqrt{E[(\Xep_s-\Xet_s)^2(\Xep_s+\Xet_s)^4)}\\
&\leq & c \, \xi^2 \lln t-s \rrn^2 \lln \ep-\eta \rrn^{\al/2},
\eean 
where we have used (\ref{majo-de-base}) to get the last inequality. 

\smallskip

\noindent
{\it Study of} $K_2:$ Likewise, using the decomposition $\Xep_s\Xep_t-\Xet_s\Xet_t=(\Xep_t-\Xet_t)\Xep_s+\Xet_t(\Xep_s-\Xet_s)$, we get, for any $\mu_1,\mu_2 \geq 0$,
\begin{eqnarray}
\lefteqn{E[ (1-e^{-\mu_1(t-s)})(1-e^{-\mu_2(t-s)}) \lcl \Xep_s\Xep_t-\Xet_s \Xet_t \rcl ]} \nonumber\\
&\leq & \mu_1\, \mu_2 \lln t-s \rrn^2 \lcl E[(\Xep_t-\Xet_t)\Xep_s ]+E[\Xet_t(\Xep_s-\Xet_s)] \rcl \label{decompo-trick}\\
&\leq & c \, \mu_1\, \mu_2 \lln t-s \rrn^2 \lcl \sqrt{E[(\Xep_t-\Xet_t)^2]}+\sqrt{E[(\Xep_s-\Xet_s)^2]} \rcl \nonumber\\
&\leq & c \, \mu_1 \, \mu_2 \lln t-s \rrn^2 \lln \ep-\eta \rrn^{\al/2} \nonumber.
\end{eqnarray}

\smallskip

\noindent
{\it Study of} $K_4:$ For our purposes, it remains in fact to estimate $E[\Rdeti_{ts}^\Delta(\xi,\mu_1) \Rdeti_{ts}^\Delta(\xi,\mu_2)]$ for any $\mu_1,\mu_2 \geq 0$. This expression can be decomposed into
\begin{multline*}
E[\Rdeti_{ts}^\Delta(\xi,\mu_1)\Rdeti_{ts}^\Delta(\xi,\mu_2)]\\
=\sum_{i,j \in \{1,3\}} E[\Rdeti_{ts}^{\Delta,i}(\xi,\mu_1)\Rdeti_{ts}^{\Delta,j}(\xi,\mu_2)]+\sum_{k\in\{1,3\}} \sum_{l\in\{1,2\}} E[ \Rti_{ts}^{\Delta,2}(\xi) \Rdeti_{ts}^{\Delta,k}(\xi,\mu_l)]+E[\Rti_{ts}^{\Delta,2}(\xi)^2],
\end{multline*}
whose terms will be treated again separately. For the last term, we have, just as above,
\bean
E[\Rti_{ts}^{\Delta,2}(\xi)^2]&=&E[(\Rti_{ts}^{\ep,2}(\xi)-\Rti_{ts}^{\eta,2}(\xi))^2]\\
& \leq & \frac{\xi^2}{4}\int_s^t du\int_s^t dv \, E\lc \lcl (\Xep_u)^2-(\Xet_u)^2 \rcl \lcl (\Xep_v)^2-(\Xet_v)^2 \rcl \rc \\
&\leq & c \, \xi^2 \lln t-s \rrn^2 \lln \ep-\eta \rrn^{\al/2}.
\eean
By the same arguments, we easily deduce 
$$E[ \Rti_{ts}^{\Delta,2}(\xi) \Rdeti_{ts}^{\Delta,3}(\xi,\mu_i)] \leq c \, \xi \lln \xi-\mu_i\rrn \lln t-s \rrn^2 \lln \ep-\eta \rrn^{\al/2} \quad \mbox{for} \ i \in {1,2}.$$
To deal with $E[\Rdeti_{ts}^{\Delta,3}(\xi,\mu_1)\Rdeti_{ts}^{\Delta,3}(\xi,\mu_2)]$, use the same trick as in (\ref{decompo-trick}) to deduce
\begin{multline*}
E[\Rdeti_{ts}^{\Delta,3}(\xi,\mu_1)\Rdeti_{ts}^{\Delta,3}(\xi,\mu_2)]\\
\leq \lln \xi-\mu_1 \rrn \lln \xi-\mu_2\rrn \int_s^t du \int_s^t dv \, E\lc \lcl \Xep_s\Xep_u-\Xet_s\Xet_u\rcl \lcl \Xep_s \Xep_v-\Xet_s\Xet_v \rcl \rc\\
 \leq  c \, \lln \xi-\mu_1\rrn \lln \xi-\mu_2\rrn \lln t-s \rrn^2 \lln \ep-\eta \rrn^{\al/2}.
\end{multline*}

\smallskip

In order to handle the terms involving $\Rdeti_{ts}^{1,\Delta}$, we resort again to the integration by parts method, which yields:
$$\Rdeti_{ts}^{1,\ep}(\xi,\mu)=c \lcl \mu \Adeti_{ts}^\ep(\xi,\mu)+\mu(\xi-\mu)\Bdeti_{ts}(\xi,\mu) \rcl,$$
where
$$\Adeti_{ts}^\ep(\xi,\mu)=-\Xep_t\int_s^t du \, e^{-\mu(t-u)}\Xep_u +\int_s^t du \, e^{-\xi(t-u)}(\Xep_u)^2,$$
$$\Bdeti_{ts}^\ep(\xi,\mu)=\int_s^t du \, e^{-\xi(t-u)} \Xep_u \int_s^u dv \, e^{-\mu(u-v)} \Xep_v.$$
Then it is readily checked, by some elementary computations, that:
\begin{eqnarray*}
E[ \Rdeti_{ts}^{\Delta,1}(\xi,\mu_i)\Rdeti_{ts}^{\Delta,1}(\xi,\mu_j)]&\leq& c \lln t-s \rrn^2 \lln \ep-\eta \rrn^{\al/2} \lcl 1+\xi^2 \rcl \prod_{k=1}^2 \lcl 1+\mu_l^2 \rcl  \\
E[\Rdeti_{ts}^{\Delta,1}(\xi,\mu_i)\Rti_{ts}^{\Delta,2}(\xi)]&\leq& c \lln t-s \rrn^2 \lln \ep-\eta \rrn^{\al/2} \lcl \mu_i \xi+\mu_i \xi^2+\mu_i^2 \xi\rcl  \\
E[ \Rdeti_{ts}^{\Delta,1}(\xi,\mu_i)\Rdeti_{ts}^{\Delta,3}(\xi,\mu_j)] &\leq& c \lln t-s \rrn^2 \lln \ep-\eta \rrn^{\al/2} \lcl 1+\xi^2 \rcl \prod_{l=1}^2 \lcl 1+\mu_l^2 \rcl.
\end{eqnarray*}
It is also easily seen, by means of the same considerations, that:
$$E[\Xdeti_{ts}^{2,\Delta,(1,1)}(\xi,\mu_1) \Xdeti_{ts}^{2,\Delta,(1,1)}(\xi,\mu_2)] \leq c \lln t-s \rrn^{4H-\al}\lln \ep-\eta \rrn^\ka \lcl 1+\xi^2 \rcl \prod_{i=1}^2 \lcl 1+\mu_l^2 \rcl,$$
for some $\ka >0$.

\smallskip

Finally, gathering all the estimates we have provided so far, and assuming the condition $\int_0^\infty d\mu \, |\hphi(\mu)| \lcl 1+\mu^2\rcl < \infty$ on $\phi$, we end up with the announced estimation, that is 
\beq\label{eq:82b}
E[(\Xti_{ts}^{2,\Delta,(1,1)}(\xi))^2 ] \leq c \lln t-s \rrn^{4H-\al} \ep^\ka \lcl 1+\xi^2\rcl,
\eeq
for a certain strictly positive $\ka$.

\medskip

\noindent
\textbf{The off-diagonal term $\mathbf{\Xti_{ts}^{\Delta,(1,2)}}$:} 
Recall that $\Xti_{ts}^{\ep,(1,2)}$ is defined by:
$$
\Xti_{ts}^{\ep,(1,2)}\int_0^\infty dy \, \hphi(y) \int_s^t dx \int_s^{x} dv \,
e^{-\xi(t-x)}e^{-y(x-v)} (\Xepun)'_{x}(\Xepde)'_{v}.
$$
Hence, the moment $E[(\Xti_{ts}^{\Delta,(1,2)})^2]$ can be written as
\begin{multline}\label{eq:83}
E\big[|\Xti^{\Delta,(1,2)}_{ts}(\xi)|^2\big]  \\
=\int_0^\infty dy_1 \, \hphi(y_1)\int_0^\infty dy_2 \, \hphi(y_2)
\int_s^t\int_s^t\int_s^{x_1}\int_s^{x_2}dx_1\,dx_2\,dx_3\,dx_4 \, \tilde M_{x_1x_2x_3x_4}^{\ep,\eta},  
\end{multline}
with
\begin{multline*}
\tilde M_{x_1x_2x_3x_4}^{\ep,\eta}=e^{-\xi(t-x_1)}e^{-y_1(x_1-x_3)}e^{-\xi(t-x_2)}e^{-y_2(x_2-x_4)}  
E\big[ \big\{(\Xepun)'_{x_1}(\Xepde)'_{x_3} - \\
(\Xetun)'_{x_1}(\Xetde)'_{x_3}\big\} 
\big\{(\Xepun)'_{x_2}(\Xepde)'_{x_4}-(\Xetun)'_{x_2}(\Xetde)'_{x_4}\big\}\big].
\end{multline*}
But according to our convention (\ref{eq:75}) for the fractional Brownian kernel, we can write
\begin{eqnarray}\label{expression-un}
\lefteqn{E\lc (\Xepun)'_{x_1}(\Xepde)'_{x_3}(\Xepun)'_{x_2}(\Xepde)'_{x_4}\rc}\\
&=& E\lc (\Xepun)'_{x_1}(\Xepun)'_{x_2}\rc E\lc(\Xepde)'_{x_3}(\Xepde)'_{x_4}\rc \nonumber \\
&=& c_H^2 \lcl K_{\ep,\ep}^{2H-2}(x_1,x_2)+ K_{\ep,\ep}^{2H-2}(x_2,x_1) \rcl \lcl K_{\ep,\ep}^{2H-2}(x_3,x_4)+ K_{\ep,\ep}^{2H-2}(x_4,x_3)\rcl,\nonumber
\end{eqnarray}
where we have used (\ref{eq:58}) in order to compute expressions like $E[(\Xepun)'_{x_1}(\Xepun)'_{x_2}]$. In the same way, one can check that
\begin{multline}\label{expression-deux}
E\lc (\Xepun)'_{x_1}(\Xepde)'_{x_3}(\Xetun)'_{x_2}(\Xetde)'_{x_4}\rc\\
=c_H^2 \lcl K_{\ep,\eta}^{2H-2}(x_1,x_2)+ K_{\ep,\eta}^{2H-2}(x_2,x_1) \rcl \lcl K_{\ep,\eta}^{2H-2}(x_3,x_4)+ K_{\ep,\eta}^{2H-2}(x_4,x_3)\rcl, 
\end{multline}
so that the difference between Expressions (\ref{expression-un}) and (\ref{expression-deux}), denoted by $\Delta_{x_1x_2x_3x_4}^{\ep,\eta}$, can be decomposed as $\Delta^{\ep,\eta}=c_H^2\{ \Ati^{\ep,\eta}+\Bti^{\ep,\eta}+\Cti^{\ep,\eta}+\Dti^{\ep,\eta}\}$, with
\begin{multline*}
\Ati^{\ep,\eta}_{x_1x_2x_3x_4} = \lcl K_{\ep,\ep}^{2H-2}(x_1,x_2)-K_{\ep,\eta}^{2H-2}(x_1,x_2)\rcl K_{\ep,\ep}^{2H-2}(x_3,x_4)\\
+K_{\ep,\eta}^{2H-2}(x_1,x_2)\lcl K_{\ep,\ep}^{2H-2}(x_3,x_4)-K_{\ep,\eta}^{2H-2}(x_3,x_4)\rcl= \Ati^{\ep,\eta,1}_{x_1x_2x_3x_4}+\Ati^{\ep,\eta,2}_{x_1x_2x_3x_4},
\end{multline*}
\begin{multline*}
\Bti^{\ep,\eta}_{x_1x_2x_3x_4} = \lcl K_{\ep,\ep}^{2H-2}(x_1,x_2)-K_{\ep,\eta}^{2H-2}(x_1,x_2)\rcl K_{\ep,\ep}^{2H-2}(x_4,x_3)\\
+K_{\ep,\eta}^{2H-2}(x_1,x_2)\lcl K_{\ep,\ep}^{2H-2}(x_4,x_3)-K_{\ep,\eta}^{2H-2}(x_4,x_3)\rcl= \Bti^{\ep,\eta,1}_{x_1x_2x_3x_4}+\Bti^{\ep,\eta,2}_{x_1x_2x_3x_4},
\end{multline*}
\begin{multline*}
\Cti^{\ep,\eta}_{x_1x_2x_3x_4} = \lcl K_{\ep,\ep}^{2H-2}(x_2,x_1)-K_{\ep,\eta}^{2H-2}(x_2,x_1)\rcl K_{\ep,\ep}^{2H-2}(x_3,x_4)\\
+K_{\ep,\eta}^{2H-2}(x_2,x_1)\lcl K_{\ep,\ep}^{2H-2}(x_3,x_4)-K_{\ep,\eta}^{2H-2}(x_3,x_4)\rcl,
\end{multline*}
\begin{multline*}
\Dti^{\ep,\eta}_{x_1x_2x_3x_4} = \lcl K_{\ep,\ep}^{2H-2}(x_2,x_1)-K_{\ep,\eta}^{2H-2}(x_2,x_1)\rcl K_{\ep,\ep}^{2H-2}(x_4,x_3)\\
+K_{\ep,\eta}^{2H-2}(x_2,x_1)\lcl K_{\ep,\ep}^{2H-2}(x_4,x_3)-K_{\ep,\eta}^{2H-2}(x_4,x_3)\rcl.
\end{multline*}
Thanks to Lemma \ref{lem:lemme-de-base}, the treatment of $\Ati^{\ep,\eta}$ in the expression (\ref{eq:83}) becomes easy. Indeed, we are allowed to deform the expression
$$A^{\ep,\eta,1}=\int_s^t dx_1 \int_s^t dx_2 \int_s^t dx_3 \int_s^t dx_4 \, e^{-\xi(t-x_1)}e^{-\mu_1(x_1-x_3)}e^{-\xi(t-x_2)}e^{-\mu_2(x_2-x_4)} \Ati^{\ep,\eta,1}_{x_1x_2x_3x_4}$$
into
\begin{multline}\label{deformation-contour}
A^{\ep,\eta,1}=\int_{\ga_{(s,t)}} dz_1 \int_{\overline{\ga_{(s,t)}}} dw_1 \, e^{-\xi(t-z_1)} e^{-\xi(t-w_1)} \lcl K_{\ep,\ep}^{2H-2}(z_1,w_1)-K_{\ep,\eta}^{2H-2}(z_1,w_1)\rcl \\
\int_{\ga_{(s,t)}(z_1)} dz_2 \int_{\overline{\ga_{(s,t)}}(w_1)} dw_2 \, e^{-\mu_1(z_1-z_2)}e^{-\mu_2(w_1-w_2)} K_{\ep,\ep}^{2H-2}(z_2,w_2),
\end{multline}
where $\ga_{(s,t)}(z)$ stands for the path $\ga_{(s,t)}$ stopped at $z$, defined similarly to \cite[Proof of Theorem 3.4]{Un}. Hence, invoking (\ref{estimation-noyau-simple}) and (\ref{estimation-noyau-double}), we get
\begin{multline*}
\lln A^{\ep,\eta,1}\rrn \leq \int_{|\ga_{(s,t)}|} \lln dz\rrn \int_{|\overline{\ga_{(s,t)}}|} \lln dw\rrn \lln K_{\ep,\ep}^{2H-2}(z,w)-K_{\ep,\eta}^{2H-2}(z,w)\rrn\\
\int_{|\ga_{(s,t)}|} \lln dz\rrn \int_{|\overline{\ga_{(s,t)}}|} \lln dw\rrn \lln K_{\ep,\ep}^{2H-2}(z,w) \rrn \leq c \lln t-s \rrn^{4H-\al} \lln \ep-\eta \rrn^\al.
\end{multline*}
The same arguments hold for $\Ati^{\ep,\eta,2}$, as well as for $\Dti^{\ep,\eta}$.

\smallskip

The estimation for $\Bti^{\ep,\eta}$ and $\Cti^{\ep,\eta}$ is less obvious, since we must cope with integrals of the form
\begin{multline*}
B^{\ep,\eta,1}=\int_s^t dx_1 \int_s^t dx_2  \, e^{-\xi(t-x_1)}e^{-\xi(t-x_2)} \lcl  K_{\ep,\ep}^{2H-2}(x_1,x_2)-K_{\ep,\eta}^{2H-2}(x_1,x_2)\rcl\\
\int_s^{x_1} dx_3 \int_s^{x_2}dx_4 \, e^{-\mu_1(x_1-x_3)}e^{-\mu_2(x_2-x_4)} K_{\ep,\ep}^{2H-2}(x_4,x_3),
\end{multline*}
for which the complex deformation (\ref{deformation-contour}) is not allowed ($K_{\ep,\ep}^{2H-2}(w_2,z_2)$ would be ill-defined for small $\ep$, since $-i(w_2-z_2)$ might be negative, see \cite{Un} for a further explanation). In fact, the result is a consequence of the technical lemma below. Indeed, with our notations, (\ref{estimation-noyau-inverse}) is equivalent to 
$$\lln B^{\ep,\eta,1}_{ts} \rrn \leq c \lln t-s \rrn^{4H-\al} \ep^\al (1+\xi^2)(1+\mu_1^2)(1+\mu_2^2),$$
which, as $\int_0^\infty d\mu \, |\hphi(\mu)| (1+\mu^2) < \infty$, gives an accurate bound for our purposes. To conclude with, it only remains to observe that the reasoning which leads to (\ref{estimation-noyau-inverse}) can be easily adapted to $\Bti^{\ep,\eta,2}$. The term $\Cti^{\ep,\eta}$ is then handled with an argument of symmetry.

\smallskip

Putting all our estimates together, we have thus proved that
\beq\label{eq:88}
E[(\Xti_{ts}^{2,\Delta,(1,2)}(\xi))^2 ] \leq c \lln t-s \rrn^{4H-\al} \ep^\al \lcl 1+\xi^2\rcl,
\eeq
for a certain $\al>0$. 

\smallskip

Owing to inequalities (\ref{eq:82b}) and (\ref{eq:88}), the proof of Lemma \ref{lem:lemme-cinq} is now easily finished. We are thus only left with the proof of the following lemma:

\begin{lemma}
Let $s<t$, $\xi,\mu_1,\mu_2 >0$, and set
\begin{multline*}
Q_{ts}^{\ep,\eta}(\xi,\mu_1,\mu_2)=\int_s^t dx_1 \int_s^t dx_2 \, e^{-\xi(t-x_1)}e^{-\xi(t-x_2)} \lcl K_{\ep,\ep}^{2H-2}(x_1,x_2)-K_{\ep,\eta}^{2H-2}(x_1,x_2)\rcl\\
\int_s^{x_1}dx_3\int_s^{x_2} dx_4 \, e^{-\mu_1(x_1-x_3)}e^{-\mu_2(x_2-x_4)} K_{\ep,\ep}^{2H-2}(x_4,x_3).
\end{multline*}
Then, for any $\al < 4H-1$, we have
\begin{equation}\label{estimation-noyau-inverse}
\big| Q_{ts}^{\ep,\eta}(\xi,\mu_1,\mu_2) \big|
\leq c \lln t-s \rrn^{4H-\al}\ep^\al (1+\xi^2) \prod_{l=1}^2(1+\mu_l^2),
\end{equation}
where the constant $c$ does not depend on $s,t,\ep,\eta,\xi,\mu_1,\mu_2$.
\end{lemma}

\begin{proof}
First, notice that the estimation is obvious if $t-s \leq 2\ep$, since then $|K_{\ep,\ep}^{2H-2}(x_4,x_3)|$  $\leq \lln t-s \rrn^{2H-2}$ and 
$$|K_{\ep,\ep}^{2H-2}(x_1,x_2)-K_{\ep,\eta}^{2H-2}(x_1,x_2)| \leq c \, \ep^{2H-2-\al} \lln \ep-\eta \rrn^\al \leq c \lln t-s \rrn^{2H-2-\al} \ep^\al.$$
From now on, we thus assume that $2\ep < t-s$. The strategy in order to control our multiple integral $Q_{ts}^{\ep,\eta}(\xi,\mu_1,\mu_2)$ consists then in two main steps: (i) Handle the exponential weights by means of successive integrations by parts (recall that $X^{\ep}$ is a smooth process for a given $\ep>0$). (ii) Control the singularities of the fBm kernel by a convenient deformation of contour.

\smallskip

A first application of the integration by parts trick gives the following identity:
\begin{eqnarray}
\lefteqn{\int_s^{x_1} dx_3\int_s^{x_2}dx_4 \, e^{-\mu_1(x_1-x_3)}e^{-\mu_2(x_2-x_4)}K_{\ep,\ep}^{2H-2}(x_4,x_3) } \nonumber\\
&=& c_1\big\{ K_{\ep,\ep}^{2H}(x_2,x_1)-e^{-\mu_1(x_1-s)}K_{\ep,\ep}^{2H}(x_2,s)-e^{-\mu_2(x_2-s)}K_{\ep,\ep}^{2H}(s,x_1) \nonumber\\
& & \hspace{5cm}+e^{-\mu_1(x_1-s)}e^{-\mu_2(x_2-s)}(2\ep)^{2H} \big\}+c_2 \Rti_{x_1x_2}(\mu_1,\mu_2)\nonumber\\
&=& c_1 \big\{ A_{x_1x_2}^{2\ep}+\Bti_{x_1x_2}^{2\ep}(\mu_1)+\Cti_{x_1x_2}^{2\ep}(\mu_2)+\Dti_{x_1x_2}^{2\ep}(\mu_1,\mu_2) \big\}+c_2 \Rti_{x_1x_2}(\mu_1,\mu_2), \label{second-noyau}
\end{eqnarray}
where $\Rti$ can be further decomposed into $\Rti_{x_1x_2}(\mu_1,\mu_2)=\Rti^1_{x_1x_2}(\mu_1,\mu_2)+\Rti^2_{x_1x_2}(\mu_1)+\Rti^3_{x_1x_2}(\mu_1,\mu_2)$, with:
\bean
\Rti^1_{x_1x_2}(\mu_1,\mu_2)&=&c \, \mu_2 \int_s^{x_1} dx_3\int_s^{x_2}dx_4 \, e^{-\mu_1(x_1-x_3)}e^{-\mu_2(x_2-x_4)}K_{\ep,\ep}^{2H-1}(x_4,x_3)  \\
\Rti^2_{x_1x_2}(\mu_1)&=&c \, \mu_1 \int_s^{x_1} dx_3 \, e^{-\mu_1(_1-x_3)}K_{\ep,\ep}^{2H}(x_2,x_3)  \\
\Rti^3_{x_1x_2}(\mu_1,\mu_2)&=&c \, \mu_1 \int_s^{x_1}dx_3 \, e^{-\mu_1(x_1-x_3)}e^{-\mu_2(x_2-s)}K_{\ep,\ep}^{2H}(s,x_3).
\eean
We have thus proved that
$$
Q_{ts}^{\ep,\eta}(\xi,\mu_1,\mu_2)=\tilde I_{ts}(\xi) + \doubletilde{II}_{ts}(\xi,\mu_1,\mu_2)
+\doubletilde{III}_{ts}(\xi,\mu_1,\mu_2)+\doubletilde{IV}_{ts}(\xi,\mu_1,\mu_2)
+c_2 \Rti_{x_1x_2}(\mu_1,\mu_2),
$$
with
\bean
\tilde I_{ts}(\xi)&=&\int_s^t dx_1 \int_s^t dx_2 \, m_{x_1,x_2}^{\ep,\eta}(\xi)  K_{\ep,\ep}^{2H}(x_2,x_1)  \\
\doubletilde{II}_{ts}(\xi,\mu_1,\mu_2)&=&
\int_s^t dx_1\int_s^t dx_2 \,  m_{x_1,x_2}^{\ep,\eta}(\xi)\Bti_{x_1x_2}^{2\ep}(\mu_1,\mu_2)\\
\doubletilde{III}_{ts}(\xi,\mu_1,\mu_2)&=&
\int_s^t dx_1\int_s^t dx_2 \,  m_{x_1,x_2}^{\ep,\eta}(\xi)\Cti_{x_1x_2}^{2\ep}(\mu_1,\mu_2)\\
\doubletilde{IV}_{ts}(\xi,\mu_1,\mu_2)&=&
\int_s^t dx_1\int_s^t dx_2 \,  m_{x_1,x_2}^{\ep,\eta}(\xi)\Dti_{x_1x_2}^{2\ep}(\mu_1,\mu_2),
\eean
where we have set $m_{x_1,x_2}^{\ep,\eta}(\xi)= e^{-\xi(t-x_1)}e^{-\xi(t-x_2)} 
\{ K_{\ep,\ep}^{2H-2}(x_1,x_2)-K_{\ep,\eta}^{2H-2}(x_1,x_2) \}$.
We will now estimate these 4 terms separately.

\smallskip

To begin with, let us consider the case of $\Iti_{ts}(\xi)$: an elementary change of variables $(u=x_1-x_2,v=x_2)$ yields:
$$\Iti_{ts}(\xi)=\int_{-(t-s)}^{t-s} du \, f_{ts}^\xi(u) \lcl (-iu+2\ep)^{2H-2}-(-iu+\ep+\eta)^{2H-2} \rcl (iu+2\ep)^{2H},$$
where
$$f_{ts}^\xi(u)=\textbf{1}_{\{u\leq 0\} } \int_{s-u}^t dv \, e^{-\xi(t-u-v)}e^{-\xi(t-v)}+\textbf{1}_{\{ u\geq 0\} } \int_s^{t-u}dv \, e^{-\xi(t-u-v)}e^{-\xi(t-v)}.$$
It is easily seen that for $u\in (0,t-s)$,
$$f_{ts}^\xi(u)=\frac{1}{2\xi} \lcl e^{-\xi u}-e^{-\xi(t-s-u)}e^{-\xi(t-s)} \rcl =f_{ts}^\xi(-u),$$
so that
$$\Iti_{ts}(\xi) =2 \, \mbox{Re} \lp \int_0^{t-s}du \, f_{ts}^\xi(u) \lcl (-iu+2\ep)^{2H-2}-(-iu+\ep+\eta)^{2H-2} \rcl (iu+2\ep)^{2H} \rp.$$
Write the latter integral as
\begin{equation}\label{integrale-gamma-epsilon}
i^{2H} \int_{[-2i\ep,t-s-2i\ep]} dz \, f_{ts}^\xi(z+2i\ep) \lcl (-iz+4\ep)^{2H-2}-(-iz+3\ep+\eta)^{2H-2} \rcl z^{2H},
\end{equation}
and deform the line $[-2i\ep,t-s-2i\ep]$ into a four-part contour $\ga_{\ep,(t-s)}=\ga^1_{\ep,(t-s)} \cup \ga^2_{\ep,(t-s)} \cup \ga^3_{\ep,(t-s)} \cup \ga^4_{\ep,(t-s)}$, where $\ga^1_{\ep,(t-s)}$ runs along the half-circle centered at the origin from $-2i\ep$ to $2i\ep$ in $\{ z : \, \mbox{Re} \, z \geq 0\}$, $\ga^2_{\ep,(t-s)}$ is the line $[2i\ep,i(t-s)]$, $\ga^3_{\ep,(t-s)}$ the line $[i(t-s),t-s+i(t-s)]$ and $\ga^4_{\ep,(t-s)}$ the line $[t-s+i(t-s),t-s-2i\ep]$.

\begin{figure}[!h]
\centering
\begin{pspicture}(0,0)(10,10)
%\psgrid
\psline{->}(1,1)(1,9) 
\psline{->}(0.2,4)(9,4) 
\psarc[linestyle=dashed]{->}(1,4){1}{270}{1}  
\psarc[linestyle=dashed](1,4){1}{0}{90} 
\psline[linestyle=dashed,ArrowInside=->,ArrowInsidePos=0.5](1,5)(1,8)(5,8)(5,3)
\psline[ArrowInside=->,ArrowInsidePos=0.5](1,3)(5,3)
\rput(0.4,3){$-2i\ep$}
\rput(0.2,8){$i(t-s)$}
\rput(0.6,5){$2i\ep$}
\rput(5.6,3.7){$t-s$}
\rput(2.4,4){$\ga^1$}
\rput(1.4,6.4){$\ga^2$}
\rput(3,7.6){$\ga^3$}
\rput(4.6,5.6){$\ga^4$}
\end{pspicture}
\caption{Deformation of $[-2i\ep,t-s-2i\ep]$}
\end{figure}
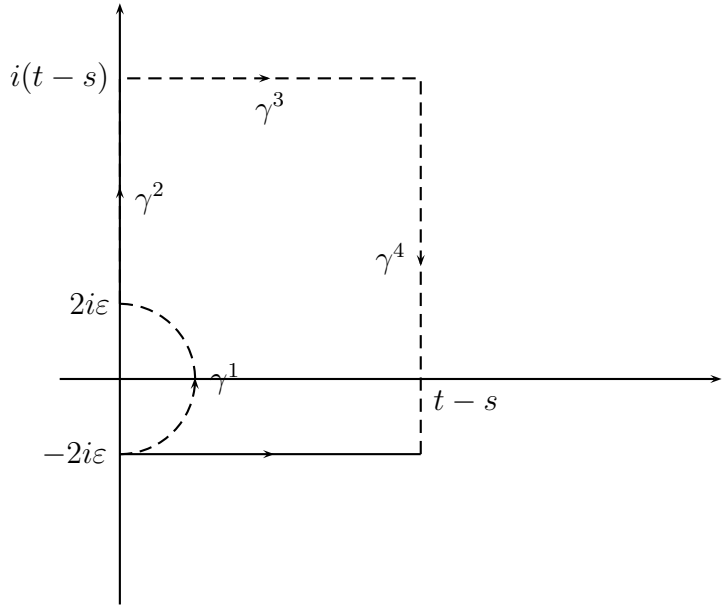

Using the decomposition 
$$e^{-\xi u}-e^{-\xi(t-s-u)}e^{-\xi(t-s)}=e^{-\xi(t-s-u)}(1-e^{-\xi(t-s)})-(e^{-\xi(t-s-u)}-1)+(e^{-\xi u}-1),$$                 
it is readily checked that $\sup_{z \in \ga_{\ep,(t-s)}} |f_{ts}^\xi(z+2i\ep)| \leq c \, \lln t-s \rrn$. The estimation of (\ref{integrale-gamma-epsilon}) on each of the $\ga_{\ep,(t-s)}^i$'s is then a matter of elementary calculations, that we proceed to detail now: for $\ga_{\ep,(t-s)}^1$, one has, if $y \in [\eta,\ep]$ and $\theta \in [-\pi/2,\pi/2]$, that $|-2i\ep\, e^{i\theta}+3\ep+y | \geq \ep$. Thanks to the fact that $\ep<(t-s)/2$, this leads to
\begin{multline*}
\Big| \int_{\ga_{\ep,(t-s)}^1} dz \, f_{ts}^\xi(z+2i\ep) \lcl (-iz+4\ep)^{2H-2}-(-iz+3\ep+\eta)^{2H-2} \rcl z^{2H} \Big|\\
\leq c  \lln t-s \rrn \ep^{4H-1} \leq c  \lln t-s \rrn^{4H-\al}\ep^\al.
\end{multline*}
For the upper bound on the path $\ga_{\ep,(t-s)}^2$, notice that if $x\in [2\ep,t-s]$, $y\in [\eta,\ep]$, then of course $|x+3\ep+y| \geq x$. Thus, for a small enough positive parameter $\al$, we obtain:
\bean
\lefteqn{\Big| \int_{\ga_{\ep,(t-s)}^2} dz \, f_{ts}^\xi(z+2i\ep) \lcl (-iz+4\ep)^{2H-2}-(-iz+3\ep+\eta)^{2H-2} \rcl z^{2H} \Big|}\\
&\leq & c  \lln t-s \rrn \int_{2\ep}^{t-s}dx \, x^{2H} (x^{2H-2-\al} \lln \ep-\eta \rrn^\al) \\
& \leq & c  \lln t-s \rrn^{4H-\al} \ep^\al \int_0^1 dx \, x^{4H-2-\al} \ \leq \ c  \lln t-s \rrn^{4H-\al}\ep^\al,
\eean
since, by hypothesis, $2+\al-4H <1$. For $\ga_{\ep,(t-s)}^3$, start with $\lln t-s-iu+3\ep+y\rrn \geq \lln t-s\rrn$ if $u\in [0,t-s]$, $y\in [\eta,\ep]$, to deduce
\bean
\lefteqn{\Big| \int_{\ga_{\ep,(t-s)}^3} dz \, f_{ts}^\xi(z+2i\ep) \lcl (-iz+4\ep)^{2H-2}-(-iz+2\ep+\eta)^{2H-2} \rcl z^{2H} \Big|}\\
&\leq & c  \lln t-s \rrn \int_0^{t-s} du \, \lln t-s \rrn^{2H-2-\al}\lln \ep-\eta \rrn^\al \lln i(t-s)+u\rrn^{2H} \ \leq \ c \lln t-s\rrn^{4H-\al}\ep^\al.
\eean
Finally, as far as $\ga_{\ep,(t-s)}^4$ is concerned, observe that for any $v \in [0,t-s+2\ep]$, $y\in [\eta,\ep]$, $\lln -i(t-s)+(t-s)-v+3\ep+y\rrn \geq \lln t-s \rrn$, and thus
\bean
\lefteqn{
\Big| \int_{\ga_{\ep,(t-s)}^4} dz \, f_{ts}^\xi(z+2i\ep) \lcl (-iz+4\ep)^{2H-2}-(-iz+2\ep+\eta)^{2H-2} \rcl z^{2H} \Big|}\\
&\leq & c \lln t-s \rrn \int_0^{t-s+2\ep}dv \, \lln t-s \rrn^{2H-2-\al}\lln \ep-\eta \rrn^\al \lln t-s+i(t-s)-iv\rrn^{2H} \\
&\leq& \ c  \lln t-s\rrn^{4H-\al}\ep^\al.
\eean
Therefore, these four elementary bounds, computed on the paths $\ga_{\ep,(t-s)}^1$ up to $\ga_{\ep,(t-s)}^4$, allow to claim that $| \Iti_{ts}(\xi)| \leq c \lln t-s \rrn^{4H-\al}\ep^\al$.

\smallskip

Consider now the term $\doubletilde{IV}_{ts}(\xi,\mu_1,\mu_2)$: it is readily checked that
\begin{multline*}
\doubletilde{IV}_{ts}(\xi,\mu_1,\mu_2)=(2\ep)^{2H}\int_s^tdx_1\int_s^t dx_2 \, e^{-\xi(t-x_1)}e^{-\xi(t-x_2)}e^{-\mu_1(x_1-s)}e^{-\mu_2(x_2-s)}\\
\lcl (-i(x_1-x_2)+2\ep)^{2H-2}-(-i(x_1-x_2)+\ep+\eta)^{2H-2} \rcl,
\end{multline*}
and perform the same deformation as in Lemma \ref{lem:lemme-de-base} to deduce
$$| \doubletilde{IV}_{ts}(\xi,\mu_1,\mu_2)| \leq c \, (2\ep)^{2H} \lln t-s \rrn^{2H-\al} \lln \ep-\eta \rrn^\al \leq c \lln t-s \rrn^{4H-\al} \ep^\al.$$

\smallskip

In order to deal with $\doubletilde{II}_{ts}(\xi,\mu_1,\mu_2)$, notice that
\begin{multline*}
\doubletilde{II}_{ts}(\xi,\mu_1,\mu_2)
=\int_s^t dx_1\int_s^t dx_2 \, e^{-\xi(t-x_1)}e^{-\xi(t-x_2)}e^{-\mu_1(x_1-s)}  \\
K_{\ep,\ep}^{2H}(x_2,s) \lcl K_{\ep,\ep}^{2H-2}(x_1,x_2)-K_{\ep,\eta}^{2H-2}(x_1,x_2) \rcl,
\end{multline*}
and write, by means of another integration by parts,
\begin{multline*}
\int_s^t dx_1 \, e^{-\xi(t-x_1)}e^{-\mu_1(x_1-s)} \lcl K_{\ep,\ep}^{2H-2}(x_1,x_2)-K_{\ep,\eta}^{2H-2}(x_1,x_2) \rcl\\
= c \Big[ e^{-\mu_1(t-s)}\lcl K_{\ep,\ep}^{2H-1}(t,x_2)-K_{\ep,\eta}^{2H-1}(t,x_2) \rcl
-e^{-\xi(t-s)} \lcl K_{\ep,\ep}^{2H-1}(s,x_2)-K_{\ep,\eta}^{2H-1}(s,x_2) \rcl \Big]\\+\Rdeti_{x_2}^4(\xi,\mu_1),
\end{multline*}
with
$$\Rdeti_{x_2}^4(\xi,\mu_1)=c \, (\xi-\mu_1) \int_s^t dx_1 \, e^{-\xi(t-x_1)}e^{-\mu_1(x_1-s)}
\lcl K_{\ep,\ep}^{2H-1}(x_1,x_2)-K_{\ep,\eta}^{2H-1}(x_1,x_2)\rcl.$$
Now, use the same strategy as for $\doubletilde{I}_{ts}(\xi,\mu_1,\mu_2)$, which consists here in writing

\begin{multline*}
\int_s^t dx_2 \, e^{-\mu_1(t-s)}e^{-\xi(t-x_2)} K_{\ep,\ep}^{2H}(x_2,s) \lcl K_{\ep,\ep}^{2H-1}(t,x_2)-K_{\ep,\eta}^{2H-1}(t,x_2) \rcl\\
=e^{-\mu_1(t-s)}\int_{[2i\ep,t-s+2i\ep]} dz \, e^{-\xi(t-s-z)}(-i)^{2H}z^{2H}\\
\lcl (-i(t-s-z)+4\ep)^{2H-1}-(-i(t-s-z)+3\ep+\eta)^{2H-1} \rcl
\end{multline*}
and deforming the line $[2i\ep,t-s+2i\ep]$ into 
$$\ga_{\ep,(t-s)}=\ga_{\ep,(t-s)}^1 \cup [-2i\ep,-i(t-s)] \cup [-i(t-s),t-s-i(t-s)] \cup [t-s-i(t-s),t-s+2i\ep],$$
with $\ga_{\ep,(t-s)}^1$ the half-circle centered at the origin from $2i\ep$ to $-2i\ep$ in $\{z: \, \mbox{Re}\, z \geq 0\}$. The same kind of elementary estimations as for $\doubletilde{I}_{ts}(\xi,\mu_1,\mu_2)$ then lead to
\begin{multline*}
\Big| \int_s^t dx_2 \, e^{-\mu_1(t-s)}e^{-\xi(t-x_2)} K_{\ep,\ep}^{2H}(x_2,s) 
\lcl  K_{\ep,\ep}^{2H-1}(t,x_2)- K_{\ep,\eta}^{2H-1}(t,x_2) \rcl \Big|  \\
\leq c \lln t-s \rrn^{4H-\al} \ep^\al.
\end{multline*}
Likewise,
\begin{multline*}
\Big| \int_s^t dx_2 \, e^{-\xi(t-s)}e^{-\xi(t-x_2)}  K_{\ep,\ep}^{2H}(x_2,s) 
\lcl  K_{\ep,\ep}^{2H-1}(s,x_2)- K_{\ep,\eta}^{2H-1}(s,x_2) \rcl \Big|  \\
\leq c \lln t-s \rrn^{4H-\al} \ep^\al,
\end{multline*}
and accordingly
$$\Big| \doubletilde{II}_{ts}(\xi,\mu_1,\mu_2)-\int_s^t dx_2 \, e^{-\xi(t-x_2)} K_{\ep,\ep}^{2H}(x_2,s) \Rdeti_{x_2}^4(\xi,\mu_1) \Big|
\leq c \lln t-s\rrn^{4H-\al} \ep^\al.$$
Those arguments can be easily adapted to $\doubletilde{III}_{ts}(\xi,\mu_1,\mu_2)$ to get
$$\Big| \doubletilde{III}_{ts}(\xi,\mu_1,\mu_2)-\int_s^t dx_1 \, e^{-\xi(t-x_1)} K_{\ep,\ep}^{2H}(s,x_1)\Rdeti_{x_1}^5(\xi,\mu_2) \Big| \leq c \lln t-s \rrn^{4H-\al}\ep^\al,$$
where
\begin{multline*}
\Rdeti_{x_1}^5(\xi,\mu_2)=c (\xi-\mu_2) \int_s^t dx_2 \, e^{-\xi(t-x_2)}e^{-\mu_2(x_2-s)}
\lcl  K_{\ep,\ep}^{2H-1}(x_1,x_2)- K_{\ep,\eta}^{2H-1}(x_1,x_2)\rcl.
\end{multline*}

We finally have to cope with the remainders $\Rti, \Rdeti^4,\Rdeti^5$. Owing to the higher regularity of those terms (as regards the kernels), it is rather clear that simple integration by parts should be sufficient to reach the expected bound.

\smallskip

Consider for instance the case of $\Rdeti^2(\xi,\mu_1)$ defined by:
\begin{eqnarray*}
\Rdeti^2(\xi,\mu_1)&\triangleq&\int_s^t dx\int_s^t dy \, e^{-\xi(t-x)}e^{-\xi(t-y)} \lcl K_{\ep,\ep}^{2H-2}(x,y)-K_{\ep,\eta}^{2H-2}(x,y) \rcl \Rti_{xy}^2(\mu_1)\\
&=&c \, \mu_1 \int_s^t dx\int_s^t dy \, e^{-\xi(t-x)}e^{-\xi(t-y)} \lcl K_{\ep,\ep}^{2H-2}(x,y)-K_{\ep,\eta}^{2H-2}(x,y) \rcl \phi^{\mu_1}(x,y),
\end{eqnarray*}
with $\phi^{\mu_1}(x,y)=\int_s^x du \, e^{-\mu_1(x-u)}K_{\ep,\ep}^{2H}(y,u)$. Another integration by parts yields:
\begin{align*}
&\int_s^t dy \, e^{-\xi(t-y)} \lcl K_{\ep,\ep}^{2H-2}(x,y)-K_{\ep,\eta}^{2H-2}(x,y) \rcl \phi^{\mu_1}(x,y)\\
=&c \Big[ \lcl K_{\ep,\ep}^{2H-1}(x,t)-K_{\ep,\eta}^{2H-1}(x,t) \rcl \phi^{\mu_1}(x,t)\\
&\hspace{5cm}-e^{-\xi(t-s)} \lcl K_{\ep,\ep}^{2H-1}(x,s)-K_{\ep,\eta}^{2H-1}(x,s)\rcl \phi^{\mu_1}(x,s) \Big]\\
&-c \int_s^t dy \lcl K_{\ep,\ep}^{2H-1}(x,y)-K_{\ep,\eta}^{2H-1}(x,y) \rcl
e^{-\xi(t-y)} \lcl \xi \, \phi^{\mu_1}(x,y)+ \frac{\partial \phi^{\mu_1}}{\partial y}(x,y) \rcl,
\end{align*}
and thus, plugging this expression into the definition of $\Rdeti^2$ and integrating by parts again, we obtain
\begin{align}\label{rdeti-2}
&\Rdeti^2(\xi,\mu_1)  \\
&=c \, \mu_1
\Big\{ \big[ \lcl (2\ep)^{2H}-(\ep+\eta)^{2H}\rcl \phi^{\mu_1}(t,t)-e^{-\xi(t-s)}\lcl K_{\ep,\ep}^{2H}(s,t)-K_{\ep,\eta}^{2H}(s,t) \rcl \phi^{\mu_1}(s,t) \big] \nonumber\\
&-e^{-\xi(t-s)}\big[ \lcl K_{\ep,\ep}^{2H}(t,s)-K_{\ep,\eta}^{2H}(t,s) \rcl \phi^{\mu_1}(t,s)  
-e^{-\xi(t-s)}\lcl (2\ep)^{2H}-(\ep+\eta)^{2H}\rcl \phi^{\mu_1}(s,s)\big] \Big\} \nonumber\\
&+\Rdeti^{2,1}(\xi,\mu_1)  +\Rdeti^{2,2}(\xi,\mu_1), \nonumber
\end{align}
with
\begin{eqnarray*}
\Rdeti^{2,1}(\xi,\mu_1)&=&c \, \mu_1 \, \xi \int_s^t dx\int_s^t dy \, q_{xy}^{\ep,\eta}(\xi)
\, \phi^{\mu_1}(x,y)  \\
 \Rdeti^{2,2}(\xi,\mu_1)&=&c \, \mu_1  \int_s^t dx\int_s^t dy \, q_{xy}^{\ep,\eta}(\xi) \,
 \frac{\partial \phi^{\mu_1}}{\partial y}(x,y),
\end{eqnarray*}
where we have set $q_{xy}^{\ep,\eta}(\xi)=e^{-\xi(t-x)}e^{-\xi(t-y)} \{ K_{\ep,\ep}^{2H-1}(x,y)-K_{\ep,\eta}^{2H-1}(x,y)\}$. Let us now integrate those last two expressions by parts with respect to $x$: we obtain
\begin{align*}
&\Rdeti^{2,1}(\xi,\mu_1)  \\
=&c_1 \, \mu_1 \, \xi \int_s^t dy \, e^{-\xi(t-y)}
\Big[ \lcl K_{\ep,\ep}^{2H}(t,y)-K_{\ep,\eta}^{2H}(t,y)\rcl \phi^{\mu_1}(t,y)  \\
&\hspace{7cm}
-e^{-\xi(t-s)} \lcl K_{\ep,\ep}^{2H}(s,y)-K_{\ep,\eta}^{2H}(s,y)\rcl \phi^{\mu_1}(s,y) \Big]\\
&+c_2 \, \mu_1 \, \xi^2 \int_s^t dx \int_s^t dy \, e^{-\xi(t-x)}e^{-\xi(t-y)} \lcl K_{\ep,\ep}^{2H}(x,y)-K_{\ep,\eta}^{2H}(x,y)\rcl \phi^{\mu_1}(x,y)\\
&+c_3 \, \mu_1 \, \xi \int_s^t dx \int_s^t dy \, e^{-\xi(t-x)}e^{-\xi(t-y)} \lcl  K_{\ep,\ep}^{2H}(x,y)-K_{\ep,\eta}^{2H}(x,y)\rcl \frac{\partial \phi^{\mu_1}}{\partial x}(x,y),
\end{align*}
and the expression for $\Rdeti^{2,2}$ is
\begin{align*}
&\Rdeti^{2,2}(\xi,\mu_1)\\
=&c_1 \, \mu_1  \int_s^t dy \, e^{-\xi(t-y)}
\Big[ \lcl K_{\ep,\ep}^{2H}(t,y)-K_{\ep,\eta}^{2H}(t,y)\rcl \frac{\partial\phi^{\mu_1}}{\partial y}(t,y)\\
&\hspace{7cm}-e^{-\xi(t-s)} \lcl K_{\ep,\ep}^{2H}(s,y)-K_{\ep,\eta}^{2H}(s,y)\rcl \frac{\partial\phi^{\mu_1}}{\partial y}(s,y) \Big]\\
&+c_2 \, \mu_1 \, \xi \int_s^t dx \int_s^t dy \, e^{-\xi(t-x)}e^{-\xi(t-y)} \lcl K_{\ep,\ep}^{2H}(x,y)-K_{\ep,\eta}^{2H}(x,y)\rcl \frac{\partial\phi^{\mu_1}}{\partial y}(x,y)\\
&+c_3 \, \mu_1  \int_s^t dx \int_s^t dy \, e^{-\xi(t-x)}e^{-\xi(t-y)} \lcl K_{\ep,\ep}^{2H}(x,y)-K_{\ep,\eta}^{2H}(x,y)\rcl \frac{\partial^2 \phi^{\mu_1}}{\partial y\partial x}(x,y).
\end{align*}

The following (easy) estimations come then into play: whenever $x,y\in[s,t]$, we have
$$|\phi^{\mu_1}(x,y)| \leq c \lln t-s\rrn^{2H+1} \ , \ \lln \frac{\partial \phi^{\mu_1}}{\partial x}\rrn \leq c \lln t-s\rrn^{2H}\lcl 1+\mu_1\rcl \ , \ \lln \frac{\partial \phi^{\mu_1}}{\partial y}\rrn \leq c \lln t-s\rrn^{2H}\lcl 1+\mu_1\rcl,$$
and if $a\in \R$, $\lln K_{\ep,\ep}^{2H}(a,0)-K_{\ep,\eta}^{2H}(a,0)\rrn \leq c \, \ep^{2H-1}\lln \ep-\eta \rrn \leq c \, \ep^{2H}$. Besides,
$$\frac{\partial^2 \phi^{\mu_1}}{\partial y\partial x}(x,y)=c \lcl K_{\ep,\ep}^{2H-1}(y,x)-\mu_1\int_s^x du \, e^{-\mu_1(x-u)}K_{\ep,\ep}^{2H-1}(y,u) \rcl.$$
Going back to (\ref{rdeti-2}), the previous estimations finally give rise to
$$|\Rdeti^2(\xi,\mu_1)| \leq c \, \ep^{2H} \lln t-s \rrn^{2H+1} \lcl \mu_1+\mu_1 \xi+\mu_1\xi^2+\mu_1^2\xi+\mu_1^2 \rcl,$$
which leads to the expected bound since $2H+1 >4H-\al$ and $2H > \al$.

\smallskip

The same arguments enable to handle $\Rti^1,\Rti^3,\Rdeti^4,\Rdeti^5$, which achieves the proof.

\end{proof}

\medskip

\noindent
{\bf Acknowledgment:}
We wish to thank M. Gubinelli for introducing us to Stroock's method for the proof of GRR type results. We also had a great benefit of some conversations with J. Unteberger concerning the complex analysis method for estimating iterated integrals of the fBm.

\end{document}